\documentclass[a4paper,11pt]{article}
\usepackage{amsfonts,amssymb,amsbsy,amsmath,epsf,eepic,graphics,graphicx}
\usepackage{color}
\usepackage{ mathrsfs }
\usepackage{bbm}
\usepackage{enumerate}
\usepackage[english]{babel}
\usepackage[hidelinks]{hyperref}
\makeatletter
\@addtoreset{enunciato}{section}
\makeatother

\newcounter{enunciato}[section]

\makeatletter
\@addtoreset{enunciat}{subsection}
\makeatother

\newcounter{enunciat}[subsection]

\newtheorem{ittheorem}{Theorem}
\newtheorem{ittheoremm}{Theorem}
\newtheorem{itlemma}{Lemma}
\newtheorem{itlemmaa}{Lemma}
\newtheorem{itproposition}{Proposition}
\newtheorem{itpropositiona}{Proposition}
\newtheorem{itdefinition}{Definition}
\newtheorem{itconjecture}{Conjecture}
\newtheorem{itassumption}{Assumption}
\newtheorem{itcorollary}{Corollary}
\newtheorem{itcorollarya}{Corollary}
\newtheorem{itremark}{Remark}

\newenvironment{remark}{\addtocounter{enunciato}{1}   	\begin{itremark}}{\end{itremark}}

\newenvironment{theorem}{\addtocounter{enunciato}{1}
	\begin{ittheorem}}{\end{ittheorem}}

\newenvironment{theoremm}{\addtocounter{enunciato}{1}
	\begin{ittheoremm}}{\end{ittheoremm}}

\newenvironment{lemma}{\addtocounter{enunciato}{1}
	\begin{itlemma}}{\end{itlemma}}

\newenvironment{lemmaa}{\addtocounter{enunciat}{1}
\begin{itlemmaa}}{\end{itlemmaa}}

\newenvironment{proposition}{\addtocounter{enunciato}{1}
	\begin{itproposition}}{\end{itproposition}}

\newenvironment{definition}{\addtocounter{enunciato}{1}
	\begin{itdefinition}}{\end{itdefinition}}

\newenvironment{conjecture}{\addtocounter{enunciato}{1}
	\begin{itconjecture}}{\end{itconjecture}}

\newenvironment{assumption}{\addtocounter{enunciato}{1}
	\begin{itassumption}}{\end{itassumption}}

\newenvironment{corollary}{\addtocounter{enunciato}{1}
	\begin{itcorollary}}{\end{itcorollary}}
\newenvironment{corollarya}{\addtocounter{enunciat}{1}
\begin{itcorollarya}}{\end{itcorollarya}}

\newcommand{\beq}[1]{\begin{equation}\label{#1}}
\newcommand{\eeq}{\end{equation}}

\newcommand{\ba}[1]{\begin{align}\label{#1}}
\newcommand{\ea}{\end{align}}

\newcommand{\bea}[1]{\begin{eqnarray}\label{#1}}
\newcommand{\eea}{\end{eqnarray}}

\newcommand{\bl}[1]{\begin{lemma}\label{#1}}
	\newcommand{\el}{\end{lemma}}

\newcommand{\bt}[1]{\begin{theorem}\label{#1}}
	\newcommand{\et}{\end{theorem}}

\newcommand{\bd}[1]{\begin{definition}\label{#1}}
	\newcommand{\ed}{\end{definition}}

\newcommand{\bp}[1]{\begin{proposition}\label{#1}}
	\newcommand{\ep}{\end{proposition}}

\newcommand{\bc}[1]{\begin{corollary}\label{#1}}
	\newcommand{\ec}{\end{corollary}}

\newcommand{\bcj}[1]{\begin{conjecture}\label{#1}}
	\newcommand{\ecj}{\end{conjecture}}

\newcommand{\bass}[1]{\begin{assumption}\label{#1}}
	\newcommand{\eass}{\end{assumption}}

\newcommand{\bpr}{\begin{proof}}
	\newcommand{\epr}{\end{proof}}

\newcommand{\bprl}[1]{\begin{proofof}{\it\ref{#1}}.\,\,}
	\newcommand{\eprl}{\end{proofof}}

\newcommand{\bprp}[1]{\begin{proofofp}{\emph{\ref{#1}}}.\,\,}
	\newcommand{\eprp}{\end{proofofp}}

\newcommand{\br}[1]{\begin{remark}\label{#1}}
	\newcommand{\er}{\end{remark}}

\newcommand{\bi}{\begin{itemize}}
	\newcommand{\ei}{\end{itemize}}

\newcommand{\ben}{\begin{enumerate}}
	\newcommand{\een}{\end{enumerate}}

\newenvironment{proof}{\noindent {\em Proof}.\,\,}
{\hspace*{\fill}$\square$\medskip}
\newenvironment{proofof}{\noindent {\em Proof of Lemma\,\,}}
{\hspace*{\fill}$\square$\medskip}
\newenvironment{proofofp}{\noindent {\em Proof of Proposition\,\,}}
{\hspace*{\fill}$\square$\medskip}
\newenvironment{proofoft}{\noindent {\em Proof of Theorem\,\,}}
{\hspace*{\fill}$\square$\medskip}

\newcommand\restr[2]{{
		\left.\kern-\nulldelimiterspace 
		#1 
		\vphantom{\big|} 
		\right|_{#2} 
}}

\newcommand\smallrestr[2]{{
		\left.\kern-\nulldelimiterspace 
		#1 
		\vphantom{|} 
		\right|_{#2} 
}}

\def \R {{\mathbb R}}
\def \N {{\mathbb N}}

\def \C {{\mathbb C}}

\def \ra {\rightarrow}

\def \barr{\begin{array}}
	\def \earr {\end{array}}

\def \P {{\mathbb P}}
\def \E {{\mathbb E}}

\def \cL {{\mathcal L}}
\def \cH {{\mathcal H}}
\def \cT {{\mathcal T}}

\def \cE {{\mathcal E}}
\def \cP {{\mathcal P}}

\def \cM {{\mathcal M}}

\def \cK {{\mathcal K}}
\def \cF {{\mathcal F}}

\def\eps {\varepsilon}

\newcommand{\qaq}{\quad \text{ and } \quad}

\newcommand{\qfa}{\quad \text{ for all }  }
 
\definecolor{darkgreen}{rgb}{0,.6,0}
\definecolor{darkagenta}{rgb}{.5,0,.5}
\definecolor{darkred}{rgb}{1,0,0}
\definecolor{darkblue}{rgb}{0,0,.4}
\definecolor{black}{rgb}{0,0,0}
\definecolor{gray}{rgb}{.4,.4,.4}
\definecolor{white}{rgb}{0.99,0.99,0.99}
\definecolor{geel1}{rgb}{1,1,0.3}

\newcommand{\Rm}{m}

\numberwithin{equation}{section}
\newcommand{\capa}{\mathrm{Cap}}
\newcommand{\cov}{\mathrm{Cov}}
\newcommand{\var}{\mathrm{Var}}
\newcommand{\ee}{\mathrm{e}}
\newcommand{\HH}{\mathrm{H}}
\newcommand{\ieps}{\frac{1}{\eps}}

\begin{document}
	\bibliographystyle{abbrv}

\title{Metastability in a continuous mean-field model\\ at low temperature
	and strong interaction 
}

\author{
	\renewcommand{\thefootnote}{\arabic{footnote}}
	K.\ Bashiri\footnotemark[1] \qaq G. Menz\footnotemark[2]
}

\footnotetext[1]{
	Institut f\"ur Angewandte Mathematik,
	Rheinische Friedrich-Wilhelms-Universit\"at, 
	Endenicher Allee 60, 53115 Bonn, Germany.
	Email:
	bashiri@iam.uni-bonn.de.
}
\footnotetext[2]{
	 UCLA Mathematics Department,
	 Box 951555,
	 Los Angeles, CA 90095-1555.
	 E-mail: gmenz@math.ucla.edu.
}

\maketitle

\begin{center}
	\emph{In memory of our dear friend and mentor Dmitry Ioffe.}
\end{center}
 
 \begin{abstract}
 \let\thefootnote\svthefootnote{
 	We consider a system of $ N \in \mathbb{N} $  mean-field interacting stochastic differential equations that are driven by Brownian noise and a single-site potential of the form  $ z \mapsto z^4/4-z^2/2 $.
 	The strength of the noise is measured by a small parameter $ \varepsilon >0$ (which we interpret as the \emph{temperature}), and we  suppose that  the strength of the interaction is given by $ J>0 $.
Choosing  the \emph{empirical mean} ($ P:\R^N \ra \R $, $ Px =1/N \sum_i x_i $) as the macroscopic order parameter for the system, we show that the resulting macroscopic Hamiltonian has two global minima, one at $ -m^\star_\eps <0 $ and one at $ m^\star_\eps>0 $. 
Following this observation, we are interested in the average transition time of the system to $ P^{-1}(m^\star_\eps) $, when the initial configuration is drawn according to a probability measure (the so-called \emph{last-exit distribution}), which is supported around the hyperplane $ P^{-1}(-m^\star_\eps) $.
Under the assumption of strong interaction, $ J>1 $, the main result is a formula for this transition time, which is reminiscent of the celebrated Eyring-Kramers formula (see \cite{BEGK04})  up to a multiplicative error term that tends to $ 1 $ as $ N \ra \infty $ and $ \eps \downarrow 0 $. 
The proof is based on the \emph{potential-theoretic approach to metastability}. 

In the last chapter we add some estimates on the metastable transition time in the high-temperature regime, where $ \eps =1 $, and for a large class of   single-site potentials. 


}
\let\thefootnote\relax\footnotetext[0]{
K.B.\   is partially  supported by the Deutsche Forschungsgemeinschaft (DFG, German Research Foundation) - Projektnummer 211504053 - SFB 1060 and by Germany's Excellence Strategy -- GZ 2047/1, projekt-id 390685813 --  ``Hausdorff Center for Mathematics'' at Bonn University.}
\end{abstract}


\section*{Introduction}
By now it is well-known that   many stochastic systems exhibit a phenomenon called   \emph{metastability}.
A typical situation for this is the following. 
First, for a relatively long time, the system is trapped in a state, (the \emph{metastable state}), which is not the (sole) equilibrium state of the system.
Second, after many unsuccessful attempts and due to local fluctuations, the system finally makes the transition to the (other) equilibrium state (the \emph{stable state}). 
In many cases, this transition is triggered by the appearance of a \emph{critical state}, and the metastable and the stable states are 
modelled through  local minima of a \emph{free energy functional} or \emph{Hamiltonian} corresponding to the system.
For a more detailed introduction to metastability, we refer to \cite[Chapter 1]{BdH15}.



In this paper, we are interested in the metastable behaviour of a  system of $ N \in \N $ stochastic differential equations  
 given by
\begin{align}\label{EqIntroSpinSystem}
\begin{split}
dx^{N,\eps}_i (t)&~=~ - \psi' \left(x^{N,\eps}_i (t)\right) \, dt - \frac{J}{N} \sum_{j=0	}^{N-1} \left( x^{N,\eps}_i (t)- x_j^{N,\eps}(t)\right)\, dt  + \sqrt{2 \eps}\,  dB_i (t),  
\end{split}
\end{align} 
where $ t \in (0,\infty),\ 0\leq i \leq N-1, $ $ \eps\in(0,1] $,    
$ B^{N} = (B_i)_{i=0,\dots,N-1}  $ is an $ N $-dimensional Brownian motion, $ J>0 $ and the \emph{single-site potential} $ \psi :\R\ra \R $ is given by 
$ \psi(z)= \frac{1}{4}z^4 - \frac{1}{2} z^2 . $
We consider the strength $ \eps $ of the Brownian noise as the \emph{temperature} of the system.

We proceed as follows. 
First, in order to analyse the system for large $ N $, we choose the \emph{empirical mean}, $ P:\R^N \ra \R,$  $ Px = 1/N \sum_{i=0}^{N-1} x_i $, as the \emph{macroscopic order parameter}. 
That is, we consider the image of the system under the map $ P $. 
Then, as a result of an improvement of the well-known Cram\'er theorem for this setting, which we call \emph{local Cram\'er theorem} (see Section \ref{SecMacrSystem}), we obtain a function  $ \bar{H}_\eps :\R\ra\R $, which we interpret as the \emph{macroscopic Hamiltonian}  of the system. 
A simple analysis shows that 
$ \bar{H}_\eps $ admits exactly two global minima at $ -m^\star_\eps <0$ and $ m^\star_\eps >0 $, and that $ \bar{H}_\eps $ admits a unique local maximum at $ 0 $.  
This fact indicates that our model  exhibits metastable behaviour with the two metastable states being the hyperplanes $ P^{-1}(m^\star_\eps) $ and $ P^{-1}(-m^\star_\eps) $.
The goal of this paper is   to compute the average transition time to  a region around $ P^{-1}(m^\star_\eps) $,   when the system is initially close to $ P^{-1}(-m^\star_\eps) $.

We tackle this goal in two different regimes,
the first one being the \emph{low-temperature regime}, where the strength $ \eps $ of the Brownian noise  tends to zero, and the second one being the \emph{high-temperature regime}, where we set $ \eps =1 $. 
We obtain the following results in this paper.
\begin{itemize}
	\item In Chapter \ref{ChapSmallNoise} we   show that in the low-temperature regime and under the assumption\footnote{The reason for this assumption is  that, if $ J>1 $, we are able to control the microscopic fluctuations via functional inequalities.
		We explain this in further detail in Remark \ref{RemarkJAss}.} that $ J>1 $, the average transition time is asymptotically given by a formula, which is of a similar form as the well-known Eyring-Kramers formula (see \cite{BEGK04}) up to a multiplicative error term that tends to $ 1 $ as $ N \ra \infty $ and $ \eps \ra 0 $. 
		Such a result is often known as \emph{Kramers' law} in the literature. See \cite{Berglund} for a review on such results. 
	\item In Chapter \ref{ChapRoughEstimates} we consider the high-temperature regime, where we only show that, as $ N \ra \infty $, the average transition time is confined to an interval  $ [\alpha\, \ee^{N\Delta}, \beta\, \ee^{N\Delta} ] $, where $ \Delta = \bar{H}_1(0)-\bar{H}_1(-m^\star_1) $, and $ 0<\alpha<\beta< \infty $ are independent of $ N $.
	This result   still holds true if we replace $ \psi $ by a large class of single-site potentials. 
\end{itemize}

Our proofs are based on the   \emph{potential-theoretic  approach to metastability}, which was 
initiated in the seminal papers \cite{BEGK01}, \cite{BEGK02} \cite{BEGK04} and \cite{eckhoff2005precise}.
Here,  one uses   tools from  potential theory to
tackle metastability.
In particular, one obtains that  the average transition time to the (other) equilibrium state can be expressed in terms of 
quantities from \emph{electric networks}. 
More precisely, in terms of \textit{capacities}, for which powerful variational principles are known.
Hence, the computation of  sharp estimates basically reduces to an  appropriate choice of test functions in those variational principles.
The reader is referred to the monograph \cite{BdH15} for an extensive treatment of this approach.

We now provide some   remarks on the historical background on metastability results in high-dimensional diffusion models. 
In the papers \cite{Barret}, \cite{BBovM}, \cite{BGesuN}, and \cite{BerglundGentz},
Kramers' law has been shown for  systems of $ N $ nearest-neighbour interacting stochastic differential equations in low temperature.
These models are considered as 
$ N $-dimensional approximations of
stochastic partial differential equations. 
A similar setting was studied in 
	\cite{BerglundFernandezGentz}, where, instead of the potential-theoretic approach, the so-called
	\emph{path-wise approach to metastability} was used. 
This approach, initiated in \cite{CGOV84}, first focused on the asymptotic exponential distribution of the transition time to the (other) equilibrium state, and then developed itself upon the \emph{Freidlin-Wentzell theory}. 
	In this approach, the asymptotic behaviour of the average transition time is computed   up to logarithmic equivalence.	
	We refer to the monograph \cite{OV04} for  a comprehensive introduction to this approach.
 
 For \emph{high-dimensional mean-field interacting systems  in the high-temperature regime} (i.e.\ for the setting of Chapter \ref{ChapRoughEstimates} in this paper),  the asymptotic behaviour, up to logarithmic equivalence, of the  average transition time has been stated without proof in \cite[Theorem 4]{dawgaetunneling}. 
 The rough estimates  from Chapter \ref{ChapRoughEstimates} provide a slightly improved version of  this conjecture under different initial conditions. We explain the relation between the results of Chapter \ref{ChapRoughEstimates} and the conjecture formulated in \cite[Theorem 4]{dawgaetunneling} in greater detail  below Theorem \ref{ThmIntroEKF} in Section \ref{SecRoughResult}.
 Up to the authors' best knowledge, besides Chapter \ref{ChapRoughEstimates} of this paper, the only other result which is aimed to verify the conjecture formulated in \cite[Theorem 4]{dawgaetunneling} is given in \cite{GvalaniSchlichting2020}.
 In this paper  the system of Chapter \ref{ChapRoughEstimates}  is considered but with more general interaction and by restricting  the stochastic differential equations \eqref{EqIntroSpinSystem} to a finite torus. 
 The main result in \cite{GvalaniSchlichting2020} is an estimate on the probability that a transition between the two metastable states occurs in a time interval which is independent of $ N $; see \cite[Theorem 1.4]{GvalaniSchlichting2020} for more details. 
 The latter result provides a lower bound for the average transition time in the conjecture formulated in  \cite[Theorem 4]{dawgaetunneling} for the case that the stochastic differential equations are restricted to the torus.

 Hence,  only very few metastability results are known for high-dimensional mean-field interacting systems  in the high-temperature regime.
 In contrast to that, the general behaviour of such systems is by now well-understood. For instance, a \emph{law of large numbers} and a \emph{large deviation principle} for a much broader class than the system \eqref{EqIntroSpinSystem} is shown in \cite{gae} and \cite{dawgae}, respectively. 
 
 Also the case of \emph{high-dimensional mean-field interacting systems in the low-temperature regime} (i.e.\ for the setting of Chapter \ref{ChapSmallNoise} in this paper) has been studied before; see for instance \cite{herrmann2016mean} or \cite{Orrieri} and the references therein.  
 However, up to the authors' best knowledge, the results  of Chapter \ref{ChapSmallNoise} are the first ones on the metastable behaviour of such systems.

 We conclude this  introduction with a comment on the fact that we consider the case of \emph{strong interaction}. 
 In the setting of Chapter \ref{ChapRoughEstimates} in this paper this assumption is essential and can not be removed. 
 More precisely,  we have to assume that the strength of the interaction $ J $ is larger than some quantity (cf.\ Assumption \ref{Assumptions}) in order  to ensure that the macroscopic Hamiltonian $ \bar{H}_1 $ (recall that $ \eps = 1 $ in Chapter \ref{ChapRoughEstimates}) admits exactly the two global minima  $ \pm m^\star_1 $.
 That is, if $ J $ is too small, the system does not admit metastable behaviour under the order parameter given by the empirical mean.
 However,  the situation is different in Chapter \ref{ChapSmallNoise}. 
 In this chapter we also assume that $ J $ is larger than some quantity.
 Namely, we suppose that $ J > 1 $.
 But, and this is the crucial difference to Chapter \ref{ChapRoughEstimates}, it can also be shown that if $ J \leq 1 $, then, for $ \eps  $ small enough,   
 $ \bar{H}_\eps $ is a double-well function and consequently the system may admit metastable behaviour. 
 The reason why we nevertheless make this assumption is  that, if $ J>1 $, we are able to control the microscopic fluctuations via functional inequalities.
 	We explain this in further detail in Remark \ref{RemarkJAss}.
 To study the metastable behaviour also for the case $ J \leq 1 $ is the content of future research. 
 Up to the author's best knowledge this is still an open problem, and there is no relevant literature on this problem yet. 

\paragraph*{Outline of the paper.}
In Chapter \ref{ChapModelAndResults}, we introduce the model, formulate the main results, and sketch the main ideas of this paper. 
In Chapter \ref{ChapSmallNoise}, we provide the full details for the proof of Kramers' law in the low-temperature regime  and under the assumption that $ J>1 $.
In Chapter \ref{ChapRoughEstimates} we compute estimates on the average transition time in the high-temperature regime. 
Finally, in the appendix, we   state some general properties of Legendre transforms,  compute certain asymptotic integrals by using Laplace's  method, and provide the proofs of the   local Cram\'er theorem and the equivalence of ensembles, which are the key ingredients    in this paper.  

\paragraph*{Notation.}\label{SecNotation}
\begin{itemize}
	\setlength\itemsep{-0.1em} 
	\item  $ \cP(Y)  $ denotes the space of Borel probability measures on the topological space $ Y $. 
	\item For $ \mu \in \cM_1(Y)  $ and a Borel map $ f : Y\ra \bar{Y} $,  $ f_\#\mu  $ is the image measure of $ \mu $ by $ f $. 
	\item In this paper, $ x $ is always an element of $ \R^N $ for $ N \in \N $, and its components are denoted by $ x_i $.
	$ z $ and $ m $ are always elements of $ \R $. 
	\item Let $ K \subset \R$ be a compact set, and let $ f: (0,1] \times \N \ra [0,\infty) $.
	In this paper, $ O_K(f(\eps,N)) $ always stands for a function, whose absolute value is bounded   by $ f $ for $ \eps $ small enough, for $ N $ large enough and uniformly in $ K $. 
	More precisely, this means that there exist constants $ C_K>0 $, $ N_K\in \N $ and $ \eps_K \in (0,1] $ such that  $ O_K(f(\eps,N)) = R_K(m,\eps,N)$ for some  function $ R_K: K\times (0,\eps_K] \times \N\cap [N_K,\infty)  \ra \R $ with     
	\begin{align}
	|R_K(m,\eps,N)| \ \leq \ C_K\,f(\eps,N) \qfa    m \in K,\, \eps\in (0,\eps_K],\,  N\geq N_K  .  
	\end{align}	
	If, in addition, we   have that $ c_K\,f(\eps,N) \ \leq \  |R_K(m,\eps,N)| $  for all $  m \in K,\, \eps\in (0,\eps_K'],\,  N\geq N_K'  $ for some $ c_K >0 $, $ N_K'\in \N $ and $ \eps_K' \in (0,1] $ we write $ \Omega_K(f(\eps,N)) $ instead of $ O_K(f(\eps,N)) $. 
\item Similarly,
	 $ O(f(\eps,N)) $   always stands for a function, whose absolute value is bounded   by $ f $ for $ \eps $ small enough and for $ N $ large enough. 
	 That is, there exist constants $ C'>0 $, $ N'\in \N $ and $ \eps' \in (0,1] $ such that  $ O(f(\eps,N)) = R(\eps,N)$ for some  function $ R:  (0,\eps'] \times \N\cap [N',\infty)  \ra \R  $ with     
	 \begin{align}
	 |R(\eps,N)| \ \leq \ C'\,f(\eps,N) \qfa     \eps\in (0,\eps'], \, N\geq N'  .  
	 \end{align}	
	 Finally, we define $ \Omega(f(\eps,N)) $ analogously as $ \Omega_K(f(\eps,N)) $.
	\item Let $ (S,d) $ be a metric space, $ \rho>0 $ and $ s\in S $. Then, define $ B_\rho(s)=\{r\in S \, | \, d(s,r)< \rho		\} $.
	\item Let Y be an Euclidean space. Then we say that  $ \mu \in \cP(Y) $ satisfies the \emph{Poincar\'e inequality} with  constant $ \varrho > 0 $ if for all $ f \in H^1(\mu) $,
	\begin{align}\label{EqPI}
	\var_{\mu} \left( f \right) 
	&~:=~ \int  \left|	f - \int   	f   d\mu  	\right|^2 d\mu 
	~\leq~ \frac{1}{\varrho}  \int  \left|  \nabla f	\right|^2 d\mu  ,
	\end{align}
	where $ \nabla  $ denotes the gradient determined by the Euclidean structure of $ Y $. 
\end{itemize}
\section{The model and the results}\label{ChapModelAndResults}
This chapter is organized as follows. 
In Section \ref{SecMicrSystem} we define the microscopic model. 
Then, in Section \ref{SecMacrSystem} we introduce the macroscopic order parameter, and collect some result on the   energy landscape of the model under this order parameter.
In Section  \ref{SecMainResult} and \ref{SecRoughResult} we formulate the two main results of this paper.
For the sake of comprehensibility, in this introductory chapter, we  only provide  rough formulations of the setting and the main results.
For the full details, we refer to the Chapters \ref{ChapSmallNoise} and \ref{ChapRoughEstimates}.

\subsection{The microscopic model}\label{SecMicrSystem}
Recall the definition of the system $ \{(x_0^{N,\eps}(t), \dots ,x_{N-1}^{N,\eps}(t))\}_{t\in (0, \infty)} $ given in \eqref{EqIntroSpinSystem}. 
For $ t \in (0,\infty)$, let $ x^{N,\eps}(t) = (x_0^{N,\eps}(t), \dots ,x_{N-1}^{N,\eps}(t)) $.
Recall that we have also introduced the \emph{single-site potential} $ \psi   $ and the parameters $ J>0 $ and $ \eps\in(0,1] $.

 The \emph{Gibbs measure} $ \mu^{N,\eps} \in \cP(\R^N) $ corresponding to this model has the form
 \begin{align}\label{EqGibbs}
 \mu^{N,\eps}(dx) = \frac{1}{Z_{\mu^{N,\eps}}} \ \ee^{ - \mathrm H^{N,\eps}(x)} dx,
 \end{align}
 where $ Z_{\mu^{N,\eps}} $ is a normalization constant, and, for $x =(x_i)_{i=0,\dots,N-1} \in \R^N $, the \emph{microscopic Hamiltonian}~$\mathrm H^{N,\eps}: \mathbb{R}^{N} \ra \mathbb{R}$ is defined by 
 \begin{align}\label{EqIntroMicrHam}
 \mathrm H^{N,\eps}(x) = \ieps \sum_{i=0}^{N-1} \psi(x_i) + \ieps \frac{J}{4N } \sum_{i,j=0}^{N-1} (x_i- x_j)^2 .
 \end{align} 
 It is well-known that $ \mu^{N,\eps} $ is the unique stationary measure of the process $ (x^{N,\eps}(t))_{t\in (0, \infty)} $.

\subsection{The macroscopic variables and the macroscopic energy landscape} \label{SecMacrSystem}
The \emph{empirical mean}~$P: \mathbb{R}^N \ra \mathbb{R}$ is defined by
 \begin{align} \label{EqEmpMagn}
 Px ~=~ \frac{1}{N} \sum_{i=0}^{N-1}x_i . 
 \end{align}
This operator will act as the \emph{order parameter} for our microscopic  system.  
That is, in order to analyse the process  $ (x^{N,\eps}(t))_{t\in (0, \infty)} $ for large $ N\in \N $, we study the image of this process under the map   $ P $.
 Therefore, intuitively, $ \bar{\mu}^{N,\eps} := P_\# \mu^{N,\eps} $ describes the (long-time) macroscopic behaviour  of our microscopic model,
 and it will be crucial to study the asymptotic behaviour of this measure.

 In fact, in Proposition  \ref{PropEpsLocalCramer},  we show that, for $ \eps \in (0,1]$ small enough and for any compact set $ K \subset \R $, 
 \begin{align}\label{EqIntroLocCram}
 \bar{\mu}^{N,\eps} (dm)  ~=~
 \ee^{-N \bar{H}_{\eps}(m)} \sqrt{\frac{\varphi_{\eps}''(m) } {2\pi}} \, dm \  \left( 1 + O_K\left( \frac{1}{\sqrt{N}}\right)  \right), 
 \end{align}
 where 
 $ \varphi_\eps: \R \ra \R  $ is the so-called \emph{Cram\'er transform} of the Gibbs measure with respect to the  single-site potential (or more precisely with respect to the effective single-site potential defined   in \eqref{EqSingleSite}) and is defined in \eqref{EqEpsMuBarStar} and \eqref{EqEpsMuBarCramer}, and the function
 $ \bar{H}_{\eps}: \R \ra \R $  is defined by\footnote{The second term in the definition of $  \bar{H}_{\eps} $ is due to the fact that the interaction part can be rewritten in terms of the square of the empirical mean; see \eqref{EqMicrHam} and \eqref{EqSingleSite}.
 	In order to fully understand the motivation behind the definition of $  \bar{H}_{\eps} $ we refer to Subsection \ref{SubSecEpsMuBar}.}
 \begin{align}\label{EqIntroEpsHBar}
 \bar{H}_\eps(z) ~=~ \varphi_\eps(z) ~-~ \ieps \frac{J}{2}z^2.
 \end{align}
Since $ \bar{\mu}^{N,\eps}  $ is the law of the empirical mean of a sequence of   random variables, 
 \eqref{EqIntroLocCram} can be seen as an improvement of the well-known \emph{Cram\'er theorem} (cf.\ \cite[Theorem 6.1.3]{demzei}) for this setting.
 This explains, why we call this result \emph{local Cram\'er theorem}.
 
%
%
%
%
%
%
%
%

Equation \eqref{EqIntroLocCram} shows that, for large $ N $ and for $ \eps  $  small enough, $ \bar{\mu}^{N,\eps}$ is very similar to a Gibbs measure with    $ \bar{H}_\eps $ playing the role of the energy function.  
Therefore,
we consider $ \bar{H}_\eps $ as the \emph{macroscopic Hamiltonian} of the system. 
This suggests   to study the  analytic properties of the function $ \bar {H}_\eps $.
We do this
%
 in Lemma \ref{LemEpsEnergyLandscape}, where we  show that, for $ \eps   $ small enough, 
 $ \bar{H}_\eps $ is a symmetric double-well function   with  two global minima at $ -m^\star_\eps<0 $ and $ m^\star_\eps>0 $, and with a local maximum at $ 0 $. 
 That is, $ \bar{H}_\eps  $ is of the form given in Figure \ref{Figdwh}. 
 \vspace{-0.1cm}
 \begin{figure}[h]
	\centering
	\includegraphics[width=0.3\linewidth, height=0.3\linewidth]{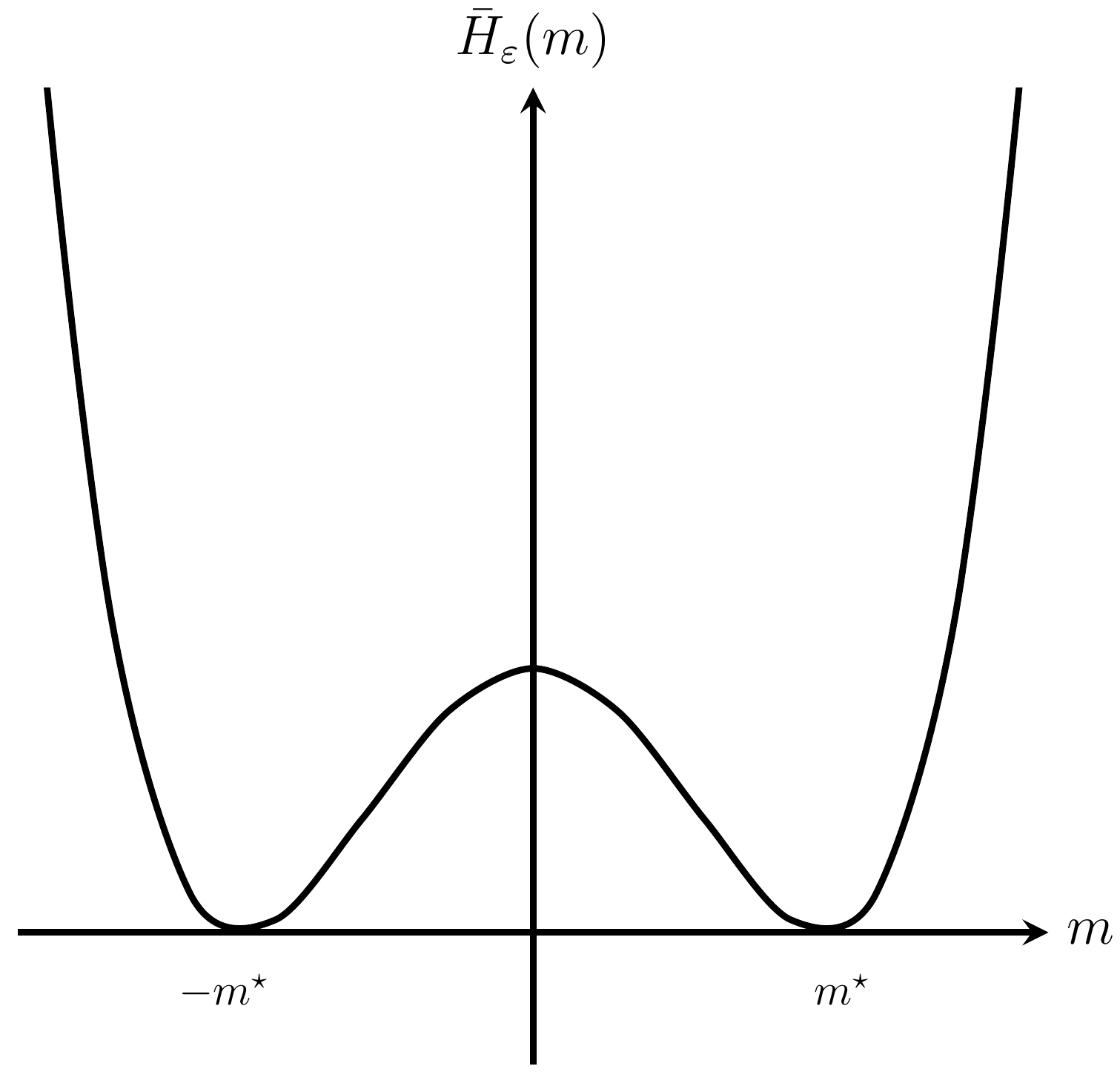}
	 	\caption{Form of the graph of the function $ \bar{H}_\eps  $.}
	\label{Figdwh}
\end{figure}
 
\subsection{The Eyring-Kramers formula at low temperature} \label{SecMainResult}
 The fact that the macroscopic free energy $ \bar{H}_\eps $ has two global minima at $ -m^\star_\eps $ and $ m^\star_\eps $ suggests that our system exhibits metastable behaviour and the hyperplanes $ P^{-1} (-m^\star_\eps) $ and  $ P^{-1} (m^\star_\eps) $ are the \emph{metastable states} of the system. 
 More precisely,
  provided that the initial condition of our system is concentrated in a small region around the   hyperplane $ P^{-1} (-m^\star_\eps) $, 
  we expect  that the average transition time   to hit  a small region around $ P^{-1} (m^\star_\eps) $ fulfils Kramers' law.  
This is the content of the main result of this paper, which is formulated in Theorem \ref{ThmIntroEpsEKF}. 
(For  a more detailed explanation and motivation of the result, we refer to Section \ref{SecEpsEyring}.)
In this theorem, we suppose that
\begin{align}\label{EqIntroJAss}
J>1.
\end{align} 
The reason for this assumption is  that, in this regime, we are able to control the microscopic fluctuations via functional inequalities.
We explain this in further detail in Remark \ref{RemarkJAss}. 
To show the metastable behaviour for the case $ J\leq 1 $  is the content of future research.   
 \begin{theoremm}\label{ThmIntroEpsEKF}
 	Suppose   \eqref{EqIntroJAss}.
%
Then, for $ \eps \in (0,1] $ small enough, and for $ N\in\N $ large enough,
\begin{align}\label{EqIntroMain}
\E_{\nu_{B^-,B^+}}[\cT]
~=~ 
\frac{2\pi \,  \sqrt{\varphi_\eps''(-m^\star_\eps) } \ 	\ee^{N(\bar{H}_\eps (0)-\bar{H}_\eps(-m^\star_\eps ))}}
{\eps\ \sqrt{\bar{H}_\eps''(-m^\star_\eps )\, |\bar{H}_\eps''(0)|   \, \varphi_\eps''(0)}} 
\,\left(1+O\left( \sqrt{\frac{\log(N)^3}{N}}\right) +O\left(\eps \right)\right),
\end{align}	
 	where 
 	\begin{itemize}
 		\item $ \cT ~=~ \inf\{t>0 \, | \, Px^{N,\eps}(t) ~\geq~ m^\star_\eps -\eta  \}   $ for some specific $ \eta ~=~ \Omega(\sqrt{\log(N)/N} \, \sqrt{\eps \log(\eps^{-1})}) $ (see \eqref{EqEpsEta} and $ \eqref{EqEpsEta2} $),
 		\item $ \nu_{B^-,B^+} $ is a probability measure, which is concentrated on the set 
 		$  P^{-1}(-m^\star_\eps+\eta)$ and is called \emph{last-exit biased distribution on $ B^- $} (see Definition \ref{DefPotTheory} and the comments afterwards for the definition and an intuitive description  of the measure $ \nu_{B^-,B^+} $, and see \eqref{EqEpsMetasets} for the definition of the sets $ B^- $ and $ B^+ $) and
 		\item  $ \E_{\nu_{B^-,B^+}}[\cT] ~=~ \int \E_{x}[\cT] \, d\nu_{B^-,B^+}(x)$.
 	\end{itemize}
  
 \end{theoremm}
Note that, since  $ \eta ~=~ \Omega(\sqrt{\log(N)/N} \, \sqrt{\eps \log(\eps^{-1})}) $ and $ \nu_{B^-,B^+} $ is concentrated on  
$  P^{-1}(-m^\star_\eps+\eta)$, the initial condition of our system is initially close to the metastable state given by the hyperplane $  P^{-1}(-m^\star_\eps)$. 
However, the initial condition is random, which is a weakness of our results. 
As it is suggested by one of the anonymous referees, we believe that we can circumvent this issue by applying the techniques of the papers \cite{ArmendarizGrosskinskyLoulakis} and \cite{gaudilliereMilanesi}. 
This is planned for future research. 

In order to prove Theorem \ref{ThmIntroEpsEKF} in Chapter \ref{ChapSmallNoise}, we proceed as follows. 

We first collect in Section \ref{SecEpsPrelim} three  important ingredients.
More precisely,  in Subsection \ref{SubSecEpsMuBar} we state the local Cram\'er theorem (i.e.\  \eqref{EqIntroLocCram}), which 
 is the key  tool in our proof to go from the microscopic variables to the macroscopic ones.  
Then, in Subsection \ref{SubSecEpsHBar} we study the analytic properties of the  macroscopic Hamiltonian $ \bar{H}_\eps$ and  
show that its graph is of the form given in Figure \ref{Figdwh}. 
And as the third ingredient, we collect in Subsection~\ref{SubsecPotTheory} the key elements from potential theory that allow us to rewrite the average transition time, $ \E_{\nu_{B^-,B^+}}[\cT] $, in terms of quantities from electric networks. 
Namely, we show that $ \E_{\nu_{B^-,B^+}}[\cT] $ is equal to the quotient of the \emph{mass of the equilibrium potential} and the \emph{capacity} (see Lemma \ref{LemMeanTransition}).

After we collect these ingredients, we introduce all missing objects of Theorem \ref{ThmIntroEpsEKF} and provide the main steps of its proof in Section \ref{SecEpsEyring}.
This proof is divided into three steps. 
The first step consists of showing the correct upper bound for the capacity in Section \ref{SecEpsUpperB}. 
This is done    by using the so-called \emph{Dirichlet principle} 
(see Lemma \ref{LemDirchletPrinc}).
Here we have to choose an appropriate test function and compute the asymptotic value of the corresponding \emph{Dirichlet form}. 
In the second step, we compute in Section \ref{SecEpsLowerB} the lower bound on the capacity by an adaptation of  the so-called  \emph{two-scale approach}, which was initiated in the paper \cite{GORV}.
This is the main point, where we use the assumption \eqref{EqIntroJAss}. 
We explain this in further detail in Remark \ref{RemarkJAss}. 
Finally, we compute in Section \ref{SecEpsNum} the asymptotic value of the mass of the equilibrium potential. 
This follows from applying standard Laplace asymptotics, and by exploiting that the graph of $ \bar{H}_\eps$ has the form of a double-well function (cf.\ Figure \ref{Figdwh}).

\subsection{Rough estimates at high temperature} \label{SecRoughResult}
We also consider in this paper the situation where the microscopic fluctuations in the system do not become negligible. That is, we study the system $ (x^{N,1}(t))_{t\in (0, \infty)} $ given by \eqref{EqIntroSpinSystem} with $ \eps =1 $.
It is not surprising that the  methods that we use for the  setting in Section \ref{SecMainResult} do not yield the precise Eyring-Kramers formula in the present case. 
The reason is that in this case, the \emph{entropy} of the paths matters substantially, i.e.\ the microscopic fluctuations do not allow to restrict solely to the macroscopic variables determined by the map $ P $.
We believe that, in order to obtain the Eyring-Kramers formula, we need to consider the \emph{empirical distribution}~$K: \mathbb{R}^N \ra \cP(\mathbb{R})$,  
\begin{align} 
Kx ~=~ \frac{1}{N} \sum_{i=0}^{N-1}\delta_{x_i}  
\end{align} 
as the order parameter instead of the empirical mean determined by the map $ P $. 
A heuristic argument for that is the following.

For all $ N\in \mathbb{N} $ and $ t\in (0, \infty) $,  let $ \gamma_N (t) =  K^N(x^{N,1}(t))  $.
Using \cite[(1.8)]{dawgaequasi}, \cite[Proposition 3.36]{BBGradFlow} and the formal Riemannian setting on the Wasserstein space introduced in \cite[Chapter 1]{Ott01} (or in \cite[Section 9.4]{FengKurtz}), we know that the evolution of $ (\gamma_N (t))_{t\in (0, \infty)} $ can be described by 
a diffusion-like equation of the form 
\begin{align}\label{EqB11}
d\langle \gamma_N(t),f\rangle= \langle\mathrm{Grad}_{\mathrm{Wass}} \mathcal{F}(\gamma_N(t)),f\rangle dt +\frac 1{\sqrt N} dM^f_t \qquad \text{for all smooth functions }f,
\end{align}
where $ \langle \cdot, \cdot \rangle $ denotes the duality product on the space of smooth functions, $ \cF $ is the \emph{free energy functional on the Wasserstein space} corresponding to $ (\gamma_N (t))_{t\in (0, \infty)} $  (see \cite[Section 3.4]{BBGradFlow}), $ \mathrm{Grad}_{\mathrm{Wass}} \cF $  is the formal gradient of $ \mathcal{F} $ in the Wasserstein space interpreted in the sense of distributions as in \cite[Definition 9.36]{FengKurtz}, and  $M^f_t$ is a martingale  for all such $f$.
Equation \eqref{EqB11}  suggests intuitively that the random perturbations of the   process $ (\gamma_N (t))_{t\in (0, \infty)} $ are of order $ 1/\sqrt{N}$; see \cite[(1.9)]{dawgaequasi} and the comments afterwards for more details. Moreover, \eqref{EqB11} shows that the potential landscape for this process is given by the free energy functional $ \mathcal{F} $.
This provides a heuristic justification that, 
in the limit as $ N \ra \infty $, one is in a weak noise setting analogously to the setting  in \cite[(1.1)]{BEGK04} but in the infinite dimensional Wasserstein space. 
We therefore believe that,  by adapting  the strategy of the proof of \cite[Theorem 3.2]{BEGK04}   to the Wasserstein setting,  it should be possible to verify Kramers' law   for $ (\gamma_N (t))_{t\in (0, \infty)} $.

More precisely, the latter should be done as follows. 
First, one uses standard results from potential theory in order to represent expected transition times between the metastable states associated to $ (\gamma_N (t))_{t\in (0, \infty)} $ in terms of Dirichlet forms on the Wasserstein space.  
Second, one has to derive sharp asymptotics of these Dirichlet forms in the limit as $ N \rightarrow \infty $.
At this point one should benefit   from the results obtained in \cite{schiavo2018} and \cite{RenesseSturm}, where a Malliavin calculus is constructed on the Wasserstein space.
The rigorous implementation of these thoughts is left for future research.

We note that, in order to show Kramers' law for $ (\gamma_N (t))_{t\in (0, \infty)} $, there are also three other approaches to metastability which seem to be applicable. 
The first one is  based on the characterization of Markov processes as unique solutions of martingale problems (see \cite{Landim19} for an introduction), the second one is based on the  analysis of the corresponding Poisson equation (see \cite{Seo}), and the third one is based on an application of the so-called \emph{Sandier-Serfaty approach} (see \cite{SandierSer} for the Sandier-Serfaty approach and \cite{ArnrichMielkePeletierSavareVeneroni} for its application to show Kramer's law).  

However, we can still obtain estimates for the mean transition time if the order parameter is given by the empirical mean, i.e.\ if the macroscopic variables are determined by the map $ P $. 
Here we   replace $ \psi $ by    single-site potentials of the form  $ z\mapsto \Psi(z) -\frac{J}{2}z^2 $, where $ \Psi:\R\ra\R $ is a symmetric and bounded perturbation of a strictly convex function (cf.~Assumption \ref{Assumptions}).
Moreover, 
  we have to assume that 
$ J > \int_{\R}  \ee^{-\Psi(z)  } \, dz  / (\int z^2 \,\ee^{-\Psi(z)  } \, dz)   $. 
This condition is necessary for $ \bar{H}_1 $ to be of the form of a  double-well function. 
(Note that the objects $\varphi_1,\bar{H}_1, \nu_{B^-_1,B^+_1}, \cT $ are defined as in Theorem \ref{ThmIntroEpsEKF} but with $ \psi $ replaced by  $z\mapsto \Psi(z) -\frac{J}{2}z^2$ and with $ \eps =1 $.) That is, in the case 
$ J \leq \int_{\R}  \ee^{-\Psi(z)  } \, dz  / (\int z^2 \,\ee^{-\Psi(z)  } \, dz)   $, we do not have a metastable behaviour for the system under the order parameter given by the empirical mean.
This is  different than in Theorem \ref{ThmIntroEpsEKF}, where we can show that $ \bar{H}_\eps $ is a double-well function also in the case $ J\leq 1 $ (see Lemma \ref{LemEpsEnergyLandscape}). 
The main result is the following statement.
 \begin{theoremm}\label{ThmIntroEKF}
	 Suppose Assumption \ref{Assumptions}.	Let $ \pm m^\star_1$ be the two global minimisers of the macroscopic Hamiltonian $ \bar{H}_1$.
	 	Then, for all $ N $ large enough and   for some $ a>0 $, which is independent of $ N $,
\begin{align}
\label{EqRoughEstL}
\E_{\nu_{B^-_1,B^+_1}}[\cT]
&~\geq~ 
\frac{2\pi \  \sqrt{\varphi_1''(-m^\star_1) } \ 	\ee^{N(\bar{H}_1(0)-\bar{H}_1(-m^\star_1))}}
{\sqrt{\bar{H}_1''(-m^\star_1)\, |\bar{H}_1''(0)|   \, \varphi_1''(0) }} 
\,\left(1+O\left( \sqrt{\frac{\log(N)^3}{N}}\right)  \right),\ \text{ and}\\
\E_{\nu_{B^-_1,B^+_1}}[\cT]
&~\leq~ (1+a)
\frac{2\pi \  \sqrt{\varphi_1''(-m^\star_1) } \ 	\ee^{N(\bar{H}_1(0)-\bar{H}_1(-m^\star_1))}}
{\sqrt{\bar{H}_1''(-m^\star_1)\, |\bar{H}_1''(0)|   \, \varphi_1''(0) }} 
\,\left(1+O\left( \sqrt{\frac{\log(N)^3}{N}}\right)  \right).
\label{EqRoughEstU}
\end{align}
 \end{theoremm}

 The proof of this result is organized in the same way as the proof of Theorem \ref{ThmIntroEpsEKF}, and is given in Chapter \ref{ChapRoughEstimates}.

Finally, we compare the results of Theorem \ref{ThmIntroEKF} with the conjecture given in  \cite[Theorem 4]{dawgaetunneling}. 
The authors of~\cite{dawgaetunneling} expect that for all $ \delta > 0 $, there exists $ N_\delta\in \N$ such that for $ N \geq N_\delta $ the average transition time is confined to the interval $ [  \ee^{N(\Delta-\delta)},   \ee^{N(\Delta+\delta)} ] $, where $ \Delta = \bar{H}_1(0)-\bar{H}_1(-m^\star_1) $. 
Here we have used the simple fact that $ \bar{H}_1(0)-\bar{H}_1(-m^\star_1) $ can be written in terms of the free energy functional $ \cF $ from \eqref{EqB11}; see \cite[Lemma 2.4]{Ba2020} or \cite[Section IV.2]{PatThesis} for more details on this relation.
In contrast to this, Theorem \ref{ThmIntroEKF} states that, for $ N $ large enough,  the average transition time is confined to an interval  $ [\alpha\, \ee^{N\Delta}, \beta\, \ee^{N\Delta} ] $, where   $ 0<\alpha<\beta< \infty $ are independent of $ N $.
Hence, in terms of asymptotic precision, our result is an improvement of the conjecture stated in \cite{dawgaetunneling}. 

However, we   note that the initial condition in our setting is very different than in the conjecture formulated in \cite[Theorem 4]{dawgaetunneling}. 
The initial configuration in Theorem \ref{ThmIntroEKF}  is randomly drawn according to the last-exit distribution $ \nu_{B^-_1,B^+_1} $, which is concentrated on  
$  P^{-1}(-m^\star_\eps+\eta)$ (which, similarly as it is explained after Theorem \ref{ThmIntroEpsEKF}, is close to the metastable state $  P^{-1}(-m^\star_\eps)$, since $ \eta = \Omega(\sqrt{\log(N)/N} )$).
In contrast to this, the  initial condition in the conjecture formulated in \cite[Theorem 4]{dawgaetunneling} is given by a sequence of deterministic points $ x^N \in \R^N $ such that the corresponding sequence of empirical measures $ (Kx^N)_{N\in\N} $ converges (in a topology which is slightly stronger than the weak topology) to a measure $ \nu $ from the \emph{basin of attraction} associated to the metastable state $  P^{-1}(-m^\star_\eps)$. More explanations and details on the basin of attraction and the energy landscape in this point of view are given in \cite[Proposition 1.4 and Lemma 2.4]{Ba2020}.

\section{The Eyring-Kramers formula at low temperature} \label{ChapSmallNoise}
In order to simplify the notation, we  omit  in this chapter  the superscripts $ N $ and $ \eps $. For example, we abbreviate $x=  x^{N,\eps},$  $\mathrm H= \mathrm H^{N,\eps} $ and $ \mu=  \mu^{N,\eps} $.
Moreover, we rewrite the microscopic Hamiltonian $ \mathrm H  $ (see \eqref{EqIntroMicrHam}) as
\begin{align}\label{EqMicrHam}
\mathrm {H}(x) ~=~ \ieps \sum_{i=0}^{N-1} \psi_J(x_i) -\ieps \frac{J}{2N } \sum_{i,j=0}^{N-1} x_ix_j
~=~ \ieps \sum_{i=0}^{N-1} \psi_J(x_i) -N\ieps \frac{J}{2 } (Px)^2,
\end{align}
where the \emph{(effective) single-site potential} $ \psi_J:\R\ra\R $ is defined by 
\begin{align}\label{EqSingleSite}
\psi_J(z) ~=~ \psi(z) ~+~ \frac{J}{2} z^2 ~=~  \frac{1}{4} z^4 ~+~ \frac{J-1}{2} z^2  .
\end{align}
Recall that, for the strength of the interaction part in this model,  we  assume that
 \begin{align}\label{EqJAss}
J>1.
\end{align}
The outline of this chapter is given   after Theorem \ref{ThmIntroEpsEKF} in  Section \ref{SecMainResult}.  

\subsection{Preliminaries} \label{SecEpsPrelim}

\subsubsection{Local Cram\'er theorem}\label{SubSecEpsMuBar}
This subsection extends the results from \cite[Proposition 31]{GORV} or \cite[Section 3]{MeOt}.
The goal is to find an asymptotic representation     for the measure $ \bar{\mu} = P_\# \mu $. 

The first observation is that   we can disintegrate $ \mu $ with respect to $ \bar{\mu} $ explicitly using the \emph{coarea formula} (\cite[Section 3.4.2]{EvGar}). 
Indeed, as in \cite[p.\ 306]{GORV}, we obtain that 
\begin{align}\label{EqDisintegration}
\int_{\R^N} f(x) \,  d\mu(x) ~=~ \int_{\R} \int_{P^{-1}(m)} f(x) \, d\mu_m(x)  \,  d\bar{\mu}(m)  
\end{align}
for all bounded and measurable $ f:\R^N \ra \R $,
where
the conditional measures (or fluctuation measures)~$\mu_m$ are given by
\begin{align}\label{EqFluctuationMeasure}
d\mu_m(x) ~=~  \mathbbm{1}_{P^{-1}(m)}(x)  \, \ee ^{ -  \ieps\sum_{i=0}^{N-1}\psi_J(x_i)   } \, d\mathcal{H}^{N-1} (x) \,  \ee^{N\varphi_{N,\eps} (m)} .
\end{align} 
Here,  $ \mathcal{H}^{N-1} $ denotes the  $ (N-1) $-dimensional Hausdorff measure and $\varphi_{N,\eps}: \mathbb{R} \to \mathbb{R}$ is defined by 
\begin{align}\label{EqEpsMuBar}
\varphi_{N,\eps} (m) ~=~ - \frac{1}{N} \log \int_{P^{-1}(m) } \ee^{ -  \ieps\sum_{i=0}^{N-1}\psi_J(x_i) }\,  d\mathcal{H}^{N-1} (x).
\end{align}
Moreover, for $ \bar{\mu} $, we obtain the representation
\begin{align}
d\bar{\mu} (m) ~=~ \frac{1}{Z_{\bar{\mu}}} \,  \ee^{ - N \varphi_{N,\eps}(m) + \ieps N \frac{J}{2}  m^2 }\, dm
\end{align}
for some normalization constant $ Z_{\bar{\mu}} $.

It turns out that the asymptotic behaviour of $ \bar{\mu} $ will be determined by the \emph{Cram\'er transform} $ \varphi_\eps $ of the measure $\ee^{-\ieps\psi_J(z)} dz $, which is defined  as the \emph{Legendre transform} of the function  
\begin{align}\label{EqEpsMuBarStar}
 \R \, \ni \, \sigma ~\mapsto~ \varphi^*_\eps(\sigma)  ~=~   \log \int_{\R} \ee^{ \sigma z-    \ieps\psi_J(z) }\,  dz \, \in \,\R.
\end{align}
That is, 
\begin{align}\label{EqEpsMuBarCramer}
\varphi_\eps(m)~=~ \sup_{\sigma \in \mathbb{R}} \left(\sigma m - \varphi^*_\eps(\sigma) \right).
\end{align}
Moreover, for $ \sigma \in \R $, we define the probability measure $ \mu^{\eps,\sigma} \in \cP(\R) $  by 
\begin{align} \label{EqEpsGCMeasure}
d\mu^{\eps,\sigma}(z)  ~=~ \ee^{-\varphi^*_\eps(\sigma) +\sigma z- \ieps\psi_J(z)} \, dz~=~ 	\frac{ \ee^{ \sigma z- \ieps\psi_J(z)} }{ \int_{\R} \   \ee^{\sigma \bar{z}-\ieps\psi_J(\bar{z})} \, d\bar{z}}\, dz.
\end{align}
$ \mu^{\eps,\sigma}   $ is closely related to $ \varphi^*_\eps $ and $ \varphi_\eps $. This can be seen in Section \ref{SecLegendre} in the appendix, where we list several properties of the Cram\'er  transform  that are used in this paper. 
In particular, we have that $ \varphi^*_\eps $ and $ \varphi_\eps $	are strictly convex and smooth, and hence, $  \varphi_\eps''(m)^{\frac{1}{2}} $ is well defined for all $ m \in \R $.

In the following proposition we state the \emph{local Cram\'er theorem}. (Recall that in Section \ref{SecMacrSystem} we explain why this result is called like that.) 
Very similar versions of this result are already known in the literature; see for instance \cite[Proposition 31]{GORV} or \cite[Section 3]{MeOt}. 
The main novelty here is that the result is uniform in $ \eps \ll 1$. 

\begin{proposition}[Local Cram\'er theorem] \label{PropEpsLocalCramer}
Suppose \eqref{EqJAss}.
Let $ K \subset \R $ be compact. 
Then,  there exist $ N_K\in \N $ and $ \eps_K>0 $ such that,  for all  $ 0<\varepsilon<\varepsilon_{K} $, $ N \geq N_K $ and   $ m\in K $,
	\begin{align}\label{EqEpsLocalCramer}
	\ee^{    -N\varphi_{N,\eps} (m)   } \ = \ \ee^{   - N \varphi_\eps (m)   }~ \frac{\sqrt{\varphi_\eps''(m)}}{\sqrt{2\pi}} ~\left( 1+O_K\left(\frac{1}{\sqrt{N}}\right)  
	\right).
	\end{align} 
	In particular, this implies that  
		\begin{align}\label{EqEpsLocalCramerMu}
	d \bar{\mu}(m) \ = \  \frac{1}{Z_{\bar{\mu}}} \  \ee^{   - N \bar{H}_\eps (m)   } ~\frac{\sqrt{\varphi_\eps''(m)}}{\sqrt{2\pi}} ~\left( 1+O_K\left(\frac{1}{\sqrt{N}}\right) 
	\right)  \, dm,
	\end{align} 
	where 
	\begin{align}\label{EqEpsHBar}
	\bar{H}_\eps(z) ~=~ \varphi_\eps(z) ~-~ \ieps \frac{J}{2}z^2.
	\end{align}
\end{proposition}
\begin{proof}
The proof is postponed to Section  \ref{SecProofLocalCramer} in the appendix.
\end{proof}

\subsubsection{Analysis of the energy landscape}\label{SubSecEpsHBar}
Proposition \ref{PropEpsLocalCramer}  indicates that the graph of  $ \bar{H}_\eps $   determines the macroscopic energy landscape  of our system  if the order parameter is given by the empirical mean (see also Section \ref{SecMacrSystem} for more comments). 
%
%
This suggests   to study the  analytic properties of  $ \bar {H}_\eps $, which is the content of the following lemma.

\begin{lemma}\label{LemEpsEnergyLandscape}
Suppose that $ J>0 $. Then, for all $ \eps\in (0,1] $  small enough, 
	\begin{enumerate}[(i)]
		\item 
		 $ \ \lim_{|t| \to \infty } \frac{1}{t^2}\varphi_\eps(t) \ = \ \infty $, \   $ \lim_{|t| \to \infty } \frac{1}{t^2} \bar{H}_\eps(t)  \ = \ \infty $, and
		\item 
$\bar{H}_\eps$ has exactly three critical points located at~$- m^\star_\eps $,~$0$ and~$ m^\star_\eps $ for some\footnote{Recall the definition of the $ O$-,$O_K$-,$\Omega$-,$\Omega_K $-notation which is given in the introduction of this paper; in this particular case, this means that there exists a constant $ c >0 $ such that $ c^{-1} \eps \leq |m^\star_\eps - 1| \leq  c \, \eps $ for all $ \eps\in (0,1] $  small enough.}
 $ m^\star_\eps = 1 +\Omega(\eps) $.
Moreover,  $ \bar{H}_\eps''(0) < 0 $, $ \bar{H}_\eps''(m^\star_\eps ) > 0 $ and $ \bar{H}_\eps''(- m^\star_\eps ) >0 $.
That is,  $\bar{H}_\eps$ has a local maximum at~$0$, and   the two global minima of $\bar{H}_\eps$ are located  at~$\pm m^\star_\eps $.
	\end{enumerate}
\end{lemma} 
\begin{proof}
Part \emph{(i)} follows from a simple argument, which is based on the fact that $ \psi_J $ is super-quadratic at infinity and on H\"older's inequality.
For instance, a proof can be found in \cite[Lemma III.2.6]{PatThesis} for a slightly more general setting.

To show part \emph{(ii)}, first note that the condition $ \bar{H}_\eps'(m) = 0 $ is equivalent to
\begin{align}\label{EqEpsCritPoint1}
m ~=~ (\varphi_\eps^*)'\left( \ieps J m\right)
~=~ \int_\R z \, \ee^{ -\varphi^*_\eps\left(\ieps J m\right) +\ieps J z m - \ieps \psi_J(z) } \, dz,
\end{align}
where we used \eqref{EqLegendre} to get the first equality, and \eqref{EqLegendreDer2} and \eqref{EqEpsGCMeasure} to get the second equality.
Then, we know from  \cite[Proposition 3.1 and Theorem 3.2]{HerrTug} that, for $ \eps $ small enough, there exist exactly three solutions $\pm m^\star_\eps $ and $0$ 
for \eqref{EqEpsCritPoint1}, where $ m^\star_\eps = 1 +\Omega(\eps) $.

We now show that $ \bar{H}_\eps''(0) < 0 $ in
 the case $ J>1 $.
 Using \eqref{EqLegendreDer} and that  $  \varphi_\eps'(0) =0  $ we have that  for $ \eps $ small enough,    
\begin{align}\label{EqCurvatureSaddle}
\bar{H}_\eps''(0)\,=\, \left( \int_\R z^2 \, d\mu^{\eps,0}(z) \right) ^{-1}\,-\, \ieps \, J 
\,=\,  \frac{J-1}{\eps} \, (1+\Omega(\eps)) \,-\, \ieps  \, J \,<\,0 , 
\end{align}
where we have applied  Corollary \ref{CorAsymp2} in the second equality with $U =\psi_J $ (and hence $z_m=0$). 
In order to show that $ \bar{H}_\eps''(0) < 0 $ in
the case  $ J<1 $, note that here the function $ \psi_J $ admits exactly two minima at $   \pm \sqrt{1-J} $. Hence,  by standard Laplace asymptotics we have  that for $ \eps  $ small enough
\begin{align}
&\bar{H}_\eps''(0)\,=\, \left( \int_\R z^2 \, d\mu^{\eps,0}(z) \right) ^{-1}\,-\, \ieps \, J 
\,=\, \left( \frac{ \int_\R z^2 \, \ee^{- \ieps\psi_J(z)} dz}{ \int_{\R} \   \ee^{-\ieps\psi_J(\bar{z})} }\right) ^{-1}\,-\, \ieps \, J \\
&~=\, \left( \frac{2\ee^{- \ieps\psi_J(\sqrt{1-J})} \left(   \sqrt{\frac{2\pi\, \eps}{\psi_J''(\sqrt{1-J})}} (1-J) + O(\eps)\right) }{2\ee^{- \ieps\psi_J(\sqrt{1-J})}\left( \sqrt{\frac{2\pi\, \eps}{\psi_J''(\sqrt{1-J})}}+ O(\eps)\right) }\right) ^{-1}\,-\, \ieps \, J
\,=\,
\Omega(1)-\ieps  J <0 .
\nonumber
\end{align} 
The same result holds also for the case $ J=1 $, since $  \bar{H}_\eps''(0) $ depends continuously on $ J $ (cf.\ Step 5.3 in the proof of \cite[Theorem 3.2]{HerrTug}).
	
By the symmetry of $ \bar{H}_\eps$, it only remains to show that 	$ \bar{H}_\eps''(m^\star_\eps ) > 0 $.
First note that, since $ m^\star_\eps = 1 +\Omega(\eps) $, for all $ J>0 $, the function $ z\mapsto \psi_J(z) - J m^\star_\eps z$ admits a unique global minimum at some point 
$ z_\eps=1 +\Omega(\eps) $.
Indeed, in the case $ J>1 $, this follows by simply observing that $ \psi_J' $ is invertible, and 
in the case $ J\leq 1 $, we have to apply Cardano's formula (see \cite[Chapters 1 and 2]{Algebra}).
(We omit the details in the latter case, since we do not use the claim of this lemma for the case $ J\leq 1 $
 in the remaining part of this paper.)
Then, as above, using \eqref{EqLegendreDer} and that $  \varphi_\eps'(m^\star_\eps) =\ieps Jm^\star_\eps  $ implies that for $ \eps  $ small enough, 
	\begin{align}\label{EqCurvatureSaddlem}
\begin{split}
	\bar{H}_\eps''(m^\star_\eps)\,&=\, \left( \int_\R \left( z -\int_\R \bar{z }\, d\mu^{\eps,\varphi_\eps'(m^\star_\eps)}(\bar{z}) \right)^2 \, d\mu^{\eps,\varphi_\eps'(m^\star_\eps)}(z) \right) ^{-1}\,-\, \ieps \, J \\
	\,&=\,  \ieps \psi_J''\left(z_\eps \right) \left( 1+ O\left( \eps\sqrt{ \log (\eps^{-1})^3}\right) \right)  \,-\, \ieps \, J \\
		\,&=\,  \ieps  \left( 3z_\eps^2 -1 \right) \left( 1+ O\left( \eps\sqrt{ \log (\eps^{-1})^3}\right) \right) ~>~0  ,
\end{split}
	\end{align}
where we have used Corollary \ref{CorAsymp3} with the function   $  (m,\eps,z)\mapsto V(m,\eps,z) = \psi_J(z) - Jm^\star_\eps z  $ and $ z_{m,\eps} = z_\eps $. This concludes the proof. 
\end{proof}

\begin{remark}
\emph{	In the remaining part of this paper, we   suppose that $ \eps  $ is small enough such that $ [-m^\star_\eps, m^\star_\eps] \subset [-2,2]$.}
\end{remark}
 
  \subsubsection{Potential-theoretic approach to metastability}\label{SubsecPotTheory}

  In this subsection, we quickly  review the key ingredients form the potential-theoretic approach to metastability that we need in our setting.   
  We follow 
  \cite[Chapter 2]{BEGK04}, where   all the omitted details can be found.

  The generator of the stochastic process $ (x(t))_{t\in (0, \infty)} $ introduced in Section \ref{SecMicrSystem} is given by
  \begin{align}
  \cL~=~ \eps \, \ee^{\HH} \left( \nabla \, \ee^{-\HH} \, \nabla \right) ,
  \end{align}
  where $ \HH $ is the microscopic Hamiltonian (recall \eqref{EqMicrHam}). 
  We need the following definitions. 
   \begin{definition}\label{DefPotTheory}
   	Let $ A,D \subset \R^N $ be   open and regular and such that $  A\cap D = \emptyset $ and $ (A\cup D )^c $ is connected.
   	For any   $ B \subset \R^N $, we write $ \cT_B \,=\, \inf\{t>0 \, | \, x(t) \in B\, \}  $.
 	\begin{enumerate}[(i)]
 		\item The \emph{equilibrium potential between $ A$ and $ D  $}, $ f^*_{A,D} $, is defined as the unique solution to the   Dirichlet problem
 		\begin{align}
 		\label{EqDirichletProb}
 		\begin{split}
 		(-\cL f)(x) &~=~ 0 , \qquad \text{ for } x \in (A\cup D )^c,\\
 		f(x) &~=~ 1 , \qquad \text{ for } x \in A,\\
 		f(x) &~=~ 0 , \qquad \text{ for } x \in D.
 		\end{split}
 		\end{align}
 		For $ x \in (A\cup D )^c $, we have the probabilistic interpretation that $ f^*_{A,D} \,=\, \P_x[\cT_A<\cT_D]$.
 		\item 
 		The \emph{equilibrium measure}, $ e_{A,D} $,  is defined as the unique measure on $ \partial  A $ such that 
 		\begin{align}
 		f^*_{A,D}(x)~=~ \int_{  \partial A} G_{D^c}(x,y)  \, e_{A,D}(dy) \quad \text{ for } x \in (A\cup D )^c,
 		\end{align} 
 		where $ G_{D^c}$ is the \emph{Green function} corresponding to $ \cL  $ on $ D^c $ (cf. \cite[(2.2) and (2.3)]{BEGK04}).
 		\item 
 		The \emph{capacity}, $ \capa(A,D) $, of the \emph{capacitor} $ (A,D) $ is defined by
 			\begin{align}
 		\capa(A,D) ~=~ \int_{ \partial A} \ee^{-\HH(y)} \, e_{A,D}(dy) .
 		\end{align} 
 		\item 
 		The \emph{last-exit biased distribution on} $ A $, $ \nu_{A,D} $, is the probability measure on $  \partial A $ defined by
 		 			\begin{align}\label{EqLastExitMeasure}
 	 \nu_{A,D}(dy)~=~ \frac{ \ee^{-\HH(y)} \, e_{A,D}(dy) }{	\capa(A,D)}.
 		\end{align} 
 	\end{enumerate}
 \end{definition}

  Before we proceed and collect the other ingredients from potential theory that we need in this paper, we would like to provide a brief intuitive description of the \emph{equilibrium measure}, $ e_{A,D} $, and the \emph{last-exit biased distribution on} $ A $, $ \nu_{A,D} $. 
  
  First, heuristically, the \emph{equilibrium measure}  $ e_{A,D}$ describes the escape probability from $ A $  towards $ D $. That is, it describes the probability to enter $ D $ without returning to $ A $. 
  The reason why this is true can be understood best in the discrete case, i.e.\ in the  case when $ A $ and $ D $ are discrete sets. Indeed, in this case, $ e_{A,D} (y) = \P_y[\cT_D < \cT_A]$ for $ y\in A $; see \cite[(7.1.19), (7.1.20) and below (7.1.23)]{BdH15} for more details.
  
  Second, heuristically, the \emph{last-exit biased distribution on $ A $}, $ \nu_{A,D} $, can be seen as the Gibbs measure of the system restricted to $ \partial A $ and weighted by the escape probability from $ A $  towards $ D $.
  Intuitively, it describes the Gibbs measure of the system  ``conditioned'' on the event that the next  \emph{exit from $ A $} is the final (or last) exit from $ A $ and leads to the transition towards $ D $. 
  In order to explain this more precisely, note that, by \eqref{EqLastExitMeasure},  
  $ \nu_{A,D} $ is given by the  measure $ \ee^{-\HH(y)} dy $  restricted to $ \partial A $ and weighted by $ e_{A,D}$. 
  Moreover, note that $ \ee^{-\HH(y)} dy $ can be seen as the Gibbs measure of the system\footnote{$ \ee^{-\HH(y)} dy $ can be seen as the Gibbs measure of the system, since it is the unique stationary measure of the stochastic process $ (x(t))_{t\in (0, \infty)} $ and has the form of a Gibbs measure with respect to  the microscopic Hamiltonian $ \HH $.} and that, as we have explained before, $ e_{A,D}$ describes the escape probability from $ A $  towards $ D $.
In this picture, the capacity, $ \capa(A,D) $, serves as the normalization constant of the restricted measure $ \nu_{A,D} $.
  
  Using the notions from Definition \ref{DefPotTheory} one can rewrite the average hitting time of $ B $ in the case that the initial condition is randomly chosen according to the last-exit distribution. This is the content of the following lemma. 
  \begin{lemma}\label{LemMeanTransition}
  	Consider the same setting as in Definition \ref{DefPotTheory}. 
  	Then,  
  	\begin{align}\label{EqMeanTransition}
\E_{\nu_{A,D}}[\cT_D] ~:=~
\int_{\partial A}\E_y[\cT_D]\, \nu_{A,D}(dy) ~=~
\frac{\int_{D^c}f^*_{A,D}(y)\, e^{-\HH(y)}\, dy}{\capa(A,D)}.
  	\end{align}
  \end{lemma}
\begin{proof}
	The proof can be found in \cite[Corollary 7.30]{BdH15}. See also \cite[(2.27)]{BEGK04}.
\end{proof}

As we already mentioned, the main advantage to use Lemma \ref{LemMeanTransition} is the availability of variational principles for the capacity. 
In this paper, we use  the so-called \emph{Dirichlet principle}, which is stated  in the following lemma. 
\begin{lemma}[Dirichlet principle]\label{LemDirchletPrinc}
	Consider the same setting as in Definition \ref{DefPotTheory}.
	Let 
	\begin{align}
	\cH_{A,D} ~=~ \left\{
	f \in H^1(\R^N\,;\,\ee^{-\HH(x)} \, dx ) ~\Big|~ \smallrestr{f}{A} =1, \ \smallrestr{f}{D} =0, \forall x \in \R^N: f(x )\in [0,1] ,\
	\right\},
	\end{align}
	and define the \emph{Dirichlet form on $ (A\cup D)^c $}, $ \cE_{(A\cup D)^c} :\cH_{A,D} \ra [0,\infty] $, by 
	\begin{align}
	\cE_{(A\cup D)^c}(f) ~=~    \eps
	\int_{(A\cup D)^c} |\nabla f(x)|^2 \, \ee^{-\HH(x)} \, dx    \quad \text{ for }f \in \cH_{A,D}.
	\end{align}
	Then, 
	 \begin{align}
	\capa(A, D) ~=~  \inf_{  f \in \cH_{A,D}} \,  \cE_{(A\cup D)^c}(f)   ~=~  \cE_{(A\cup D)^c}(f^*_{A,D}) .
	\end{align}
	\end{lemma}
  \begin{proof}
  	The proof can be found in \cite[Theorem 7.33]{BdH15}. The statement can also be found in \cite[(2.15)]{BEGK04}.
  	\end{proof}

 \subsection{The proof of Theorem \ref{ThmIntroEpsEKF}}\label{SecEpsEyring}
 In this section we first introduce all missing ingredients of Theorem \ref{ThmIntroEpsEKF}. 
 Afterwards we provide the proof of this theorem. 
 However, as we explained after the formulation of Theorem \ref{ThmIntroEpsEKF} in  Section \ref{SecMainResult}, the main three steps of this proof are moved to Section \ref{SecEpsUpperB}, Section \ref{SecEpsLowerB} and Section \ref{SecEpsNum}, respectively. 
 
Recall that, under \eqref{EqJAss} and for $ \eps $ small enough, the macroscopic Hamiltonian admits exactly two global minima $ \pm m^\star_\eps $. 
We therefore consider the hyperplanes $ P^{-1}(-m^\star_\eps) $ and $ P^{-1}(m^\star_\eps) $ as the \emph{metastable sets} in our system. 

The goal in this paper is to use the potential-theoretic setting from Subsection \ref{SubsecPotTheory} to compute  the average transition time from  $ P^{-1}(-m^\star_\eps) $ to $ P^{-1}(m^\star_\eps) $ for the stochastic process $ (x(t))_{t\in (0, \infty)} $ introduced in Section \ref{SecMicrSystem}. 
However, due to technical reasons, we have to   modify this goal in two ways. 

First, instead of considering $ P^{-1}(-m^\star_\eps) $ and $ P^{-1}(m^\star_\eps) $ as the metastable sets, we rather consider 
$ P^{-1}(-m^\star_\eps+\eta) $ and $ P^{-1}(m^\star_\eps-\eta) $, where 
\begin{align}\label{EqEpsEta}
\eta =  \frac{\sqrt{2}}{\sqrt{N \bar{H}_\eps''(-m^\star_\eps)}}\sqrt{ \log (N\eps^{-1})}.
\end{align}
Note that, by using the arguments from Step 3 of the proof of Proposition \ref{PropEpsUpperBound}, we  have that 
\begin{align}\label{EqEpsEta2}
	 \eta \, =\, \Omega\left(\sqrt{\frac{\log(N)}{N}} \, \sqrt{\eps \log(\eps^{-1})}\right)  .
\end{align}
Heuristically, the reason for this shift is the following. 
In the proof of Theorem \ref{ThmIntroEpsEKF} we have to compute the integral in the  numerator on the right-hand side of \eqref{EqMeanTransition}. 
Using the disintegration \eqref{EqDisintegration}, Proposition \ref{PropEpsLocalCramer} and the fact that $ \bar{H}_\eps $ has its global minima at $ -m^\star_\eps $ and $ m^\star_\eps $, we see that this integral is concentrated on the sets 
$ \{\,	x \, | \, Px \in  [\pm m^\star_\eps- \eta, \pm m^\star_\eps+ \eta 	]\,	\} $.
Hence,  in order to apply Laplace's method,  we need that the equilibrium potential is equal to $ 1 $ or equal to $ 0 $ on these sets, respectively.

Second, instead of running the system from some specific point in $ P^{-1}(-m^\star_\eps+\eta) $, we rather have to initialise our system randomly according to the last-exit biased distribution $ \nu_{B^-,B^+} $, where  $ B^-, B^+\subset \R^N $ are defined by 
\begin{align}\label{EqEpsMetasets}
\begin{split}
B^- ~=~ \{\	x\in \R^N \ | \ Px ~\leq~  - m^\star_\eps+ \eta 	\	\}
\qaq 
B^+~=~\{	\ x\in \R^N \ | \ Px ~\geq~ m^\star_\eps -\eta 	\	\}.
\end{split}
\end{align} 
Note that $ \nu_{B^-,B^+} $ is a probability measure supported on $ \partial B^-= P^{-1}(-m^\star_\eps+\eta) $. 
The main reason for the choice of this initial distribution is that we can exploit the formula \eqref{EqMeanTransition}. 
However, in a finite-dimensional setting, such as in \cite{BEGK04}, we could also obtain an asymptotic expression for  $\E_y[\cT_{B^+}] $ for $ y  \in \partial B^- $. 
This is done by using Harnack inequalities. 
But since these inequalities depend on the dimension of the base space, 
we are not able to transfer the strategy used in \cite{BEGK04} to our high-dimensional setting. 
 
We are now in the position to prove Theorem \ref{ThmIntroEpsEKF}.

   \begin{proofoft}{\emph{\ref{ThmIntroEpsEKF}}.}  
Note that  $ \cT \equiv \cT_{B^+} $, where the stopping time $ \cT $ is defined in Theorem \ref{ThmIntroEpsEKF} and $ \cT_{B^+} $ is defined in Definition \ref{DefPotTheory}.
Then, by using \eqref{EqMeanTransition}, we have that 
\begin{align}
\E_{\nu_{B^-,B^+}}[\cT_{B^+}]
~=~ \frac{ 	\int_{(B^+)^c}f^*_{B^-,B^+}(y)\, \ee^{-\mathrm{H}(y)}\, dy }{\capa(B^-, B^+)} .
\end{align}
Then, the proof follows from a simple application of the Propositions \ref{PropEpsUpperBound}, \ref{PropEpsLowerBound} and \ref{PropEpsNum}. 
   \end{proofoft}

\subsection{Upper bound on the capacity}\label{SecEpsUpperB}

 \begin{proposition}\label{PropEpsUpperBound}
 	Consider the same setting as in Theorem \ref{ThmIntroEpsEKF}.
Then, for $ N\in \N  $ large enough and $ \eps\in (0,1] $ small enough,
 	 \begin{align}\label{EqUpperBound}
 	\capa(B^-, B^+) ~\leq~ \eps\,  \frac{1}{ 2\pi } \, 
 	                                \ee^{-  N \bar{H}_\eps (0)}   \,  \sqrt{  |\bar{H}_\eps''(0)|  } \, \sqrt{  \varphi_\eps''(0)    }  \,
 	   \left( 1 + O \left( \sqrt{\frac{\log(N)^3}{N}}   \right) \right).
 	\end{align}
 \end{proposition}
\begin{proof}
	We will obtain the upper bound by using the Dirichlet principle (Lemma \ref{LemDirchletPrinc}).
	That is, we introduce a suitable test function and show that the corresponding Dirichlet form is asymptotically given by the right-hand side of \eqref{EqUpperBound}.
	
	\textbf{Step 1.} [Choice of the test function $ f $.]\\
Let $ \rho>0 $ and the function $ h:[-\rho, \rho] \ra [0,1] $ be defined by
	\begin{align}\label{EqEpsHarmMacr} 
		\rho =   \sqrt{ \log (N) \, N^{-1} \,  |\bar{H}_\eps''(0)|^{-1}}  \qaq
	h^*(m) ~=~ \frac{\int_{m}^{\rho}   \varphi''_\eps(z)^{-\frac{1}{2}} \  \ee^{ N\bar{H}_\eps(z)} \, dz}{\int_{-\rho}^{\rho}  \varphi''_\eps(z)^{-\frac{1}{2}} \  \ee^{ N\bar{H}_\eps(z)} \, dz},
	\end{align}
	which is well-defined, since $ \varphi_\eps $ is strictly convex.
Then, $ h^* $ is the equilibrium potential between $ \{-\rho\} $ and $ \{\rho\} $ corresponding to the invariant measure $ \mathbbm{1}_{(-\rho, \rho)}(z) \, \varphi''_\eps(z)^{\frac{1}{2}} \  \ee^{- N\bar{H}_\eps(z)} \, dz $. That is, $ h^* $ is the unique solution to the Dirichlet problem (cf.\ Definition \ref{DefPotTheory}, \cite[Section 7.2.5]{BdH15} or \cite[(11.2.18) and (11.2.19)]{BdH15})
\begin{align}
\begin{split}
\left(-\frac{1}{N} \frac{d}{dm} -  \frac{d}{dm }  \frac{1}{2N}\log\left(  \varphi''_\eps(m) \right) +\frac{d}{dm } \bar{H}_\eps(m)  \right)   \frac{d}{dm} h^*(m) &~=~ 0 \qquad \text{for }m\in (-\rho, \rho),\\
h^*(-\rho) &~=~ 1 ,\\
h^*(\rho) &~=~ 0 .
\end{split}
\end{align} 
The test function that we use in this proof is given by
\begin{align}
f(x)= 
\begin{cases}
1 &\text{if } Px~\leq~ -\rho, \\
0 &\text{if } Px~\geq~ \rho ,\\
h^*(Px) &\text{if } Px~\in~ (-\rho,\rho) .
\end{cases}
\end{align}

\textbf{Step 2.} [Estimation of the Dirichlet form of  $ f $.]\\
Using the Dirichlet principle (Lemma \ref{LemDirchletPrinc}) we have that 
 \begin{align}
\begin{split}
 \frac{1}{ Z_{ \mu }}\capa(B^-, B^+) 
 &\leq \eps \, 
  \int_{\{x\in\R^N\,|\,Px\in (-\rho,\rho)\}} \sum_{i=0}^{N-1} \left|\partial_i h^*\left( \frac{1}{N} \sum_{i=0}^{N-1} x_i\right) \right|^2 \ d\mu\\
  &=  \eps \, 
  \frac{1}{N}
   \int_{\{x\in\R^N\,|\,Px\in (-\rho,\rho)\}}  |  (h^*)'(Px)|^2 \ d\mu\\
  &=  \eps \, 
 \frac{1}{N}
 \int_{-\rho}^{\rho}\int_{P^{-1}(m)}   |  (h^*)'(m)|^2 \ d\mu_m d\bar{\mu}(m),
 \end{split}
  \end{align}
  where, in the first inequality, we use Lemma \ref{LemDirchletPrinc} and the definition of the test function $ f $, in the first equality, we use the definition of $Px$, and in the second equality, we use the disintegration \eqref{EqDisintegration}.
  
  Applying   Proposition \ref{PropEpsLocalCramer} for $ K  = [-2,2]  $ and the definition of $ h^* $ yields that  
   \begin{align}
   \label{EqEpsUpper00}
\nonumber \frac{1}{ Z_{ \mu }}   \capa(B^-, B^+) 
   &\leq      \,   \frac{\eps}{\sqrt{2\pi}}  
 \frac{1}{N}\frac{1}{Z_{\bar{\mu}}}
 \int_{-\rho}^{\rho}    |  (h^*)'(m)|^2 \      \sqrt{\varphi''_\eps(m)} \,  \ee^{ -N\bar{H}_\eps(m)}  \, dm  \, \left( 1+O\left(\frac{1}{\sqrt{N}}\right) \right)   \\
   &=  \, 
   \frac{\eps}{\sqrt{2\pi}}  
   \frac{1}{N}\frac{1}{Z_{\bar{\mu}}}
   \left( \int_{-\rho}^{\rho}  \frac{1}{  \sqrt{\varphi''_\eps(m)}} \,  \ee^{ N\bar{H}_\eps(m)}  \, dm \right)^{-1} \, \left( 1+O\left(\frac{1}{\sqrt{N}}\right)\right) .
    \end{align}
    In Step 4 and 5 of this proof we show that for $ m~\in~ [-\rho,\rho] $, 
      \begin{align}
\label{EqEpsUpper0}
  \sqrt{\varphi''_\eps(m)} 
  &~=~  \sqrt{\varphi''_\eps(0)} \, \left( 1+O_K\left(\sqrt{\frac{\eps\,\log(N)}{N}}\right)\right), \text{ and }\\
  \bar{H}_\eps(m) &~=~ \bar{H}_\eps(0)    + \frac{1}{2}m^2   \bar{H}_\eps''(0)  
  +O_K\left( \sqrt{\eps}\sqrt{\frac{ \,\log(N)}{N}}^3\right).
\label{EqEpsUpper}
  \end{align}  
  And since, by the coarea formula, $ Z_{\bar{\mu}} \sqrt{N} = Z_{\mu} $,  \eqref{EqEpsUpper0} and \eqref{EqEpsUpper} imply that 
   \begin{align}
   \label{EqEpsUpper01}
  \begin{split}
  \capa(B^-, B^+) 
  &\leq    \frac{ \eps \, \sqrt{\varphi''_\eps(0)} \, \ee^{-N\bar{H}_\eps(0) }}{\sqrt{2\pi}\sqrt{N}}  \, 
  \left( \int_{-\rho}^{\rho}     \ee^{\frac{1}{2}Nm^2   \bar{H}_\eps''(0)    }   \, dm \right)^{-1} \, \left( 1+O\left(\sqrt{\frac{\log(N)^3}{N}}\right)\right)  .
  \end{split}
  \end{align}
 The integral on the right-hand side of \eqref{EqEpsUpper01} can be estimated from below as follows.
    \begin{align}  
   \nonumber
 \int_{-\rho}^{\rho}     \ee^{\frac{1}{2}Nm^2   \bar{H}_\eps''(0)    }   \, dm 
  &~\geq~ 
  \left(  \int_{\{ x\in \R^2 \, |\, |x|\leq \rho \}}     \ee^{\frac{1}{2}N|x|^2   \bar{H}_\eps''(0)    }   \, dx    \right)^{\frac{1}{2}} 
  ~=~
  \sqrt{\frac{2\pi}{N| \bar{H}_\eps''(0)   |}}
  \left(  1-     \ee^{ \frac{1}{2} N\rho^2    \bar{H}_\eps''(0)   }    \right)^{\frac{1}{2}} \\
  &~=~
  \sqrt{\frac{2\pi}{N| \bar{H}_\eps''(0)   |}}
  \left(  1+O\left(\frac{1}{\sqrt{N}}\right)    \right) ,
 \label{EqEpsUpper02} 
  \end{align}   
  Combining \eqref{EqEpsUpper01} with \eqref{EqEpsUpper02}  concludes the proof of \eqref{EqUpperBound}.
  
  \textbf{Step 3.} [Some a priori estimates.]\\
  Before we show \eqref{EqEpsUpper0} and  \eqref{EqEpsUpper}, we collect some a priori estimates. 
  First, we note that there exists $ c>0 $ such that for all $ m \in K $ and $ \eps  $ small enough, 
    \begin{align}
  \label{EqEpsUpper4}
  |\varphi''_\eps(m)| ~=~\varphi''_\eps(m) ~=~ \frac{1}{(\varphi_\eps^*)''(\varphi'_\eps(m))} ~\in ~ 
  \left[\frac{c^{-1}}{\eps}, \frac{c}{\eps}\right],
  \end{align}
  where we used \eqref{EqLegendreDer} in the second equality, and then applied  Lemma \ref{LemMoments} \emph{(iii)}.
  Moreover, recall that in the proof of Lemma \ref{LemEpsEnergyLandscape}, we have seen that  
$   |\bar{H}_\eps''(0)| ~=~ 1/\eps (1+O(\eps)) $.
Therefore, for $ \eps  $ small enough, $   |\bar{H}_\eps''(0)| ~\geq ~ 1/(4\eps) $.

Next, we note that it is shown in Lemma \ref{LemMoments} that there exists $ \eps_K>0 $ and a   function 
$ \tau :K \times [-\eps_K,\eps_K] \ra \R $  such that 
$ \sup_{0<\eps<\eps_K} \sup_{m \in K  } |\tau(m, \eps)| <\infty  $
and $ 	\varphi_{\eps }'(m) = \ieps \tau(m, \eps) $ for all $ m \in K $ and   $ 0<\varepsilon<\varepsilon_{K} $.
Then,  we see that there exists $ c'>0 $ such that for all $ m\in [-\rho,\rho] \subset K $  and for $ \eps  $ small enough,
  \begin{align} 
  \label{EqEpsUpper5}
\begin{split}
   |\varphi'''_\eps(m)| &~=~ 
   \left| \frac{(\varphi_\eps^*)'''(\varphi'_\eps(  m)) }{(\varphi_\eps^*)''(\varphi'_\eps(  m))^{3}}\right| ~=~ \left| \int_{\R} \left(z-  (\varphi_\eps^*)'(\varphi'_\eps(  m))\right)^{3}\,  \, d\mu^{\eps,\varphi'_\eps(  m)}(z) \right| \,   \varphi''_\eps(m)^{3}\\
 &~=~ \left| \int_{\R} \left(z-  \int_\R \bar{z }\, d\mu^{\eps,\varphi_\eps'(m)}(\bar{z})\right)^{3}\,  \, d\mu^{\eps,\varphi'_\eps(  m)}(z) \right| \,   \varphi''_\eps(m)^{3}
 ~\in ~ 
 \left[\frac{(c')^{-1}}{\eps}, \frac{c'}{\eps}\right],
\end{split}
   \end{align} 
  where we used the second equation in \eqref{EqLegendreDer} in the first step, the first equation in \eqref{EqLegendreDer} and the third equation in \eqref{EqLegendreDer2} in the second step, the first equation in \eqref{EqLegendreDer2} in the third step, and \eqref{EqEpsGCMeasure}, \eqref{EqEpsUpper4} and Corollary \ref{CorAsymp3} for the function  $ (m,\eps,z)\mapsto V(m,\eps,z) = \psi_J(z) +  \tau(m, \eps) z $  in the last step.
  
  \textbf{Step 4.} [Proof of \eqref{EqEpsUpper0}.]\\
    By Taylor's formula, we have for some $ \theta  \in [0,1] $,
    \begin{align}
\label{EqEpsUpper1}
     \sqrt{\varphi''_\eps(m)} 
     ~= ~ 
      \sqrt{\varphi''_\eps(0)}~\left( 1~+~
      m ~
     \frac{ \varphi'''_\eps(\theta m)}{ 2\sqrt{\varphi''_\eps(0)}\sqrt{\varphi''_\eps(\theta m)}}\right) .
    \end{align} 
Then, by  the estimates from Step 3, 
\begin{align}\label{EqEpsUpper3}
\begin{split}
&\left|
m 
\frac{ \varphi'''_\eps(\theta m)}{ 2\sqrt{\varphi''_\eps(0)}\sqrt{\varphi''_\eps(\theta m)}}\right| 
~\leq~ \frac{1}{2}\, \rho  \, c \, c'
 ~= ~  \frac{1}{2}\, \sqrt{\frac{\log(N)}{N}} \,\frac{c\, c'}{\sqrt{|\bar{H}_\eps''(0)|}} \,   
~\leq ~ c\, c' \sqrt{\frac{\eps\, \log(N)}{N}} .
\end{split}
\end{align}
In combination with \eqref{EqEpsUpper1}, this yields \eqref{EqEpsUpper0}.

\textbf{Step 5.} [Proof of \eqref{EqEpsUpper}.]\\
Again by Taylor's formula, for some $ \theta'\in[0,1] $,
        \begin{align}
    \bar{H}_\eps(m) ~=~ \bar{H}_\eps(0)    ~+~ \frac{1}{2}m^2   \bar{H}_\eps''(0)  
    ~+~\frac{1}{6}m^3   \bar{H}_\eps'''(\theta' m) .
    \label{EqEpsUpper2}
    \end{align}
Similarly as in  Step 4, we have that 
        \begin{align}
\begin{split}
        \label{EqEpsUpper21}
\left| m^3   \bar{H}_\eps'''(\theta' m) \right|
~\leq ~\rho^3 \,
\left|    \varphi_\eps'''(\theta' m) \right|
~\leq~  \sqrt{\frac{\log(N)}{N}}^3 \,\frac{1}{\sqrt{|\bar{H}_\eps''(0)|}^3}\, \frac{c'}{\eps} 
 ~\leq ~ 8\,c'\, \sqrt{\eps} \,
\sqrt{\frac{ \log(N)}{N}}^3 ,
\end{split}
 \end{align}
 which concludes the proof of \eqref{EqEpsUpper}.
\end{proof}

\subsection{Lower bound on the capacity}  \label{SecEpsLowerB}
In this section, we prove the   lower bound on the capacity.
The proof is inspired by the two-scale approach, which was initiated in \cite{GORV}.
To apply this approach, we use that by the \emph{Bakry-\'Emery theorem} (see for instance \cite[p.\ 305]{GORV} and \cite[Remark 1.2]{MS}),   $ \mu_m$ satisfies the \emph{Poincar\'e inequality} (recall \eqref{EqPI} from the introduction) with  constant $ (J-1)/\eps $.
That is, for all $ N \in \N $, $ m \in \R $ and  $ f \in H^1(\mu_m) $,
\begin{align}\label{EqEpsPI}
\var_{\mu_m} \left( f \right) 
&~:=~ \int  \left|	f - \int   	f   d\mu_m  	\right|^2 d\mu_m  
~\leq~ \frac{\eps}{J-1}  \int  \left|	(\mathrm{id} - NP^tP) \nabla f	\right|^2 d\mu_m  ,
\end{align}
where $ P^tm= (1/N)  (m,\dots,m) \in \R^N $ for $ m\in \R $. 
\begin{proposition}\label{PropEpsLowerBound}
	Consider the same setting as in Theorem \ref{ThmIntroEpsEKF}.
	Then, for $ N\in \N  $ large enough and $ \eps\in (0,1] $ small enough,
	\begin{align}\label{EqLowerBound}
	\capa(B^-, B^+) \,\geq\,  \frac{\eps}{ 2\pi } \, 
	\ee^{-  N \bar{H}_\eps (0)}   \,  \sqrt{  |\bar{H}_\eps''(0)|  } \, \sqrt{  \varphi_\eps''(0)    } 
\, \left(1+O\left( \sqrt{\frac{\log(N)^3}{N}}\right) +O\left(\eps \right)\right) .
	\end{align}
\end{proposition}
 \begin{proof} 
%
%
 	Let $ f~=~f^*_{B^-, B^+}(x) $ (recall Definition \ref{DefPotTheory} and Lemma \ref{LemDirchletPrinc}) and, for $ m\in K:=[-2,2] $, let 
 	\begin{align}
 		\bar{f}(m)~=~\int_{P^{-1}(m)} f \,  d\mu_m. 
 	\end{align}
 	
 	As in \cite[Section 2.1]{GORV}, we split the gradient $ \nabla f $ into its fluctuation part 	$ (\mathrm{id} - NP^tP) \nabla f $ and its macroscopic part $  NP^tP \nabla f $. Note that 
 	\begin{align}\label{RoughLBEqEps19}
 		\left|	(\mathrm{id} - NP^tP) \nabla f	\right|^2 ~+~  \left|	  NP^tP \nabla f	\right|^2~=~  \left|	  \nabla f	\right|^2,
 	\end{align}
 	and by \cite[Lemma 21]{GORV},
 	\begin{align}\label{RoughLBEqEps199}
 \int_{P^{-1}(m)} 	 P \nabla f	 \,  d\mu_m 
=   \frac{\bar{f}'(m)}{N} + P\cov_{\mu_m} \left( f ,   \nabla  \HH\right) \qfa m\in K=[-2,2],
 \end{align}
 	where $ \HH $ is the microscopic Hamiltonian defined in \eqref{EqMicrHam}, and, for two functions $ g,h\in L^1(\mu_m) $,
 	\begin{align}
 	\cov_{\mu_m} \left( g,h\right)~=~\int g \left(	h - \int   	h d\mu_m  	\right) d\mu_m  .
 	\end{align}
 	
 	Using \eqref{EqDisintegration}, the fact that $ |NP^tPx|^2~=~N |Px|^2 $ for all $ x \in \R^N $, Jensen's inequality and   \eqref{RoughLBEqEps199}, we obtain that 
 	\begin{align}\label{RoughLBEqEps11}
 		\begin{split}	\int  \left|	 NP^tP \nabla f	\right|^2 \, d\mu  
 			&\geq  N\int_{-m^\star_{\eps}+\eta_{\eps}}^{m^\star_{\eps}-\eta_{\eps}} \int_{P^{-1}(m)}   \left|	 P \nabla f	\right|^2 \,  d\mu_m \, d\bar{\mu}(m) \\
 			&\geq N \int_{-m^\star_{\eps}+\eta_{\eps}}^{m^\star_{\eps}-\eta_{\eps}}  \left|\int_{P^{-1}(m)}  	 P \nabla f	 \,  d\mu_m\right|^2 \, d\bar{\mu}(m)\\
 			&= N \int_{-m^\star_{\eps}+\eta_{\eps}}^{m^\star_{\eps}-\eta_{\eps}}  \left| \frac{\bar{f}'(m)}{N} + P\cov_{\mu_m} \left( f ,   \nabla  \HH\right) \right|^2 \, d\bar{\mu}(m).
 		\end{split} 
 	\end{align}
 	Then, using  Young's inequality, we have that for all $ \theta \in [0,1] $, 
 	\begin{align}\label{RoughLBEqEps111}
 		\begin{split}
 			\int  \left|	 NP^tP \nabla f	\right|^2 \, d\mu  
 			&~\geq~ 
 			(1-\theta ) \frac{1}{N}   \int_{-m^\star_{\eps}+\eta_{\eps}}^{m^\star_{\eps}-\eta_{\eps}}  \left| \bar{f}'(m) \right|^2 \, d\bar{\mu}(m) 
 			\\ &\qquad	+~ \left(1 - \frac{1}{\theta }\right) N \int_{-m^\star_{\eps}+\eta_{\eps}}^{m^\star_{\eps}-\eta_{\eps}}   \left|  P\cov_{\mu_m} \left( f ,  \nabla  \HH\right) \right|^2 \, d\bar{\mu}(m).  
 		\end{split}
 	\end{align}
 	Later in this proof we show that 
 	\begin{align}\label{RoughLBEqEps112}
 		 	\frac{1}{N}   \int_{-m^\star_{\eps}+\eta_{\eps}}^{m^\star_{\eps}-\eta_{\eps}}  \left| \bar{f}'(m) \right|^2 \, d\bar{\mu}(m) 
 		\geq 
 		\frac{\ee^{-  N \bar{H}_{\eps} (0)}}{ 2\pi\,Z_{\mu} }  
 		\,  \sqrt{  |\bar{H}_{\eps}''(0)|  }\sqrt{  \varphi_{\eps}''(0)    } 
 		\left( 1 + O \left( \sqrt{\frac{ \log(N)^3}{N}} \right) \right), 
 		\end{align}
and that for some constant $ c>0 $, which is independent of $ \eps $ and $ N $, 
 	\begin{align}\label{RoughLBEqEps113} \begin{split}
 	\int_{-m^\star_{\eps}+\eta_{\eps}}^{m^\star_{\eps}-\eta_{\eps}}   \left|  P\cov_{\mu_m} \left( f ,   \nabla  \HH\right) \right|^2 \, d\bar{\mu}(m)
 		&~\leq~ \frac{c }{N}
 		\left(  \eps  +\frac{1}{\sqrt{N}}\right)  \int \left|(\mathrm{id} - NP^tP)\nabla f\right|^2 \, d\mu \\
 		&\qquad \times \left( 1  + O(\eps)+ O \left(\frac{ 1}{ \sqrt{N}} \right) \right).
 	\end{split}
 	\end{align}
 	By combining  \eqref{RoughLBEqEps111}, \eqref{RoughLBEqEps112} and \eqref{RoughLBEqEps113} and by choosing 
 	\begin{align}
 	\theta~=~
 	\frac{ 	c  \left(  \eps  +\frac{1}{\sqrt{N}}\right)}{1+ 	c  \left(  \eps  +\frac{1}{\sqrt{N}}\right)},
 	\end{align}
 	we infer that 
 	\begin{align}
\nonumber
 	\int  \left|	 NP^tP \nabla f	\right|^2 \, d\mu  
 	\,\geq\,&
 	\frac{ 	1}{1+ 	c  \left(  \eps  +\frac{1}{\sqrt{N}}\right)} \, \frac{\ee^{-  N \bar{H}_{\eps} (0)}}{ 2\pi\,Z_{\mu} }  
 	\,  \sqrt{  |\bar{H}_{\eps}''(0)|  }\sqrt{  \varphi_{\eps}''(0)    } 
 	\left( 1 + O \left( \sqrt{\frac{ \log(N)^3}{N}} \right) \right)
 	\\ & -  ~ \int \left|(\mathrm{id} - NP^tP)\nabla f\right|^2 \, d\mu  \, \left( 1  + O(\eps)+ O \left(\frac{ 1}{ \sqrt{N}} \right) \right).  
 	\end{align}
 	Together with \eqref{RoughLBEqEps19} this yields that 
 	 	\begin{align}
 	\nonumber
 	\int  \left|	  \nabla f	\right|^2 \, d\mu  
 	~\geq ~  \frac{\ee^{-  N \bar{H}_{\eps} (0)}}{ 2\pi\,Z_{\mu} }  
 	\,  \sqrt{  |\bar{H}_{\eps}''(0)|  }\sqrt{  \varphi_{\eps}''(0)    }  \, 
 	\left( 1 + O(\eps)+ O \left( \sqrt{\frac{ \log(N)^3}{N}} \right) \right).
 	\label{RoughLBEqEps1112}
 	\end{align}
 	Applying now Lemma \ref{LemDirchletPrinc} 
 	implies \eqref{EqLowerBound}. It only remains to show \eqref{RoughLBEqEps112} and \eqref{RoughLBEqEps113}.
 	
 	\emph{Proof of \eqref{RoughLBEqEps112}.}
 	Note that by Proposition \ref{PropEpsLocalCramer},
 	\begin{align}
 		\begin{split}
 			&	\frac{1}{N}   \int_{-m^\star_{\eps}+\eta_{\eps}}^{m^\star_{\eps}-\eta_{\eps}}  \left| \bar{f}'(m) \right|^2  \, d\bar{\mu}(m) \\  
 			&~	= 	\frac{1}{N Z_{\bar{\mu}}}  \int_{-m^\star_{\eps}+\eta_{\eps}}^{m^\star_{\eps}-\eta_{\eps}} |\bar{f}'|^2 \, \frac{\sqrt{\varphi_{\eps}''(m)}}{\sqrt{2\pi}}\, \ee^{-N \bar{H}_{\eps}(m)} \, dm\, 
 			\left( 1 + O \left( \sqrt{\frac{ \log(N)^3}{N}} \right) \right).
 		\end{split}
 	\end{align}
 	Then, by the fact that $ 1 = \bar{f} ( -m^\star_{\eps}+\eta_{\eps} )   =1 - \bar{f} ( m^\star_{\eps}-\eta_{\eps} )     $ and by our knowledge on one-dimensional capacities  (see for instance \cite[Section 7.2.5]{BdH15}), 
 	\begin{align}
 		\begin{split}
 			\int_{-m^\star_{\eps}+\eta_{\eps}}^{m^\star_{\eps}-\eta_{\eps}} &|\bar{f}'|^2 \, \frac{\sqrt{\varphi_{\eps}''(m)}}{\sqrt{2\pi}}\, \ee^{-N \bar{H}_{\eps}(m)} \, dm\\
 			&~\geq 
 			\inf_{\substack{ h : \\
 					h ( -m^\star_{\eps}+\eta_{\eps} )   =1,\\ \  h  ( m^\star_{\eps}-\eta_{\eps} )   =0}}
 			\int_{-m^\star_{\eps}+\eta_{\eps}}^{m^\star_{\eps}-\eta_{\eps}} |h'|^2 \, \, \frac{\sqrt{\varphi_{\eps}''(m)}}{\sqrt{2\pi}}\,\, \ee^{-N \bar{H}_{\eps}(m)} \, dm \\
 			&~	=  \frac{1}{\sqrt{2\pi} }\left( \int_{-m^\star_{\eps}+\eta_{\eps}}^{m^\star_{\eps}-\eta_{\eps}} \,  \sqrt{\varphi_{\eps}''(r)}^{-1} \, \ee^{N\bar{H}_{\eps}(r)}\, dr\right) ^{-1}.
 		\end{split}
 	\end{align}
 	Recalling that $  \max_{m \in [-m^\star_{\eps}+\eta_{\eps},m^\star_{\eps}-\eta_{\eps}] }\bar{H}_{\eps}(m) =\bar{H}_{\eps}(0) $ and $ \sqrt{N} Z_{\bar{\mu}} =  Z_{\mu} $ by the coarea formula, we conclude \eqref{RoughLBEqEps112} from  standard Laplace asymptotics.
 	
 	\emph{Proof of \eqref{RoughLBEqEps113}.}
 	Since  $ \mu_m $ is supported on $ P^{-1}(m) $, we have that
 	\begin{align}
 		P\cov_{\mu_m} \left( f ,  \nabla  \HH\right)~=~\ieps\frac{1}{N}\, \cov_{\mu_m} \left( f , \sum_{i=0}^{N-1} x_i^3 \right),
 	\end{align}
 	Then, using H\"older's inequality and \eqref{EqEpsPI},
 	\begin{align}\label{EqEps0GradFBar}
 		\begin{split}
 			&  	\left|     P\cov_{\mu_m} \left( f ,   \nabla  \HH\right) \right|^2    
 			~\leq~  \frac{1}{\eps^2}   \frac{1}{N^2}   \var_{\mu_m} \left( f \right)   \  \var_{\mu_m} \left(    \sum_{i=0}^{N-1}x_i^3 \right)  \\
 			& \qquad ~\leq~ \frac{1}{(J-1)^2N^2}\   \int  \left|	(\mathrm{id} - NP^tP) \nabla f	\right|^2 d\mu_m  \
 			\int \left|(\mathrm{id} - NP^tP)\nabla \sum_{i=0}^{N-1}x_i^3\right|^2 d\mu_m     .
 		\end{split}
 	\end{align} 
 	It remains to show that the second integral on the right-hand side of \eqref{EqEps0GradFBar} is bounded from above by $ c'(\eps N +\sqrt{N})$ for some constant $ c'>0 $ which is independent of $ \eps $ and $ N $. 
 	
 	First we observe that by symmetry, 
 	 	\begin{align}\label{EqEpsGradH1}
 	\begin{split}
 	\int &\left|(\mathrm{id} - NP^tP)\nabla \sum_{i=0}^{N-1}x_i^3\right|^2 d\mu_m 
 	~=~ 
 	\int \sum_{i=0}^{N-1}\left|  3x_i^2 - \frac{1}{N} \sum_{j=0}^{N-1}  3 x_j^2 \right|^2 d\mu_m \\
 	&~=~ 9N
 	\int  \left|  x_0^2 - \frac{1}{N} \sum_{j=0}^{N-1}   x_j^2 \right|^2 d\mu_m \\
 	&~=~ 9N  
 	\int  x_0^4 \
 	 d\mu_m
 	 -18\sum_{j=0}^{N-1}
 	 	\int x_0^2    x_j^2
 	 d\mu_m
 	 +\frac{9}{N}\sum_{l=0}^{N-1}
 	 	\int   \sum_{j=0}^{N-1}  x_l^2x_j^2
 	 d\mu_m 	\\ 
 	 &~=~ 9(N-1)  
 	 \int  x_0^4 \
 	 d\mu_m 
 	 -9(N-1)
 	 \int x_0^2    x_1^2 \ 
 	 d\mu_m .
 	\end{split}
 	\end{align}  
Then, applying Proposition  \ref{PropEpsOfObserv},  the right-hand side of \eqref{EqEpsGradH1} is lower or equal to
 	\begin{align}\label{EqEpsLBProof}
 			9(N-1)   \int \left|  z^2 - \int  z^2 \  d\mu^{\eps,\varphi_{\eps}'(m)}  \right|^2 d\mu^{\eps,\varphi_{\eps}'(m)}\  ~+~ O_K \left(  \sqrt{N} \right) .
 	\end{align}  
 	It remains to show that 
 	 	\begin{align}\label{EqEpsLBProof1}
    \int \left|  z^2 - \int  z^2 \  d\mu^{\eps,\varphi_{\eps}'(m)}  \right|^2 d\mu^{\eps,\varphi_{\eps}'(m)}\  ~=~ O_K \left(  \eps \right) .
 	\end{align}  
 	In order to show \eqref{EqEpsLBProof1}, we again apply the Bakry-\'Emery theorem (see e.g.\ \cite[p.\ 305]{GORV} and \cite[Remark 1.2]{MS}) to observe that the  measure $ \mu^{\eps,\varphi_{\eps}'(m)} $ satisfies the Poincar\'e inequality (see \eqref{EqPI}) with constant $ (J-1)/\eps $. 
 	Hence, 
 	 	 	\begin{align}\label{EqEpsLBProof2}
 	\int \left|  z^2 - \int  z^2 \  d\mu^{\eps,\varphi_{\eps}'(m)}  \right|^2 d\mu^{\eps,\varphi_{\eps}'(m)}\  ~=~ 
 	\var_{\mu^{\eps,\varphi_{\eps}'(m)}} \left( z^2\right) 
 	~\leq~
 	 \frac{4\eps}{J-1}
\int   z^2   d\mu^{\eps,\varphi_{\eps}'(m)} .
 	\end{align}  
Then, using Lemma \ref{LemLegendre} and Lemma \ref{LemMoments} we obtain that for all $ m \in K=[-2,2] $,
 	 	 	\begin{align}\label{EqEpsLBProof22}
\begin{split}  
\int   z^2   d\mu^{\eps,\varphi_{\eps}'(m)} &~=~ \int   \left( z -\int    \bar{z}\, d\mu^{\eps,\varphi_{\eps}'(m)}(\bar{z}) \right) ^2   d\mu^{\eps,\varphi_{\eps}'(m)} + \left( \int    \bar{z}\, d\mu^{\eps,\varphi_{\eps}'(m)}(\bar{z} )\right) ^2 \\ 
&~\leq~ \hat{c}\,  \eps   ~+~m^2 ~\leq~ \hat{c} \, \eps ~+~4
\end{split}
\end{align}  
for some  constant $ \hat{c}>0 $ which is independent of $ \eps $ and $ N $. 
Combining \eqref{EqEpsLBProof2} and \eqref{EqEpsLBProof22} yields \eqref{EqEpsLBProof1}.
 	This concludes the proof of \eqref{RoughLBEqEps113}.  
 \end{proof}
 
\begin{remark}\label{RemarkJAss}
\emph{
	The proof of \eqref{RoughLBEqEps113} is the main reason for the assumption  \eqref{EqJAss}.
Indeed, in this step, we use that, under \eqref{EqJAss},  the (effective) single-site potential $ \psi_J $ defined in \eqref{EqSingleSite} is strictly convex so that we can apply the Bakry-\'Emery theorem, which in turn yields that we have a good control on the covariance term $ P\cov_{\mu_m} \left( f ,  \nabla  \HH\right) $ in \eqref{RoughLBEqEps113} for small $ \eps $.
Note that, intuitively, the quantity $ P\cov_{\mu_m} \left( f ,  \nabla  \HH\right) $ describes the microscopic fluctuation of the system around the hyperplane $ P^{-1}(m) $. 
 }
\end{remark}
 
 In \eqref{EqEpsLBProof} we    use  that we can pass from expectations with respect to   $ \mu_m  $ to expectations with respect  to $ \otimes_{i=1}^N \mu^{\eps,\varphi_{\eps}'(m)}  $. 
 Such a statement is known in the literature as the \emph{equivalence of observables} (see \cite{KwMe18a}).
 The result in our setting is formulated in the following proposition. The proof is postponed to the appendix.  

\begin{proposition}[Equivalence of observables]\label{PropEpsOfObserv}
 	Let $ K \subset \R $ be compact. 
Let $ \ell \in \N $, and let  $ b:\R^\ell \ra \R $ be such that 
\begin{align}\label{EqEquOfObserv0}
\sup_{m \in K} \int_{\R^\ell  } |b (z_{0},\dots, z_\ell  )|^2 \, d\mu^{\eps,\varphi_{\eps}'(m),\ell}(z_{\eps},\dots, z_\ell  ) ~<~ \infty ,
\end{align} 
where  $\mu^{\eps,\varphi_{\eps}'(m),\ell}~=~\otimes_{i=1}^\ell \mu^{\eps,\varphi_{\eps}'(m)}  $.
Then, there exist $ C_{b,K,\ell}, \eps_{b,K,\ell} >0  $,  $ N_{b,K,\ell} \in \N  $ such that for all   $ N \geq N_{b,K,\ell}  $, 
\begin{align}\label{EqEquOfObserv}
\sup_{0<\eps< \eps_{b,K,\ell}}	\sup_{m \in K}~		\left|  \int_{P^{-1}(m)}  b(x_{0},\dots, x_\ell  ) \, d\mu_m  -\int_{\R^\ell  } b (z_{0},\dots, z_\ell  ) \, d\mu^{\eps,\varphi_{\eps}'(m),\ell} \right|   
~  \leq~    C_{b,K,\ell} \frac{1}{\sqrt{N}} .
\end{align}
\end{proposition} 
\begin{proof}
The proof is postponed to Section \ref{SecProofEquOfObserv}.
\end{proof}

 \subsection{The mass of the equilibrium potential}\label{SecEpsNum}
 \begin{proposition}\label{PropEpsNum}
 	Consider the same setting as in Theorem \ref{ThmIntroEpsEKF}.
 	Then, for $ \eps\in (0,1] $ small enough,
 	\begin{align}\label{EqEpsNumerator}
 	\int_{(B^+)^c}f^*_{B^-,B^+}(y)\, \ee^{-\mathrm{H}(y)}\, dy ~=~
 	\frac{	\ee^{-N\bar{H}_\eps(-m^\star_\eps )}}{\sqrt{\bar{H}_\eps''(-m^\star_\eps ) }} \,\sqrt{\varphi_\eps''(-m^\star_\eps) }\
 	\left( 1+O\left(\sqrt{\frac{\log(N)^3}{N}}\right)\right).
 	\end{align}
 \end{proposition}
 \begin{proof}     
In this proof,    $ C $   denotes a varying positive constant, which is independent of  $ \eps  $ and $N $. 

	\textbf{Step 1.} [Splitting into four regions.]\\
	Recall the definition of $ \eta  $ in \eqref{EqEpsEta}.
Let $ R>2 $ be a positive number, which is independent of $ N   $, $ \eps $ and  $ m \in K$, and whose precise value is chosen later in Step 3.
Using that $ f^*_{B^-,B^+}(y)=1 $ for $ y\in B^- $, we split the left-hand side of \eqref{EqEpsNumerator} according to  this $ R $ in the following way.
\begin{align} 
\begin{split}
&\int_{(B^+)^c} f^*_{B^-,B^+}(y)\,\ee^{-\mathrm{H}(y)}\, dy\\  
&~=~
\int_{ \{  P \,\in \, [-m^\star_\eps -\eta,-m^\star_\eps +\eta]	\}}  \ee^{-\mathrm{H}(y)}\, dy~+~
\int_{ \{    P \,\in \,[-m^\star_\eps +\eta,m^\star_\eps -\eta) 	\}} f^*_{B^-,B^+}(y)\,\ee^{-\mathrm{H}(y)}\, dy\\
&\qquad +~ 
\int_{ \{  P \,\in \,[-R,-m^\star_\eps -\eta]	  	\}}  \ee^{-\mathrm{H}(y)}\, dy
 ~+~
\int_{ \{x \in \R^N \, | \, Px\, <\, -R 		\}} \ee^{-\mathrm{H}(y)}\, dy \\
&~=:~ I~+~II~+~III~+~IV.
\end{split}
\label{EqEpsNum11}
\end{align}
In Step 2 we compute the asymptotic value of the term $ I $, and in Step 3 and 4 we show that the terms $ II, III $ and $ IV$ are of lower order than $ I $.  

\textbf{Step 2.} [Estimation of the term $ I $.]\\
Note that, using the same arguments as in Step 4 and Step 5 of the proof of Proposition \ref{PropEpsUpperBound}, for all $ m \in [-m^\star_\eps-\eta,-m^\star_\eps+\eta] $, 
\begin{align}
\begin{split}
\label{EqEpNum0}
\sqrt{\varphi''_\eps(m)} 
&~=~  \sqrt{\varphi''_\eps(-m^\star_\eps)} \, \left( 1+O_K\left(\sqrt{\frac{\eps\log(N\eps^{-1})}{N}}\right)\right), \quad \text{ and }\\
\bar{H}_\eps(m) &~=~ \bar{H}_\eps(-m^\star_\eps)    + \frac{1}{2}(m+m^\star_\eps)^2   \bar{H}_\eps''(-m^\star_\eps)  
+O_K\left( \sqrt{\eps}\sqrt{\frac{ \log(N\eps^{-1})}{N}}^3\right).
\end{split}
\end{align}  
Then, using the coarea formula in the same way that we did in Subsection \ref{SubSecEpsMuBar} and applying Proposition \ref{PropEpsLocalCramer} for the compact set $ [-R,R] $, we observe that 
\begin{align}
\begin{split}
I&~=~
\sqrt{N}\,
\int_{-m^\star_\eps-\eta}^{-m^\star_\eps+\eta}  \ee^{ - N \varphi_{N,\eps}(m) + \ieps N \frac{J}{2}  m^2 }\, dm\\
&~=~ \sqrt{N}\,
\int_{-m^\star_\eps-\eta}^{-m^\star_\eps+\eta} \ee^{   - N \bar{H}_\eps (m)   } \, \frac{\sqrt{\varphi_\eps''(m)}}{\sqrt{2\pi}} \, dm~  \left( 1+O\left(\frac{1}{\sqrt{N}}\right)\right)  .
\end{split}
\end{align}
Using  \eqref{EqEpNum0} and arguing as in the proof of Proposition \ref{PropEpsUpperBound}, we have that for $ \eps $ small enough,
\begin{align}
\nonumber
I&~=~\frac{\sqrt{N}}{\sqrt{2\pi}} \, 
\ee^{-N\bar{H}_\eps(-m^\star_\eps )}\,\sqrt{\varphi_\eps''(-m^\star_\eps) }\, 
\int_{-\eta}^{\eta}\, \ee^{-N\bar{H}_\eps''(-m^\star_\eps )\frac{m^2}{2} }
\, dm\, \left( 1+O\left(\sqrt{\frac{ \log(N)^3}{N}}\right)\right)\\ \label{EqEpsNum3}
&~=~  \frac{	\ee^{-N\bar{H}_\eps(-m^\star_\eps )}}{\sqrt{\bar{H}_\eps''(-m^\star_\eps ) }} \,\sqrt{\varphi_\eps''(-m^\star_\eps) }\
\left(  1+O\left(\sqrt{\frac{ \log(N)^3}{N}}\right)\right).
\end{align}

\textbf{Step 2.} [Estimation of the terms $ II $ and $ III $.]\\
We only consider the   term   $ II $.   The   term $ III $   can be estimated in the same way. 
By using that $ |f^*_{B^-,B^+}| \leq 1 $ and by applying the coarea formula and Proposition \ref{PropEpsLocalCramer} as in Step 1, we have that 
\begin{align}\label{EqEpsNum12}
\begin{split}
|II|&~\leq ~ 
C\, \sqrt{N}\, 
\int_{-m^\star_\eps+\eta}^{m^\star_\eps-\eta} \ee^{   - N \bar{H}_\eps (m)   } \,  \sqrt{\varphi_\eps''(m)} \, dm  .
\end{split}
\end{align}
Note that, similarly  as in Step 3 of the proof of Proposition \ref{PropEpsUpperBound}, we have that $   |\bar{H}_\eps''(-m^\star_\eps )| = \Omega(1/\eps)$. 
Together with \eqref{EqEpsUpper4}, this shows that $ I = \Omega(\ee^{-N\bar{H}_\eps(-m^\star_\eps )}) $. 
In the following we prove that $ II \leq O(\ee^{-N\bar{H}_\eps(-m^\star_\eps )}\sqrt{N}^{-1}) $, which shows that $ II $ is of lower order than $ I$.
Since 
$  \bar{H}_\eps  $ is symmetric and has its two global minima at $ \pm m^\star_\eps $, 
we have that  
\begin{align}
\inf_{m \in [-m^\star_\eps+\eta,m^\star_\eps-\eta]}\bar{H}_\eps(m) ~=~ \bar{H}_\eps(-m^\star_\eps+\eta)  . 
\end{align} 
Then, by  \eqref{EqEpsUpper4},  \eqref{EqEpNum0} and the definition of $ \eta $ (see \eqref{EqEpsEta}), 
\begin{align}
\label{EqEpsNum33}
&|II| \leq
\frac{C\sqrt{N}}{\sqrt{\eps}}\, \ee^{-N\bar{H}_\eps(-m^\star_\eps+\eta)} 
\leq
\frac{C\sqrt{N}}{\sqrt{\eps}}\,\ee^{-N(\bar{H}_\eps(-m^\star_\eps )+\bar{H}_\eps''(-m^\star_\eps )\frac{\eta^2}{2}) } 
=
\frac{C\sqrt{\eps}}{\sqrt{N}}\, \ee^{-N\bar{H}_\eps(-m^\star_\eps )}.
\end{align}	 

\textbf{Step 3.} [Estimation of the term $ IV $.]\\
Using Jensen's inequality, we have that $ \sum_{i=0}^{N-1} x_i^4 ~\geq~ N (Px)^4 $. Then, via the coarea formula, 
\begin{align}
\begin{split}
|IV|&
 ~ \leq ~
\int_{\{x \in \R^N \, | \, Px\, <\,- R 			\}}  \ee ^{ -  \ieps\sum_{i=0}^{N-1}\frac{J-1}{2}y_i^2   }\  \ee ^{ -\ieps N\frac{1}{4} (Py)^4 +  \ieps N\frac{J}{2} (Py)^2   } \ dy  \\
&~ = ~\, \sqrt{N}\, \int_{-\infty}^{-R} \ee ^{ -\ieps N\frac{1}{4} m^4 +  \ieps N\frac{J}{2} m^2   } \ \int_{P^{-1}(m)}   \ee ^{ -  \ieps\sum_{i=0}^{N-1}\frac{J-1}{2}y_i^2   }\  d\cH^{N-1}\, dm.
\end{split}
\end{align}
In Lemma \ref{LemAppLocalCramerQuadr}, we show that for all $ m\in \R $,
\begin{align}
\int_{P^{-1}(m)}   \ee ^{ -  \ieps\sum_{i=0}^{N-1}\frac{J-1}{2}y_i^2   }\  d\cH^{N-1} ~=~
\ee^{-N \ieps \frac{J-1}{2} m^2 + N \frac{1}{2} \log (2\pi \,\eps\, (J-1)^{-1})} ~ \sqrt{\frac{J-1}{\eps\, 2\pi}}.
\end{align}
Therefore, by \cite[Lemma 1.1]{BovBBM}, we have that for $ \eps $ small enough, 
\begin{align}
\nonumber
| IV| 
&\, \leq\,
\sqrt{N}\, \int_{-\infty}^{-R} \ee ^{ -\ieps N\frac{1}{4} m^4 +  \ieps N\frac{1}{2} m^2   } \  dm \ \sqrt{\frac{J-1}{\eps\, 2\pi}} 
\, \leq \,
\sqrt{N}\, \int_{-\infty}^{-R} \ee ^{ -\ieps N\frac{1}{2} (\frac{R^2}{2} - 1 ) m^2   } \  dm \ \sqrt{\frac{J-1}{\eps\, 2\pi}}\\
&~ = ~ C \int_{-\infty}^{-R\sqrt{\frac{N}{\eps} (\frac{R^2}{2} - 1 )}} \ee ^{ - \frac{1}{2}   m^2   } \, dm  
\,\leq \,  C\  \sqrt{\frac{\eps}{N}}\  \ee ^{ - \frac{1}{2}   \frac{N}{\eps} (\frac{R^2}{2} - 1 )R^2   } .
\label{EqEpsNum4}
\end{align} 
Note that $ \bar{H}_\eps(-m^\star_\eps ) \leq \frac{c}{\eps} $ for some $ c >0 $.
Indeed,   by Lemma \ref{LemMoments} \emph{(ii)}, we have that for some bounded function $ \tau $, 
\begin{align}\label{EqEpsNum50}
\begin{split}
|\varphi_\eps(-m^\star_\eps )| &~=~ \left|-\int_{-m^\star_\eps}^{0} \varphi_\eps'(m)\, dm ~+~ \varphi_\eps(0) \right| 
~=~ \left|-\ieps\int_{-m^\star_\eps}^{0}\tau(m, \eps)\, dm ~+~ \varphi_\eps(0) \right| \\
&~\leq~ \ieps \, \|\tau(\cdot, \eps)\|_{\mathrm{L}^\infty(K\,;\,dm)} \,m^\star_\eps ~+~ | \varphi_\eps(0)  |,
\end{split}
\end{align}
and by  \eqref{EqLegendre} and \eqref{EqMomentLaplace2k}, 
\begin{align}\label{EqEpsNum51}
\varphi_\eps(0)~=~- \varphi_\eps^*(0)
~=~\log \int_{\R} \ee^{  -  \ieps  \psi_J(z) }\,  dz
~\leq ~\frac{1}{2}\log \left( C \eps \right) .
\end{align}
Combining \eqref{EqEpsNum50} and \eqref{EqEpsNum51} with the definition of $ \bar{H}_\eps $, shows that $ \bar{H}_\eps(-m^\star_\eps ) \leq \frac{c}{\eps} $ for some $ c >0 $.
Then,  choosing  $ R $ large enough,  \eqref{EqEpsNum4} implies that
\begin{align}\label{EqEpsNum5}
\begin{split}
IV&~ = ~ O \left( \frac{1}{\sqrt{N}} \frac{	\ee^{-N\bar{H}_\eps(-m^\star_\eps )}}{\sqrt{\bar{H}_\eps''(-m^\star_\eps ) }} \,\sqrt{\varphi_\eps''(-m^\star_\eps) }\right) .
\end{split}
\end{align}
This shows that the term $ IV $ is of lower order than $ I $, and concludes the proof. 
\end{proof}

 \section{Rough estimates at high temperature}\label{ChapRoughEstimates}
 In this chapter, we consider the same system as in Chapter \ref{ChapSmallNoise}, but with two key differences. First, we do not consider the low-temperature regime here, that is, throughout this chapter, we suppose that $ \eps =1 $. 
 The second difference is that, instead of $ \psi(z) = z^4/4-z^2/2 $, we consider here a   class of single-site potentials   given by functions of the form $z\mapsto  \Psi(z) -\frac{J}{2}z^2 $, where $ \Psi: \mathbb{R} \to \mathbb{R} $ satisfies Assumption \ref{Assumptions}.
 
 Hence, the microscopic Hamiltonian~$\HH^{N,1}: \mathbb{R}^N \to \mathbb{R}$ in this chapter is given by
\begin{align}
\HH^{N,1}(x)~=~\sum_{i=0}^{N-1}\left( \Psi(x_i) -\frac{J}{2}x_i^2\right)  +   \frac{J}{4N } \sum_{i,j=0}^{N-1} (x_i- x_j)^2 
~=~\sum_{i=0}^{N-1}\Psi(x_i)  -   \frac{J}{2N } \sum_{i,j=0}^{N-1}  x_i x_j  ,
\end{align}
where $ J>0 $. 
 We make the following assumptions on 
the single-site potential~$\Psi $. 
\begin{assumption}\label{Assumptions}
	\begin{enumerate}[(1)]
		\item There is a splitting $\Psi= \Psi_c + \Psi_b$ for some $ \Psi_c , \Psi_b \in C^2(\R) $, and there are constants~$ 0<c ,c' < \infty$ such that $ 		\Psi_c''(z) \geq c     \mbox{ and }   |\Psi_b|_{C^2} \leq c'.  $
		\item $\Psi(z)=\Psi(-z)$ for all~$z \in \mathbb{R}$. 
		\item $ z \mapsto \Psi'(z) $ is  convex on $ [0,\infty) $.
		\item  If $ \Psi_c $ is a quadratic function of the form $ 	\Psi_c(x)~=~c_\Psi x^2 + c_\Psi'x + c_\Psi'' $
		for some $ c_\Psi , c_\Psi',c_\Psi''  \in \R$, then we suppose that $ c_\Psi > J $. 
		\item $ 1/J < \int_{\R} z^2 \, \ee^{-\Psi(z)  } \, dz  / (\int \ee^{-\Psi(z)  } \, dz)   $.  
		\item $ \sigma \mapsto \int_{\R} (\Psi''(z))^2 \, \ee^{-\Psi(z) + \sigma z  } \, dz    / (\int \ee^{-\Psi(z)+ \sigma z } \, dz)    $ is locally bounded on $ \R $.  
	\end{enumerate}
\end{assumption}
\begin{remark}
	\emph{If $ \Psi= \Psi_c $ is a quadratic function, then Assumption \ref{Assumptions} is not fulfilled for any choice of $ J $.
		However, we do not expect that  Kramers' law holds true in this case, since  the macroscopic Hamiltonian $ \bar{H}_1 $ is not of double-well form, where $ \bar{H}_1 $ is defined as in \eqref{EqEpsHBar} with $ \psi $ being replaced by the function $z\mapsto  \Psi(z) -\frac{J}{2}z^2 $ and with $ \eps = 1$.
Indeed, from \eqref{EqAppLocalCramerQuadr1}, we see that $ \bar{H}_1 $ is a quadratic function and hence not of double-well form. }
\end{remark}

This chapter is organized similarly as Chapter \ref{ChapSmallNoise}. 
That is, in Section \ref{SecPrelim} we introduce the local Cram\'er theorem and show that the macroscopic Hamiltonian  has a double-well structure. 
In Section \ref{SecEyring} we provide the proof of Theorem \ref{ThmIntroEKF}. 
The main part here  is the proof of the lower bound on the capacity.
This is done in Section \ref{SecLowerB}.

\subsection{Preliminaries}\label{SecPrelim}
  \subsubsection{Local Cram\'er Theorem.}\label{SubSecMuBar}
  Replacing $ \psi $ by the function $z\mapsto  \Psi(z) -\frac{J}{2}z^2 $ and setting $ \eps = 1$, we 
  define the Gibbs measure $\mu^{N,1}$ by \eqref{EqGibbs}, and   introduce a disintegration of $\mu^{N,1}$ as  $\mu^{N,1}(dx)\ = \ \mu^{N,1}_m(dx)  \bar{\mu}^{N,1}(dm) $
  as in  and Subsection \ref{SubSecEpsMuBar}. 
   Analogously, we define the quantities $ \varphi_{N,1} $, $ \varphi^*_1 $, $  \varphi_1$ and  $ \mu^{1,\sigma} $  by \eqref{EqEpsMuBar}, \eqref{EqEpsMuBarStar}, \eqref{EqEpsMuBarCramer} and \eqref{EqEpsGCMeasure}, respectively,
 by replacing $ \psi $ by the function $z\mapsto  \Psi(z) -\frac{J}{2}z^2 $ and setting $ \eps = 1$.
 Then, the local Cram\'er theorem in this chapter is given as follows. 
 \begin{proposition}[Local Cram\'er theorem] \label{PropLocalCramer}
 	Suppose Assumption \ref{Assumptions}.
 	Then, for $ N $ large enough,
 	\begin{align}\label{EqLocalCramer}
 	\ee^{    -N\varphi_{N,1} (m)   } 
 	\ = \ \ee^{   - N \varphi_1 (m)   }~ \frac{\sqrt{\varphi_1''(m)}}{\sqrt{2\pi}} ~\left( 1+O\left(\frac{1}{\sqrt{N}}\right)  
 	\right).
 	\end{align} 
 	In particular,  
 	\begin{align}\label{EqLocalCramerMu}
 	d \bar{\mu}^{N,1}(m) \ = \  \frac{1}{Z_{\bar{\mu}^{N,1}}} \  \ee^{   - N \bar{H}_1 (m)   } ~\frac{\sqrt{\varphi_1''(m)}}{\sqrt{2\pi}} ~\left( 1+O\left(\frac{1}{\sqrt{N}}\right) 
 	\right)  \, dm.
 	\end{align}  
 \end{proposition}
 \begin{proof}
 	Using the same notation and proceeding as in the proof of Proposition \ref{PropEpsLocalCramer}, we observe that it suffices to show that
 	\begin{align}\label{e_local_large_deviations3}
 	\left| g_{N,m}(0) -  \frac{1}{\sqrt{2\pi}} \right| 
 	= O\left(\frac{1}{\sqrt{N}}\right).
 	\end{align}
 	However, this was already shown in \cite[Proposition 3.1 and Lemma 3.2]{MeOt}.
 \end{proof}

\subsubsection{Analysis of the energy landscape.}
In the following lemma we show that the macroscopic Hamiltonian  $ \bar{H}_1  $ has the form of a double-well function with at least quadratic growth at infinity.

\begin{lemma}\label{LemEnergyLandscape}
	Suppose  Assumption \ref{Assumptions}. 
	If $ \Psi_c $ is a quadratic function, then let $ c_\Psi $ denote the leading order coefficient. 
	Otherwise, let $ c_\Psi = \infty $. Then, we have that
	\begin{enumerate}[(i)]
		\item   $ \liminf_{|t| \ra
			 \infty } \frac{\varphi_1(t)}{t^2}~ \geq~  c_\Psi   $, \ 
		 $ \liminf_{|t| \ra
			 \infty } \frac{\bar{H}_1(t)}{t^2} ~\geq~  c_\Psi -J/2 $,
		\item  there exists $ K_J >0 $ and $ \delta >0 $ such that   $  \varphi_1'(t) \geq  (J+\delta)t  $ for all $ t \geq K_J $ and    $  \varphi_1'(t) \leq  (-J - \delta)t  $ for all $ t \leq -K_J $, and
		\item $\bar{H}_1$ has exactly three critical points located at~$- m^\star_1 $,~$0$ and~$ m^\star_1 $ for some 
		$ m^\star_1 >0$.
		Moreover,  $ \bar{H}_1''(0) < 0 $, $ \bar{H}_1''(m^\star_1 ) > 0 $ and $ \bar{H}_1''(- m^\star_1 ) >0 $.
		That is,  $\bar{H}_1$ has a local maximum at~$0$, and   the two global minima of $\bar{H}_1$ are located  at~$\pm m^\star_1 $.
%
%
	\end{enumerate}
\end{lemma}
\begin{proof}
		Since $ \varphi_1(t)= \varphi_1(-t)$ for all $ t\in \R $, it suffices to prove all   claims only on $ [0,\infty) $.
		
	 \emph{(i).}
	As in Lemma \ref{LemEpsEnergyLandscape}, this statement follows from a simple argument given in    \cite[Lemma III.2.6]{PatThesis}. 
	
	\emph{(ii).}
	From part \emph{(i)} and Assumption \ref{Assumptions} \emph{(4)}, we know that there exist $ K' >0 $ and $ \delta' >0 $ such that
	$ \varphi_1(t) \geq (J+\delta') t^2 $  for all $ t \geq K' $. 
	Using that $ t \mapsto \varphi_1'(t) $ is increasing (since $ \varphi_1 $ is strictly convex) we obtain that for all $ t \geq K' $,
	\begin{align}\label{LemVerifyingAssumptionsOfMOEq0}
(J+\delta') \,t^2 ~\leq~ \varphi_1(t) ~=~ \int_0^t \varphi_1'(r) \, dr + \varphi_1(0)
	~\leq~ \varphi_1'(t)t + \varphi_1(0),
	\end{align}  
	which concludes the claim.
	
	\emph{(iii).}
	Before we show the claims, note that the function	 $ z \mapsto \varphi_1'(z) $ is  convex on $ [0,\infty) $. 
	Indeed, from \cite[Theorem 1.2 c)]{MR0395659}, we know that Assumption  \ref{Assumptions} yields that 
 $ z \mapsto (\varphi_1^*)'(z) $ is   concave on $ [0,\infty) $ (cf.\ \cite[Remark IV.0.4]{PatThesis}). 
 Hence, for $ w>z $, we have that $ (\varphi_1^*)''(\varphi_1'(w)) \leq  (\varphi_1^*)''(\varphi_1'(z)) $, since, due to the  convexity of $ \varphi_1 $, we have that $ \varphi_1'(w) \geq \varphi_1'(z) $.
 Therefore, 
 \begin{align}
 \varphi_1''(w ) ~=~ \frac{1}{(\varphi_1^*)''( \varphi_1'(w))} ~\geq~   \frac{1}{(\varphi_1^*)''( \varphi_1'(z))} ~=~ \varphi_1''(z),
 \end{align}
 which shows that $ z \mapsto \varphi_1'(z) $ is  convex. 

To show that $\bar{H}_1$ admits a local maximum at $ 0 $,  we observe that, since $  \varphi_1'(0) =0 $, we have that $ \bar{H}_1'(0) = 0 $.
Moreover,   	Assumption \ref{Assumptions}   implies that $   (\varphi_1^*)''(0)   > 1/J $. 
Therefore, $   \varphi_1''(0)   < J $ and  $ \bar{H}_1''(0) < 0 $. 

It remains to show that there exists a unique point $ m^\star_1  \in(0,\infty)$ such that $ \bar{H}_1'(m^\star_1) = 0 $  and  $ \bar{H}_1''(m^\star_1) > 0 $.
Using again that $   \varphi_1''(0)   < J $,  we infer that  for $ z >0 $ small enough,
		\begin{align}
		\varphi_1'(z) ~=~ \int_0^z \varphi_1''(r) \, dr ~<~ J z.
		\end{align}
		Moreover, by part \emph{(ii)}, we know that there exists  $ m^\star_1 >z>0$ such that 
		\begin{align}\label{EqEnergyLandscape}
		\varphi_1'(m^\star_1 )~=~J m^\star_1  \quad \text{ and } \quad \varphi_1'(z) ~<~ J z \ \text{ for all }  z \in (0,m^\star_1)  .
		\end{align}
		However, the mean value theorem implies that there exists $ z' \in (0,m^\star_1)    $ such that $ \varphi_1''(z') > J $. Together with the fact that $ \varphi_1''$   is non-decreasing, this implies that 
		$ \varphi_1''(z) > J $ for all $  z \geq m^\star_1$. 
		This in turn yields that 
		\begin{align}\label{EqEnergyLandscape1}
		\varphi_1'(z ) ~>~ J z \ \text{ for all }  z > m^\star_1 \qquad \text{ and } \quad \bar{H}_1''(m^\star_1 ) ~>~ 0.
		\end{align}  
		Combining \eqref{EqEnergyLandscape} and \eqref{EqEnergyLandscape1}  shows that, at $ m^\star_1  $, there is the unique global minimum of 
		$\bar{H}_1$ on $ [0, \infty) $. 
\end{proof}

\subsection{Proof of Theorem \ref{ThmIntroEKF}} \label{SecEyring}
 In this section we provide the proof of  Theorem \ref{ThmIntroEKF}. 
 The bulk part of the proof  is a straightforward adaptation from the proof of Theorem \ref{ThmIntroEpsEKF}. 
 However,   the proof of the lower bound on the capacity is modified, since the (effective) single-site potential is not convex   in this chapter.
 The new proof of the lower bound is moved to Section \ref{SecLowerB}.
%

\begin{proofoft}{\emph{\ref{ThmIntroEKF}}.} 
Let $ \pm m^\star_1$ be the two global minimisers of the macroscopic Hamiltonian $ \bar{H}_1$.	
Let $ \eta_1 >0 $ and $ B^-_1, B^+_1\subset \R^N $ be defined by \eqref{EqEpsEta} and \eqref{EqEpsMetasets} with $ \eps=1 $.
As in  the proof of Theorem \ref{ThmIntroEpsEKF}, the starting point is the formula \eqref{EqMeanTransition}. 
Then, proceeding exactly as in the proofs of Proposition \ref{PropEpsUpperBound}  and Proposition \ref{PropEpsNum}, we can show that 
	\begin{align}
\label{EqRoughEstCap}
&\capa(B^-_1, B^+_1) ~\leq~  \frac{1}{ 2\pi } 
\ee^{-  N \bar{H} (0)}   \,  \sqrt{  |\bar{H}''(0)|  }\sqrt{  \varphi_1''(0)    }  
	~\left(1+O\left( \sqrt{\frac{\log(N)^3}{N}}\right)  \right), \ \text{ and}\\
	&\int_{(B^+_1)^c}f^*_{B^-_1,B^+_1}(y)\, \ee^{-\mathrm{H}(y)}\, dy ~=~
	\frac{	\ee^{-N\bar{H}_1(-m^\star_1 )}}{\sqrt{\bar{H}_1''(-m^\star_1 ) }} \,\sqrt{\varphi_1''(-m^\star_1) }\
	\left( 1+O\left(\sqrt{\frac{\log(N)^3}{N}}\right)\right),
\label{EqRoughEstNum}
\end{align}
which yields \eqref{EqRoughEstL}. Finally, \eqref{EqRoughEstU} follows from combining \eqref{EqRoughEstNum} with Proposition  \ref{PropNum}.
This concludes the proof of this theorem. 
\end{proofoft}

\subsection{Rough lower bound on the capacity}\label{SecLowerB}
In this section we prove the rough lower bound on the capacity.
We proceed as in the proof of Proposition \ref{PropEpsLowerBound}. 
Recall that the critical estimate in the proof of Proposition \ref{PropEpsLowerBound} is given by \eqref{RoughLBEqEps113}, where we apply the Poincar\'e inequality for the fluctuation measure  with a constant which is of order $ 1/\eps $ (see \eqref{EqEpsPI}).
By using the strict convexity of the (effective) single-site potential, \eqref{EqEpsPI} is  a consequence of the Bakry-\'Emery theorem.
Since the (effective) single-site potential is not assumed to be strictly convex in this chapter, the Bakry-\'Emery theorem is not applicable here. 
However,  instead, we can apply \cite[Theorem 1.6]{MeOt}, where it is shown that for all $ N \in \N $ and $ m \in \R $, $ \mu^{N,1}_m$ satisfies the Poincar\'e inequality with a constant $ \varrho >0 $, which is independent of  $ N   $ and   $ m $.
That is, for all $ N \in \N $ and $ m \in \R $ and for all $ f \in H^1(\mu_m^{N,1}) $,
\begin{align}\label{EqMuMPI}
\var_{\mu_m^{N,1}} \left( f \right) 
&~=~ \int  \left|	f - \int   	f   d\mu^{N,1}_m  	\right|^2 d\mu^{N,1}_m  
~\leq~ \frac{1}{\varrho}  \int  \left|	(\mathrm{id} - NP^tP) \nabla f	\right|^2 d\mu^{N,1}_m  ,
\end{align}
where $ P^tm= (1/N)  (m,\dots,m) \in \R^N $ for $ m\in \R $. 
This is the main ingredient of the proof of the following proposition. 
  
\begin{proposition}\label{PropNum}
	Consider the same setting as in Theorem \ref{ThmIntroEKF}.
	Let 
	\begin{align}
	a~=~    \frac{1}{\rho^2} \max_{m \in [-m^\star_1,m^\star_1] }   \int \left|  \Psi'' - \int   \Psi''  d\mu^{1,\varphi_1'(m)}  \right|^2 d\mu^{1,\varphi_1'(m)} ,
	\end{align}
	which is finite due to Assumption \ref{Assumptions}. Then, for $ N $ large enough,
	\begin{align}\label{RoughLBEq10}
	\capa(B^-_1, B^+_1) ~\geq~ \frac{1}{1+a} \,
	\frac{1}{ 2\pi }  
	\ee^{-  N \bar{H} (0)}   \,  \sqrt{  |\bar{H}''(0)|  }\sqrt{  \varphi_1''(0)    } 
	\left( 1 + O \left( \sqrt{\frac{ \log(N)^3}{N}} \right) \right).
	\end{align}
\end{proposition}
\begin{proof}
	Let $ f=f^*_{B^-_1 , B^+_1}(x) $. 
	We proceed exactly as in the proof of Proposition \ref{PropEpsLowerBound}, and obtain that for all  $ \theta \in [0,1] $, 
		\begin{align}\label{RoughLBEq111}
\begin{split}
			\int  \left|	 NP^tP \nabla f	\right|^2 \, d\mu^{N,1}  
&~\geq~ 
  (1-\theta )  ~
  \frac{\ee^{-  N \bar{H} (0)}}{ 2\pi\,Z_{\mu^{N,1}} }  
  \,  \sqrt{  |\bar{H}''(0)|  }\sqrt{  \varphi_1''(0)    } 
  \left( 1 + O \left( \sqrt{\frac{ \log(N)^3}{N}} \right) \right)
	\\ &\qquad	+~ \left(1 - \frac{1}{\theta }\right) N \int_{-m^\star_1+\eta_1}^{m^\star_1-\eta_1}   \left|  P\cov_{\mu_m^{N,1}} \left( f , \nabla  H\right) \right|^2 \, d\bar{\mu}^{N,1}(m).  
\end{split}
	\end{align}
Therefore,   choosing $ \theta = a/(1+a) $ it remains to show that 
	\begin{align} \label{RoughLBEq113}
	 \int_{-m^\star_1+\eta_1}^{m^\star_1-\eta_1}   \left|  P\cov_{\mu_m^{N,1}} \left( f , \nabla  H\right) \right|^2 \, d\bar{\mu}^{N,1}(m)
	\leq 
	\frac{a}{N} \int \left|(\mathrm{id} - NP^tP)\nabla f\right|^2 \, d\mu^{N,1}
		\left( 1 + O \left(\frac{ 1}{ \sqrt{N}} \right) \right).
	\end{align}
 
 In order to show \eqref{RoughLBEq113}, note that as in \eqref{EqEps0GradFBar}, 
	\begin{align}\label{Eq0GradFBar}
\begin{split}
&	\left|     P\cov_{\mu_m^{N,1}} \left( f , \nabla  H\right) \right|^2    
	~\leq~  \frac{1}{N^2}   \var_{\mu_m^{N,1}} \left( f \right)   \  \var_{\mu_m^{N,1}} \left(    \sum_{i=0}^{N-1}\Psi'(x_i) \right)  \\
	& \qquad ~\leq~ \frac{1}{\varrho^2N^2}\   \int  \left|	(\mathrm{id} - NP^tP) \nabla f	\right|^2 d\mu^{N,1}_m  \
 \int \left|(\mathrm{id} - NP^tP)\nabla \sum_{i=0}^{N-1}\Psi'(x_i)\right|^2 d\mu^{N,1}_m     .
\end{split}
	\end{align} 
	Then, we proceed analogously to \eqref{EqEpsGradH1} to observe that by the equivalence of ensembles (Proposition \ref{PropEquOfObserv}), 
%
	\begin{align}\label{EqGradH}
	\begin{split}
	\int &\left|(\mathrm{id} - NP^tP)\nabla \sum_{i=0}^{N-1}\Psi'(x_i)\right|^2 d\mu^{N,1}_m \\
	&
	~\leq~ N \max_{m \in [-m^\star_1 ,m^\star_1 ] }   \int \left|  \Psi'' - \int   \Psi''  d\mu^{1,\varphi_1'(m)}  \right|^2 d\mu^{1,\varphi_1'(m)}\ \left( 1 + O \left( \frac{ 1}{\sqrt{N}} \right) \right) .
	\end{split}
	\end{align}  
	This concludes the proof of \eqref{RoughLBEq113}.  
\end{proof}

It remains to show the equivalence of observables, which was used in \eqref{EqGradH}. This is done in the following proposition. 

\begin{proposition}[Equivalence of observables]\label{PropEquOfObserv}
	Let $ \ell \in \N $, and let  $ b:\R^\ell \ra [0,\infty) $ be such that 
\begin{align}
\sup_{m \in [-m^\star_1,m^\star_1]} \int_{\R^\ell  } |b (z_1,\dots, z_\ell  )|^2 \, d\mu^{1,\varphi_1'(m),\ell}(z_1,\dots, z_\ell  ) < \infty ,
\end{align} 
where  $\mu^{1,\varphi_1'(m),\ell}~=~\otimes_{i=1}^\ell \mu^{1,\varphi_1'(m)}  $.
Then there exists $ C_b \in (0,\infty)  $  such that for  $ N $ large enough,
\begin{align} 
\sup_{m \in [-m^\star_1,m^\star_1]}	~	&\left[   \int_{P^{-1}(m)}  b(x_1,\dots, x_\ell  ) \, d\mu_m^{N,1}  -\int_{\R^\ell  } b (z_1,\dots, z_\ell  ) \, d\mu^{1,\varphi_1'(m),\ell} \right|    ~\leq    C_b \frac{1}{\sqrt{N}} .
\end{align}
\end{proposition} 
\begin{proof}  
	Proceeding exactly as in the proof of Proposition \ref{PropEpsOfObserv}, we observe that the claim is proven once we show that
\begin{enumerate}[(i)]
	\item the local Cram\'er theorem holds true in this setting,  
	\item $ \sup_{m\in \R}~ \sum_{k=1}^{3} \int_{\R} \left| \frac{z-m}{s_{1}(m)}\right| ^k \, d\mu^{1, \varphi_1'(m) }(z) \, < \, \infty$, where 
	$ s_{1}(m) = \varphi_1''(m)^{\frac{1}{2}} $,	and
	\item there exists  $ c >0 $ such that $  \sup_{m\in \R}~  \left|\int_{\R}  \ee^{iz \xi} d\mu^{1, \varphi_1'(m) }(z)\right| \leq c |s_{1}(m) \xi|^{-1} $ for all $ \xi \in \R $. 
\end{enumerate}
Claim (i) is shown in Proposition \ref{PropLocalCramer}, and claim (ii) and (iii) are shown in \cite[Lemma 3.2]{MeOt}. 
	This concludes the proof of  this proposition.
\end{proof}



\appendix

\section*{Appendix}
\addcontentsline{toc}{section}{Appendix}
\renewcommand{\thesubsection}{A.\arabic{subsection}}
\numberwithin{equation}{subsection}

This appendix is organized as follows. 
In Section \ref{SecLegendre} we collect several properties of   Cram\'er transforms and the cumulant generating functions. 
In Section \ref{SecLaplace}  we derive   asymptotic expressions 
for certain integrals by using standard 
Laplace asymptotics. 
Then, in Section \ref{SecLowerB} we apply these results to estimate the moments and  the Fourier transforms of the measure $ \mu^{\eps,\varphi_\eps'(m)} $ (see \eqref{EqEpsGCMeasure}) for small $ \eps $.
Finally, in Section \ref{SecProofLocalCramer} and Section \ref{SecProofEquOfObserv} we state and prove \emph{the local Cram\'er theorem} and \emph{the equivalence of observables}, respectively. 

We note that the proofs in  Section \ref{SecProofLocalCramer} and Section \ref{SecProofEquOfObserv} remain true if  we replace the effective single-site potentials $ \psi_J $ by some general strictly convex function $ V $.  

\subsection{Properties of the Cram\'er transform}\label{SecLegendre}
\begin{lemmaa}\label{LemLegendre}
	Let $ W\in C^\infty(\R)  $ be  such that $
	\liminf_{|z|\ra\infty}\, W''(z) \,> \,0 .
	$
	Let 
	\begin{align}\label{EqLogKom}
	\chi^*(\sigma)  ~=~   \log \int_{\R} \ee^{ \sigma z-    W(z) }\,  dz, \quad 
	\text{ for } \sigma \in \R ,
	\end{align}
	and let $ \chi $ denote its  Legendre transform,
	i.e.
	\begin{align}\label{EqWCramer}
	\chi(m)= \sup_{\sigma \in \mathbb{R}} \left(\sigma m - \chi^*(\sigma) \right).
	\end{align}
	For all $ \sigma\in \R  $, define 
	$ \mu^\sigma \in \cP(\R) $ by 
				\begin{align} \label{EqGCMeasure}
 d\mu^\sigma(z)  ~=~ \ee^{-\chi^*(\sigma) +\sigma z-W(z)} \, dz~=~ 	\frac{ \ee^{\sigma z-W(z)} }{ \int_{\R} \   \ee^{\sigma \bar{z}-W(\bar{z})} \, d\bar{z}}\, dz  .
	\end{align}
	Then, the following statements hold true. 
	\begin{enumerate}[(i)]
		\item $ \chi^* $ and $ \chi $
		are strictly convex and smooth. If $ W $ is even, then  $ \chi^* $ and $ \chi $ are also even.
		\item For $ m \in \R $, we have that 
		\begin{align}\label{EqLegendre}
		\chi(m)~=~\chi'(m) m -\chi^*(\chi'(m))
		\quad
		\text{ and }
		\quad
		(\chi^*)'(\chi'(m))~=~m
		.
		\end{align}  
		In particular, 
				\begin{align}\label{EqLegendreDer}
		\chi''(m)~=~\frac{1}{	(\chi^*)''(\chi'(m))}
		\quad
		\text{ and }
		\quad
	\chi'''(m)~=~\frac{-(\chi^*)'''(\chi'(m))}{	(\chi^*)''(\chi'(m))^3}
		.
		\end{align}  
		\item For all $ \sigma\in \R  $, 
		\begin{align} \label{EqLegendreDer2}
		\begin{split} 
		(\chi^*)'\left( \sigma \right) &~=~  \frac{\int_{\R} z\, \ee^{\sigma z-W(z)} \, dz}{\int_{\R} \ee^{\sigma z-W(z)} \, dz}~=~  \int_{\R} z\,  d\mu^\sigma(z)  , \\
		(\chi^*)''\left( \sigma\right) &~=~   \int_{\R} \left( z-(\chi^*)'\left( \sigma \right)\right)^{2}\,  d\mu^\sigma(z) ,   \\
		(\chi^*)'''\left( \sigma\right) &~=~   \int_{\R} \left( z-(\chi^*)'\left( \sigma \right)\right)^{3}\,  d\mu^\sigma(z) ,     \\
		(\chi^*)^{(4)}\left( \sigma \right) ~+~3\, (\chi^*)''\left(\sigma \right)^2&~=~  	 \int_{\R} \left( z-(\chi^*)'\left( \sigma \right)\right)^{4}\,  d\mu^\sigma(z)    .
		\end{split}
		\end{align}
	\end{enumerate}	  
\end{lemmaa}

\begin{proof}
	These are standard results that follow from some elementary computations. We refer to  \cite[Lemma III.2.5]{PatThesis}
and \cite[Lemma 41]{GORV} for more details. 
\end{proof}

\subsection{Some asymptotic integrals}\label{SecLaplace}

The main result in this section is the following lemma, which is based on Laplace  asymptotics.
In the proof we  use the same strategy as in   \cite[Lemma A.3]{HerrTug}.
\begin{lemmaa}\label{LemAsymp}
Let $ \cK \subset \R   $ be a compact set.	
Let $ U \in C^{0,\infty}(\cK \times \R)  $, and for $ \Rm \in \cK $, let $ U_\Rm(z) = U(\Rm,z) $.
Suppose that there exist $ \alpha >0$ and  $ R >0 $ such that, for all  $ \Rm \in \cK $, $ U_{\Rm}  $ admits a unique global minimum at some point $ z_{\Rm} \in \R $ with $ U_{\Rm}''(z_{\Rm}) > R^{-1} $ and such that  $ U_{\Rm}(z) \geq \alpha z^2 $ for all $ z\in [-R,R]^c $.
Furthermore, we assume that the map $ \Rm \mapsto z_{\Rm} $ is  bounded on $ \cK $. 
	Then, for each $ k\in \N_0 $, $ \Rm \in \cK $ and $ \eps\in (0,1] $, 
	\begin{align}\label{EqMomentLaplace2k}
	\int_{\R} \left( z-z_{\Rm}\right)^{2k}\,  \ee^{-\ieps U_{\Rm}(z)} \, dz
	\,=\,    \ee^{-\ieps U_{\Rm}(z_{\Rm})}\, \frac{\sqrt{2\pi}\, (2k-1)!!\, \eps^{k+\frac{1}{2}}}{U_{\Rm}''(z_{\Rm})^{k+\frac{1}{2}}} \, \left( 1 \,+\, O_{\cK}\left(  \sqrt{\eps\, \log (\eps^{-1})^3}\right)  \right)  ,
	\end{align}	
	where for $ n\in \N $, $ n!! $ denotes the double factorial, and we make the convention that  $ (-1)!!:=1 $.
	Moreover,
		\begin{align}\label{EqMomentLaplace2k+1}
\begin{split}
	\int_{\R} \left( z-z_{\Rm}\right)^{2k+1}\,  \ee^{-\ieps U_{\Rm}(z)} \, dz
	&=\,  -\, \ee^{-\ieps U_{\Rm}(z_{\Rm})}
	\frac{ \sqrt{2\pi}(2k+3)!! \,U_{\Rm}'''(z_{\Rm}) \eps^{k+\frac{3}{2}}}{6U_{\Rm}''(z_{\Rm})^{k+\frac{5}{2}}}
	\left(1+ O_{\cK}\left(  \sqrt{\eps \log (\eps^{-1})^3}\right)\right) .
\end{split}
	\end{align}
\end{lemmaa}
\begin{proof}
	Fix $ \Rm \in \cK $. 
	In this proof, let $ C $ denote a varying positive constant, which is independent of $ \eps $ and $ \Rm $.
	
		\textbf{Step 1.} [Proof of \eqref{EqMomentLaplace2k}.]\\ 
	Let      $ \rho = \sqrt{2(k+1)\,\eps \, \log(\eps^{-1})}/\sqrt{U_{\Rm}''(z_{\Rm})}  $ and $ \bar{U}_{\Rm}(z) = U_{\Rm}(z+z_{\Rm})  $.
	Let $ \bar{R} \geq  R+ \sup_{\Rm \in \cK} ( |z_{\Rm}| + \sqrt{|U_{\Rm}(z_{\Rm})|/\alpha} )   $ be such that, for some $ \iota >0 $, $ y^{2k} \leq e^{\iota y^2 } $ for all $ y\in [-\bar{R},\bar{R}]^c $.
	Then,
	\begin{align}\label{EqMomentLaplace3}
	\begin{split}
	\int_{\R} &\left( z-z_{\Rm}  \right)^{2k}\,  \ee^{-\ieps U_{\Rm}(z)}\, dz
	~=~
	\int_{\R}  y^{2k}\,  \ee^{-\ieps \bar{U}_{\Rm}(y)}\, dy \\
	&~=~\int_{-\rho}^\rho  y^{2k}\,  \ee^{-\ieps \bar{U}_{\Rm}(y)}\, dy~+~
	\int_{B_{\bar{R}}(0)^c}  y^{2k}\,  \ee^{-\ieps \bar{U}_{\Rm}(y)}\, dy~+~
	\int_{B_{\bar{R}}(0)\setminus B_\rho(0)}  y^{2k}\,  \ee^{-\ieps \bar{U}_{\Rm}(y)}\, dy   \\
	&~=:~ I~+~II~+~III . 
	\end{split}
	\end{align} 
	In the following we show that $ I $ provides the main contribution and that $ II $ and $ III $
are negligible.
	
	\textbf{Step 1.1.} [Estimation of the term $ I $.]\\
	Note that by Taylor's formula, for some $ \theta \in [0,1] $, 
	\begin{align}
	\label{EqMomentLaplace41}
	\bar{U}_{\Rm}(y) ~=~ U_{\Rm}(z_{\Rm}) ~+~  \frac{1}{2}\,y^2\, U_{\Rm}''(z_{\Rm})
	~+~ \frac{1}{6}\,y^3\, \bar{U}_{\Rm}'''(\theta y).
	\end{align}
	By using that   $ \bar{U}_{\Rm}'''  $  is locally bounded (uniformly in $ \Rm \in \cK $), we see that there exists some $ c>0 $ such that $   |\bar{U}_{\Rm}'''(\theta y)| \,\leq\, c $ for all $ y \in [-\rho,\rho] $.
	Therefore, 
	\begin{align}
	\label{EqMomentLaplace4}
	\ee^{\frac{-c\rho^3}{6 \eps}}
	~\leq~  \frac{\int_{-\rho}^\rho  y^{2k}\,  \ee^{-\ieps \bar{U}_{\Rm}(y)}\, dy}{\ee^{-\ieps U_{\Rm}(z_{\Rm} )}\,\int_{-\rho}^\rho  y^{2k}\,  \ee^{-\ieps \frac{1}{2}\,y^2\, U_{\Rm}''(z_{\Rm})}\, dy}
	~\leq~ \ee^{\frac{c\rho^3}{6 \eps}}.
	\end{align}
	Thus, by using the definition of $ \rho $ and by some standard Gaussian computations applied to the denominator in \eqref{EqMomentLaplace4},  
	we infer that 
	\begin{align} 
	\label{EqMomentLaplace31}
	I &~=~ \sqrt{\frac{2\pi\eps}{U_{\Rm}''(z_{\Rm})}} ~ \ee^{-\ieps U_{\Rm}(z_{\Rm} )}\left( 
	\eps^{k}\,  \frac{ (2k-1)!!}{U_{\Rm}''(z_{\Rm})^{k}}  ~+~ O_{\cK}\left( \eps^{k+\frac{1}{2}} \sqrt{ \log (\eps^{-1})^3}\right)\right).
	\end{align}
	
	\textbf{Step 1.2.} [Estimation of the term $ II $.]\\
	We know that $ \bar{U}_{\Rm}(y) \geq \alpha y^2   $ and $ y^{2k} \leq e^{\iota y^2 } $ for all $ y\in [-\bar{R},\bar{R}]^c $.
	Hence, by \cite[Lemma 1.1]{BovBBM},
	\begin{align}
	II~\leq~
	2 
	\int_{\bar{R}}^\infty \ee^{- \left(\frac{\alpha}{\eps} - \iota \right) y^2}\, dy
	~\leq ~ 
	C\,    \ee^{  -  \frac{\alpha}{\eps} \bar{R}^2} .
	\end{align}
	Since $ \alpha \bar{R}^2 > |U_{\Rm}(z_{\Rm}) | $, this shows that 
	\begin{align}  
	\label{EqMomentLaplace32}
	II &~=~  \ee^{-\ieps U_{\Rm}(z_{\Rm} )} \, O_{\cK}\left( \eps^{k+1}\sqrt{ \log (\eps^{-1})^3}\right) .
	\end{align}

	\textbf{Step 1.3.} [Estimation of the term $ III $.]\\
	Since 
	$  \bar{U}_{\Rm}  $ has its unique minimum in $ 0 $,
	we have that for $ \eps  $ small enough, 
	$ \inf_{y \in B_{\bar{R}}(0)\setminus B_\rho(0)}\bar{U}_{\Rm}(y) = \bar{U}_{\Rm}(\rho)\wedge  \bar{U}_{\Rm}(-\rho)$.
	Without restriction, we suppose that $ \bar{U}_{\Rm}(\rho) \leq \bar{U}_{\Rm}(-\rho) $.
	Then, by using   \eqref{EqMomentLaplace41} and the arguments from Step 1.1,
	\begin{align}   
	\label{EqMomentLaplace303}
	|III |&~\leq~ 2 \bar{R}\,  \bar{R}^{2k}\, \ee^{-\ieps U_{\Rm}(z_{\Rm} +\rho )}~\leq~ C \,
	\ee^{-\ieps U_{\Rm}(z_{\Rm} )} \, 
	\ee^{-\ieps U_{\Rm}''(z_{\Rm} ) \frac{1}{2}\rho^2} .
	\end{align}
	Using the definition of $ \rho $, we have shown that 
	\begin{align}   
	\label{EqMomentLaplace33}
	III &~=~  \ee^{-\ieps U_{\Rm}(z_{\Rm} )} \, O_{\cK}\left( \eps^{k+1}\sqrt{ \log (\eps^{-1})^3}\right) .
	\end{align}  

	\textbf{Step 2.} [Proof of \eqref{EqMomentLaplace2k+1}.]\\
	To show \eqref{EqMomentLaplace2k+1} we proceed exactly as in Step 1 (with   $ \bar{\rho} = \sqrt{2(k+2)\,\eps \, \log(\eps^{-1})}/\sqrt{U_{\Rm}''(z_{\Rm})}  $ replacing $ \rho $) but with  the only difference
	that  here we estimate  the leading order  term $ I $ in the following way.
	The idea is based on   Step 2.3 of the proof of \cite[Lemma A.3]{HerrTug}. 
	First, by adding one more term in the Taylor expansion in  \eqref{EqMomentLaplace41}, we have that for some $ \theta \in [0,1] $, 
	\begin{align}
	\label{EqMomentLaplace411}
	\bar{U}_{\Rm}(y) ~=~ U_{\Rm}^0 ~+~  \frac{1}{2}\,y^2\, U_{\Rm}^2
	~+~ \frac{1}{6}\,y^3\, U_{\Rm}^3
	~+~ \frac{1}{24}\,y^4\, \bar{U}_{\Rm}^{(4)}(\theta y),
	\end{align}
	where for $ i=0,1,2,3 $, we abbreviate $ U_{\Rm}^i := U_{\Rm}^{(i)} (z_{\Rm}) $.
	Then,
	\begin{align}\nonumber
	&\ee^{\ieps U_{\Rm}(z_{\Rm} )}I ~=~	\ee^{\ieps U_{\Rm}^0}\, \int_{-\bar{\rho}}^{\bar{\rho}}  y^{2k+1} \,  \ee^{-\ieps \bar{U}_{\Rm}(y)}\, dy
	~=~	\int_{-\bar{\rho}}^{\bar{\rho}}  y^{2k+1} \,  
	\ee^{-\ieps (\frac{y^2}{2} U_{\Rm}^2
		+ \frac{y^3}{6}  U_{\Rm}^3
		+ \frac{y^4}{24}  \bar{U}_{\Rm}^{(4)}(\theta y))}
	\, dy\\
		\begin{split}\label{EqMomentLaplace111}
	&~=~ - \frac{1}{6\eps}  U_{\Rm}^3 \  \int_{-\bar{\rho}}^{\bar{\rho}} y^{2k+4}\,  \ee^{-\ieps \frac{1}{2}\,y^2\, U_{\Rm}^2}      \, dy
	 ~-~ \frac{1}{24\eps}\ 
 \int_{-\bar{\rho}}^{\bar{\rho}} y^{2k+5}\, \bar{U}_{\Rm}^{(4)}(\theta y)\, 	\ee^{-\ieps \frac{y^2}{2} U_{\Rm}^2 } \, dy \\
	&\quad
	+~  \int_{-\bar{\rho}}^{\bar{\rho}} y^{2k+1}\,  	\ee^{-\ieps \frac{y^2}{2} U_{\Rm}^2 }
	\left( 	\ee^{-\ieps (\frac{y^3}{6}  U_{\Rm}^3
		+ \frac{y^4}{24 }  \bar{U}_{\Rm}^{(4)}(\theta y))}-1 + \frac{y^3}{6\eps}  U_{\Rm}^3
	+ \frac{y^4}{24\eps}  \bar{U}_{\Rm}^{(4)}(\theta y)\right) \, dy  
	\end{split}\\
	&~=:~ I_1~+~ I_2 ~+~ I_3.\nonumber
	\end{align}
	We now show that the term $ I_1 $ provides the dominant contribution and that $ I_2  $ and $ I_3  $ are of lower order than $ I_1 $.
Concerning $ I_1 $, simple Gaussian computations as in Step 1.1 yield that 
\begin{align}
I_1 ~=~ 
- \frac{1}{6}  U_{\Rm}^3 \   \sqrt{\frac{2\pi\eps}{U_{\Rm}''(z_{\Rm})}} ~  \left( 
\eps^{k+1}\,  \frac{ (2k+3)!!}{U_{\Rm}''(z_{\Rm})^{k+2}}  ~+~ O_{\cK}\left( \eps^{k+1+\frac{1}{2}} \sqrt{ \log (\eps^{-1})^3}\right)\right).
\end{align}
For $ I_2 $ we use that $ \bar{U}_{\Rm}^{(4)} $ is locally bounded to obtain that 
\begin{align}
|I_2|
~\leq~ 
C \ieps \int_{-\bar{\rho}}^{\bar{\rho}} |y|^{2k+5}\,   	\ee^{-\ieps \frac{y^2}{2} U_{\Rm}^2 } \, dy
~\leq~ 
C \eps^{k+2}.
\end{align}
Finally, to estimate the term $ I_3 $,  note that 
  $ \bar{U}_{\Rm}^{'''} $ and $ \bar{U}_{\Rm}^{(4)} $ are locally bounded, and that 
$  \bar{\rho}^3/\eps\leq C\sqrt{\eps \log(\eps^{-1})} $.
Then, by using the inequality  $ |\ee^{-x} -1 +x | \leq |x|^2 \ee^{|x|} $, 
	\begin{align}
	\begin{split}
	|I_3| &
	~\leq  ~  
\int_{-\bar{\rho}}^{\bar{\rho}} |y|^{2k+1}\,  	\ee^{-\ieps \frac{y^2}{2} U_{\Rm}^2 } \, \ee^{|\ieps (\frac{y^3}{6}  U_{\Rm}^3
	+ \frac{y^4}{24 }  \bar{U}_{\Rm}^{(4)}(\theta y))|}
\left( 	  \frac{y^3}{6\eps}  U_{\Rm}^3
+ \frac{y^4}{24\eps}  \bar{U}_{\Rm}^{(4)}(\theta y)\right)^2 \, dy\\
&
~\leq  ~   \frac{C}{\eps^2}\, \ee^{  C\frac{\bar{\rho}^3}{\eps}  } \, 
\int_{-\bar{\rho}}^{\bar{\rho}} |y|^{2k+1}\,  	\ee^{-\ieps \frac{y^2}{2} U_{\Rm}^2 } \,
(|y|^6+|y|^8) \, dy
~\leq  ~  C \eps^{k+2}.
	\end{split}
	\end{align}
	This concludes the proof of \eqref{EqMomentLaplace2k+1}.	
\end{proof}

\begin{corollarya}\label{CorAsymp2}
	Consider the same setting as in Lemma \ref{LemAsymp}. Then, 
	\begin{align}\label{EqMomentLaplace2}
	\frac{ \int_{\R} \left( z-z_{\Rm}\right)^{2k}\,  \ee^{-\ieps U_{\Rm}(z)} \, dz}{\int_{\R}  \ee^{-\ieps U_{\Rm}(z)} \, dz} &~=~     \eps^{k}\,   \frac{ (2k-1)!!}{U_{\Rm}''(z_{\Rm})^{k}} ~+~ O_{\cK}\left( \eps^{k+\frac{1}{2}}\sqrt{ \log (\eps^{-1})^3}\right)   ,\qaq\\
	\label{EqMomentLaplace1}
	\frac{ \int_{\R}   (z-z_{\Rm} )^{2k+1}\,    \ee^{-\ieps U_{\Rm}(z)} \, dz}{\int_{\R}  \ee^{-\ieps U_{\Rm}(z)} \, dz} &~=~ 
	- \frac{  (2k+3)!! \,U_{\Rm}'''(z_{\Rm}) \eps^{k+1}}{6U_{\Rm}''(z_{\Rm})^{k+2}}  + O_{\cK}\left( \eps^{k+\frac{3}{2}} \sqrt{ \log (\eps^{-1})^3}\right) .
	\end{align}  
	Moreover,  
	\begin{align}\label{EqMomentLaplace24}
	\frac{ \int_{\R} \left( \bar{z}-\frac{\int_{\R} z\, \ee^{-\ieps U_{\Rm}(z)} \, dz}{\int_{\R} \ee^{-\ieps U_{\Rm}(z)} \, dz}\right)^{2k}\,  \ee^{-\ieps U_{\Rm}(\bar{z})} \, d\bar{z}}{\int_{\R}  \ee^{-\ieps U_{\Rm}(z)} \, dz} &=     \eps^{k}\,  \frac{ (2k-1)!!}{U_{\Rm}''(z_{\Rm})^{k}}  + O_{\cK}\left( \eps^{k+1}\sqrt{ \log (\eps^{-1})^3}\right)   ,
\\
\label{EqMomentLaplace13}
\begin{split}
	\frac{ \int_{\R} \left( \bar{z}-\frac{\int_{\R} z\, \ee^{-\ieps U_{\Rm}(z)} \, dz}{\int_{\R} \ee^{-\ieps U_{\Rm}(z)} \, dz}\right)^{2k+1}\,  \ee^{-\ieps U_{\Rm}(\bar{z})} \, d\bar{z}}{\int_{\R}  \ee^{-\ieps U_{\Rm}(z)} \, dz} &=    
	- \frac{ 2k (2k+1)!! \,U_{\Rm}'''(z_{\Rm}) \eps^{k+1}}{6U_{\Rm}''(z_{\Rm})^{k+2}}  + O_{\cK}\left( \eps^{k+\frac{3}{2}} \sqrt{ \log (\eps^{-1})^3}\right) .
\end{split}
	\end{align} 
\end{corollarya}
\begin{proof}
	To show \eqref{EqMomentLaplace2}, similarly as in  Step 5 in the proof of \cite[Lemma A.3]{HerrTug}, we apply \eqref{EqMomentLaplace2k} both to the numerator and to the denominator on the left-hand side of   \eqref{EqMomentLaplace2}. 
	Analogously,  we apply \eqref{EqMomentLaplace2k+1}   to the numerator and \eqref{EqMomentLaplace2k} to the denominator   to show \eqref{EqMomentLaplace1}.  
	
	To show  \eqref{EqMomentLaplace24}, we first introduce the measure $ d\nu(z) = \ee^{-\ieps U_{\Rm}(\bar{z})} / (\int_{\R}  \ee^{-\ieps U_{\Rm}(z)} \, dz)\, d\bar{z}  $. Then, the left-hand side of \eqref{EqMomentLaplace24} is equal to 
	\begin{align}\label{EqMomentLaplace21}
	\begin{split}
	& \int_{\R} \left(z-z_{\Rm}  \right)^{2k}\,  d\nu(z)
	 ~+~ 
	 \sum_{\ell=0}^{2k-1} \binom{2k}{\ell} \left(
	 \int_{\R}(z_{\Rm}-z) \,  d\nu(z) \right)^{2k-\ell}
	 \int_{\R} \left(z-z_{\Rm}  \right)^{\ell}\,  d\nu(z)  .
	\end{split}
	\end{align} 
	Using \eqref{EqMomentLaplace2} and  \eqref{EqMomentLaplace1}, it is easy to see that for each $ \ell = 0,\dots, 2k-1 $, 
		\begin{align}\label{EqMomentLaplace211}
	\begin{split} \left(
	\int_{\R}(z_{\Rm}-z) \,  d\nu(z) \right)^{2k-\ell}
	\int_{\R} \left(z-z_{\Rm}  \right)^{\ell}\,  d\nu(z) 
	~=~ O_{\cK}\left(\eps^{2k-\ell+\lceil\frac{\ell}{2} \rceil} \right)
	~\leq~ O_{\cK}\left(\eps^{k+1} \right).
	\end{split}
	\end{align} 
	Combining \eqref{EqMomentLaplace21}, \eqref{EqMomentLaplace211} and \eqref{EqMomentLaplace2} yields   \eqref{EqMomentLaplace24}.  
	
	It remains to show \eqref{EqMomentLaplace13}. 
	Similarly as in \eqref{EqMomentLaplace21}, we have that the left-hand side of  \eqref{EqMomentLaplace13} is equal to 
	\begin{align}\label{EqMomentLaplace131}
	& \int_{\R} \left(z-z_{\Rm}  \right)^{2k+1}\,  d\nu(z)
	~+~  (2k+1)
	\int_{\R}(z_{\Rm}-z) \,  d\nu(z)  
	\int_{\R} \left(z-z_{\Rm}  \right)^{2k}\,  d\nu(z) \\ \label{EqMomentLaplace132}
	&~+~
	\sum_{\ell=0}^{2k-1} \binom{2k+1}{\ell} \left(
	\int_{\R}(z_{\Rm}-z) \,  d\nu(z) \right)^{2k+1-\ell}
	\int_{\R} \left(z-z_{\Rm}  \right)^{\ell}\,  d\nu(z) .
	\end{align} 
	As above, we observe that all the summands in \eqref{EqMomentLaplace132} are of lower order. 
	Then, using \eqref{EqMomentLaplace2} and \eqref{EqMomentLaplace1} for the two terms in  \eqref{EqMomentLaplace131}, 
	we infer \eqref{EqMomentLaplace13}. 
\end{proof}

Finally in this section, we note that Lemma \ref{LemAsymp} and Corollary \ref{CorAsymp2} remain true if the function $ U $ is allowed to slightly depend on $ \eps $. This is the content of the following corollary.
\begin{corollarya}\label{CorAsymp3}
Let $ \cK \subset \R   $ be a compact set and let $ \eps_{\cK}>0 $.	
Let $ \tilde{V} \in C^{0,\infty}(\cK  \times \R)  $ and let $ \tau :\cK \times [-\eps_{\cK},\eps_{\cK}] \ra \R $ be  such that 
\begin{align}
	 \sup_{0<\eps<\eps_{\cK}} \sup_{m \in \cK  } |\tau(m, \eps)| <\infty  .
\end{align}
Let $ V:\cK \times [-\eps_{\cK},\eps_{\cK}] \times \ra \R $ be defined by  
\begin{align}
 V(m,\eps,z)  = \tilde{V}(m,z) + \tau(m, \eps) z  \qfa \  (m,\eps,z) \in \cK \times [-\eps_{\cK},\eps_{\cK}] \times \R.
 \end{align}
For each $ (m,\eps) \in \cK \times [-\eps_{\cK},\eps_{\cK}]   $, we write 
 $  V_{m,\eps}(z)  = V(m,\eps,z)   $ for all $ z\in\R $.
Suppose that there exists $ \alpha >0$ and  $ R >0 $ such that, for all  $ (m,\eps) \in \cK \times [-\eps_{\cK},\eps_{\cK}] $, $ V_{m,\eps}  $ admits a unique global minimum at some point $ z_{\Rm,\eps} \in \R $ with $ V_{\Rm,\eps}''(z_{\Rm,\eps}) > R^{-1} $ and such that  $ V_{\Rm,\eps}(z) \geq \alpha z^2 $ for all $ z\in [-R,R]^c $.
Furthermore, we assume that the map $ (\Rm,\eps) \mapsto z_{\Rm,\eps} $ is  bounded on $ \cK \times [-\eps_{\cK},\eps_{\cK}] $.
	Then, the statements \eqref{EqMomentLaplace2}--\eqref{EqMomentLaplace13} hold true if  $ U_m    $ is replaced by $ V_{m,\eps} $ and $ z_m $ by $ z_{\Rm,\eps} $.
\end{corollarya}
\begin{proof}
	We only have to repeat the proofs of Lemma \ref{LemAsymp} and Corollary \ref{CorAsymp2} with $ V_{m,\eps}  $ replacing $ U_{m} $. 
	Since all the arguments we used in these proofs also work for $ V_{\Rm,\eps} $, this concludes the proof.
\end{proof}

\subsection{A priori estimates for the measure $ \mu^{\eps , \varphi_\eps'(m)} $}

  For the proof of the local Cram\'er theorem and   the equivalence of observables we need    some estimates on certain moments and Fourier transforms of $ \mu^{\eps , \varphi_\eps'(m)} $ (see \eqref{EqEpsGCMeasure}).

\begin{lemmaa}  \label{LemMoments}
	Recall the definition of $ \psi_J $, $ \varphi^*_{\eps } $, $ \varphi_{\eps } $ and $ \mu^{\eps , \sigma} $ given in  \eqref{EqSingleSite}, \eqref{EqEpsMuBarStar}, \eqref{EqEpsMuBarCramer} and \eqref{EqEpsGCMeasure}.
Notice that the inverse $ (\psi_J' )^{-1} $ of $ \psi_J'  $ exists.
	\begin{enumerate}[(i)]
		\item 
		Let $ \tilde{K}\subset \R $ be compact. Then, for all $  \lambda   \in \tilde{K}  $,  $ (\varphi^*_{\eps })'\left( \ieps \lambda \right)~=~ (\psi_J' )^{-1} (\lambda) \, +\, \Omega_{\tilde{K} }(\eps)  $.
		
		\item For all compact intervals $ K \subset \R $    there exists $ \eps_K>0 $ and a   function 
		$ \tau :K \times [-\eps_K,\eps_K] \ra \R $  such that 
		$ \sup_{0<\eps<\eps_K} \sup_{m \in K  } |\tau(m, \eps)| <\infty  $
		and $ 	\varphi_{\eps }'(m) = \ieps \tau(m, \eps) $ for all $ m \in K $ and   $ 0<\varepsilon<\varepsilon_{K} $.
		\item For $ m \in \R $, let 
		\begin{align}\label{EqSEps}
			s_{\eps}(m)~=~[(\varphi_{\eps }^*)''(\varphi'_{\eps }(m))]^{\frac{1}{2}}.
		\end{align}
	Note that $ s_\eps $ is well-defined, since $  \varphi_\eps^* $ is strictly convex (see Lemma \ref{LemLegendre}).
		Then, for each compact interval $ K \subset \R $, there exist     $ C_K>0 $ and $\eps_K>0 $ such that for all $ m \in K $ and for all $ 0<\varepsilon<\varepsilon_{K} $,
		\begin{align}\label{EqMoments1}
 s_{\eps}(m)^2~=~  \Omega_{K}(\eps) 
		\quad \text{ and } \quad
		\sum_{k=1}^{4} \int_{\R} \left| \frac{ z-m}{s_{\eps}(m)}\right|^{k}\,  \, d\mu^{\eps,\varphi_{\eps }'(m)}(z)
		~\leq~ C_K.
		\end{align}
	\end{enumerate}
\end{lemmaa}

\begin{proof}
	\emph{(i).}		
	Note that for all $ \lambda \in \tilde{K} $, the function $ U( \lambda  ,z) = \psi_J (z) - \lambda z $ satisfies the same conditions as the function $ U $ from Corollary \ref{CorAsymp2}.
	In particular, $ U_{ \lambda  }  $ admits a unique global minimum at  $ (\psi_J' )^{-1} (\lambda)  $. Thus, part \emph{(i)} follows immediately from Lemma \ref{LemLegendre} and \eqref{EqMomentLaplace1}.  
	
	\emph{(ii).} Let $ K=[a,b] $ for some $ a,b\in \R $ with $ a<b $.
	Set $ F(m) =  	(\varphi^*_{\eps })'\left(  \psi_J'  (m )  /\eps \right)  $.
	From part \emph{(i)}, we know that for $ \eps $ small enough, 
	\begin{align}
	\begin{split}
	F\left( a-1\right)  &~=~  a-1 ~+~\Omega_{[a-1,b+1]}(\eps) ~<~ a, \qaq \\
	F\left( b+1\right) &~=~  b+1 ~+~\Omega_{[a-1,b+1] }(\eps) ~>~ b.
	\end{split}
	\end{align}
	Therefore, by the continuity of $ F $ and the mean value theorem, $ F([a-1,b+1]) \supset K$. We also know that $ F: [a-1,b+1] \ra F([a-1,b+1]) $ is bijective, since $ F $ is strictly increasing. 
	Setting now $ \tau(m, \eps) =  \psi_J'  (F^{-1}( m)) $ for $ m \in K $
	yields  that 
	\begin{align}
	(\varphi^*_{\eps })'\left( \ieps \tau(m, \eps) \right)= m \qquad \text{for all }m \in K .
	\end{align}
	Since $ \varphi_{\eps }'=((\varphi^*_{\eps })')^{-1} $ (cf. \eqref{EqLegendre}), this 
	concludes the proof of part \emph{(ii)}.
	
	\emph{(iii).}  
	Let $ V(m,\eps,z) = \psi_J (z) -   \tau(m, \eps) z $.
	Then, using part \emph{(ii)},
	 Lemma \ref{LemLegendre} and Corollary \ref{CorAsymp3}, we know that  for $ k=2,4 $ and for all $ m \in K  $,  
	\begin{align}\label{EqMoments2}
\begin{split}
\int_{\R} \left|   z-m \right|^{k}\,  d\mu^{\eps, \varphi_{\eps }'(m)}(z)&~=
\frac{ \int_{\R} \left( \bar{z}-\frac{\int_{\R} z\, \ee^{ -\ieps V_{m,\eps}(z)} \, dz}{\int_{\R}  \ee^{ -\ieps V_{m,\eps}(z)} \, dz}\right)^{k}\,  \ee^{-\ieps V_{m,\eps}(\bar{z})} \, d\bar{z}}{\int_{\R}  \ee^{ -\ieps V_{m,\eps}(z)} \, dz}\\
&~=\, \eps^{\frac{k}{2}}\,    \frac{ (k-1)!!}{\psi_J ''((\psi_J' )^{-1} (\tau(m, \eps)))^{\frac{k}{2}}} ~+~ O_{K}\left( \eps^{\frac{k+1}{2}}\sqrt{ \log (\eps^{-1})^3}\right).
\end{split}
\end{align}
Then, for $ k=2 $, the left-hand side of \eqref{EqMoments2} equals $ s_{\eps}(m)^2 $ (cf. Lemma \ref{LemLegendre}).
	Thus, \eqref{EqMoments2} proves the first claim in \eqref{EqMoments1}, since the map $ ( m,\eps) \mapsto \psi_J ''((\psi_J' )^{-1} (\tau(m, \eps))) $ is locally bounded. 
	Moreover, due to H\"older's inequality, to show the second claim in \eqref{EqMoments1}, it suffices to show that there exists $ \eps_K'>0 $ such that 
	\begin{align}\label{EqMoments}
 \sup_{0<\eps<\eps_K'}	\sup_{m\in K}~	\int_{\R} \left| \frac{ z-m}{s (m)}\right|^{4}\,  \, d\mu^{\eps,  \varphi_{\eps }'(m)}(z)
	~<~ \infty.
	\end{align}
	However, combining \eqref{EqMoments2} for $ k =4 $ and the first claim in \eqref{EqMoments1},  implies \eqref{EqMoments}. 
This conclude the proof of   part \emph{(iii)}.
\end{proof}

\begin{lemmaa}\label{LemFourierTrafo}
Consider the same setting as in Lemma \ref{LemMoments}.
Let $ K \subset \R $ be compact, and abbreviate $ \hat{z}(m) =  (z-m)/s (m)  $.
Then, there exists $ C_K, \eps_K >0 $ such that for all   $ \hat{\xi} \in \R \setminus \{0\}$,
 \begin{align}\label{EqLocalCramer6}
 \sup_{0<\eps<\eps_K}	\sup_{m\in K }~  \left| 
 \int_{\R}
  \ee^{i\hat{z}(m)\hat{\xi}}  \  d\mu^{\eps,  \varphi_{\eps }'(m)}(z)\right| ~\leq~ \frac{C_K}{|\hat{\xi}|}.
 \end{align}
 \end{lemmaa}
 \begin{proof}
 Fix $ m \in K $. 
  		In this proof $ C \in (0,\infty) $ denotes a  constant, which is independent of  $ \eps  $ and $ m  $,  and may change every time it appears.
  		
Let $ V_{m,\eps}(z) = \psi_J(z) - \tau(m, \eps)z $, where $ \tau(m, \eps) $ is introduced in Lemma \ref{LemMoments}.
 Then,   by partial integration (as in \cite[p.\ 37]{MeOt}) and by   \eqref{EqMoments1},
 \begin{align}\label{EqLocalCramer7}
 \begin{split}
 \left| 
  \int_{\R}
   \ee^{i\hat{z}(m)\hat{\xi}}  \  d\mu^{\eps,  \varphi_{\eps }'(m)}(z)\right|
    &~=~   
 \frac{s_\eps(m)}{|\hat{\xi}|}\, \ieps\,
 \left| 
 \frac{\int_{\R} \ee^{i\hat{z}(m)\hat{\xi}} \,   V_{m,\eps}'(z)  \, \ee^{ - \ieps V_{m,\eps}(z)}\, dz
 }{\int_{\R}   \ee^{ - \ieps V_{m,\eps}(z)}\, dz} \right| \\
 &~\leq ~  
 \frac{C}{|\hat{\xi}|\, \sqrt{\eps}}\,  
  \frac{\int_{\R} | V_{m,\eps}'(z)|  \, \ee^{ - \ieps V_{m,\eps}(z)}\, dz
  }{\int_{\R}   \ee^{ - \ieps V_{m,\eps}(z)}\, dz}.
 \end{split}
 \end{align}
 Let $ z_{m,\eps} $ be the unique global minimum of $ V_{m,\eps} $, and let $ \rho = C'\,\sqrt{\eps \, \log(\eps^{-1})}  $ for some $ C'  >0 $ large enough.
 Then, using the same arguments as in   the proof of Lemma \ref{LemAsymp},
 we see that the integral in the  numerator on the right-hand side of \eqref{EqLocalCramer7} is concentrated around 
 $ B_\rho ( z_{m,\eps}) $, i.e.
 \begin{align}\label{EqLocalCramer10}
 \begin{split}
 \int_{\R} | V_{m,\eps}'(z)|  \, \ee^{ - \ieps V_{m,\eps}(z)}\, dz& =    
 \int_{-\rho}^{\rho} | V_{m,\eps}'(z_{m,\eps}+z)|  \, \ee^{ - \ieps V_{m,\eps}(z_{m,\eps}+z)}\, dz + O_{K}(\eps^{2}\, \ee^{ - \ieps V_{m,\eps}(z_{m,\eps})})
 .
 \end{split}
 \end{align}
 Moreover, by Taylor's formula for some $ \theta,\theta'\in[0,1] $ (cf.  \eqref{EqMomentLaplace41}),
 \begin{align}\label{EqLocalCramer9}
 \begin{split}   
  \int_{-\rho}^{\rho} &| V_{m,\eps}'(z_{m,\eps}+z)|  \, \ee^{ - \ieps V_{m,\eps}(z_{m,\eps}+z)}\, dz  \\
 &~= ~  \ee^{ - \ieps V_{m,\eps}(z_{m,\eps})} \,\int_{-\rho}^{\rho} | z V_{m,\eps}''(z_{m,\eps}+\theta z)|  \, \ee^{ - \ieps V_{m,\eps}''(z_{m,\eps}) \frac{1}{2} z^2 -\ieps V_{m,\eps}'''(z_{m,\eps}+\theta'z) \frac{1}{6} z^3 }\, dz  
  \\
 &~\leq ~  C\,    \ee^{ - \ieps V_{m,\eps}(z_{m,\eps})} \,\int_{-\rho}^{\rho} | z|  \, \ee^{ - \ieps V_{m,\eps}''(z_{m,\eps}) \frac{1}{2} z^2  }\, dz 
  ~\leq ~ C \,  \ee^{ - \ieps V_{m,\eps}(z_{m,\eps})} \,\eps .
 \end{split}
 \end{align}
 Combining \eqref{EqLocalCramer7}, \eqref{EqLocalCramer10} and \eqref{EqLocalCramer9} and applying \eqref{EqMomentLaplace2k}  to the  denominator in the right-hand side of \eqref{EqLocalCramer7} yields   \eqref{EqLocalCramer6}.
 This concludes the proof.
 \end{proof}

\subsection{Proof of the local Cram\'er theorem}\label{SecProofLocalCramer}
In this section we prove the local Cram\'er theorem (Proposition \ref{PropEpsLocalCramer}). 
The main ideas of the proof are the same as in \cite[Proposition 31]{GORV} or \cite[Section 3]{MeOt}. 
The main difficulty here is to show that   the  estimates are uniform in $ \eps \ll 1$.

 \begin{proofofp}{\emph{\ref{PropEpsLocalCramer}}.} 
 Fix $ m \in K  $.
 		In this proof $ C \in (0,\infty) $ denotes a varying constant, which is independent of $ N   $, $ \eps  $ and $ m  $,  but may depend on $ K $.
 		
 	Let  $ s_{\eps}(m)   $ be defined by \eqref{EqSEps}.
 	In order   to simplify the presentation here, for any function $ f: \R \ra \R $ and for all $ z\in \R $, we abbreviate 
 	\begin{align}\label{EqAbrrev}
 	\langle f \rangle ~=~ \int_{\R} f(z) \, d\mu^{\eps, \varphi_\eps'(m) }(z) 
 	\qquad
 	\text{and}
 	\qquad
 	\hat{z}~=~ \frac{z-m}{s_{\eps}(m)}.
 	\end{align} 
 	
 	\textbf{Step 1.} [New representation of $ \ee^{ - N \varphi_{N,\eps}(m) } $.]\\
 	Let $ (X_i)_i $ be a sequence of random variables  that are independent and identically distributed with common law $ \mu^{\eps,  \varphi_\eps'(m)} $. Let 
 	\begin{align}
 	\tilde{S}_{\eps,m,N}~=~ \frac{1}{\sqrt{N}} \sum_{i=0}^{N-1} (X_i-m), 
 	\end{align}
 	and let $ \tilde{g}_{\eps,m,N} $ denote the Lebesgue density of the distribution of $ \tilde{S}_{\eps,m,N} $.
 	As in \cite[(31)]{MeOt}, using the coarea formula, we have that   
 	\begin{align}\label{EqLocalCramer1}
 	\tilde{g}_{\eps,m,N} (0) ~=~ \ee^{  N \varphi_{\eps}(m)  -N\varphi_{N,\eps}(m)     }.
 	\end{align}
 	Moreover, let $ g_{\eps,m,N} $ be the Lebesgue density of the distribution of 
 	\begin{align}
 	S_{\eps,m,N}= \frac{1}{\sqrt{N}} \sum_{i=0}^{N-1} \frac{X_i-m}{s_{\eps}(m)}.
 	\end{align} 
 	Then, by Lemma \ref{LemLegendre},
 	\begin{align}
 	g_{\eps,m,N}(0) ~=~ \tilde{g}_{\eps,m,N} (0) \, s_{\eps}(m) ~=~  \tilde{g}_{\eps,m,N} (0) \,  \varphi_{\eps}''(m)^{-\frac{1}{2}}.
 	\end{align}
 	Therefore, it suffices to show that for $ \eps $ small enough, 
 	\begin{align}\label{EqLocalCramer3}
 	\left| g_{\eps,m,N}(0) -  \frac{1}{\sqrt{2\pi}} \right| 
 	~=~ O_{K }\left(\frac{1}{\sqrt{N}}\right).
 	\end{align}
 	We show \eqref{EqLocalCramer3} by mimicking the arguments of the proof of \cite[Proposition 3.1]{MeOt}.	
 	Therefore, as in \cite[(44)]{MeOt}, we apply the inverse Fourier transform to obtain that 
 	\begin{align}\label{EqLocalCramer2}
 	2\pi \, g_{\eps,m,N}(0) ~=~ \int_{\R} 
 	\left\langle
 	\ee^{i\frac{1}{\sqrt{N}}\hat{z}\hat{\xi}}
 	\right\rangle^N \,  d \hat{\xi},
 	\end{align}
 	and we  split this integral  according to some $ \delta >0  $ (which is chosen in Step 2) as
 	\begin{align} 
 	\begin{split}
 	\int_{\R}      \left\langle   \ee^{i \frac{1}{\sqrt{N}} \hat{z}  \hat{\xi}}   \right\rangle^{N} d \hat{\xi} 
 	& ~ =~
 	\int_{\{|\frac{\hat{\xi}}{\sqrt{N}}|\leq \delta\}}  \left\langle   \ee^{i \frac{1}{\sqrt{N}} \hat{z}  \hat{\xi}}   \right\rangle^{N} d \hat{\xi}  ~ +~ 
 	\int_{\{|\frac{\hat{\xi}}{\sqrt{N}}|> \delta\}}      \left\langle   \ee^{i \frac{1}{\sqrt{N}} \hat{z}  \hat{\xi}}   \right\rangle^{N} d \hat{\xi}  \\
 	& ~ =:~
 	I ~+~ II .
 	\end{split}
 	\end{align}  
 	In the following we compute the asymptotic value of $ I $, and show that $ II $ is of lower order than $ I $.

 \textbf{Step 2.} [Estimation of the term $ I $.]\\
 From Lemma \ref{LemMoments} we know that there exists $ \eps_K>0  $ such that 
 \begin{align}\label{EqLocalCramer022}
\sup_{0<\eps<\eps_K}	\sup_{m\in K } ~ \sum_{k=1}^3 \langle |\hat{z}|^k \rangle ~\leq~ C .
 \end{align} 
 Then, as in \cite[(46)]{MeOt}, the estimate \eqref{EqLocalCramer022}  yields that 
 there exist $ \hat{\delta},c_K > 0   $   and a complex-valued function $ h $ such that 
 for all $ |\hat{\xi}|\leq \hat{\delta}  $ and all $ 0<\varepsilon<\varepsilon_{K} $,
 \begin{align}\label{EqLocalCramer02}
 \left\langle   \ee^{i   \hat{z}  \hat{\xi}}   \right\rangle ~=~ \ee^{-h(\hat{\xi})}
 \quad \text{and} \quad 
 \left| h(\hat{\xi}) - \frac{1}{2} \hat{\xi}^2\right| ~\leq~ c_K  |\hat{\xi}|^3  .
 \end{align} 
 Indeed, applying Taylor's formula yields that for all $ \xi \in \R $, 
 \begin{align}\label{EqLocalCramer03}
 \left\langle   \ee^{i   \hat{z}  \hat{\xi}}   \right\rangle ~=~
 1 - \frac{1}{2} \hat{\xi}^2  + \frac{1}{3}\hat{\xi}^3  \left\langle i  \hat{z}^3 \ee^{i  \theta_\xi \hat{z}  \hat{\xi}}   \right\rangle \qquad \text{for some $ \theta_\xi \in[0,1] $.}
 \end{align} 
Notice that $ \sup_{0<\eps<\eps_K,m\in K,\xi\in \R}\left|\left\langle i  \hat{z}^3 \ee^{i  \theta_\xi \hat{z}  \hat{\xi}}   \right\rangle \right| \leq C $ by   \eqref{EqLocalCramer022}.
Therefore, by applying Taylor's formula for the function $ \zeta \mapsto \log(1+\zeta) $ for $ \zeta\in \C $ with $ |\zeta| $ small enough, we obtain \eqref{EqLocalCramer02}.
 
 As a consequence of \eqref{EqLocalCramer02}, by choosing $ \delta < \hat{\delta} $, we have that
 \begin{align}
 I~=~
  	\int_{\{|\frac{\hat{\xi}}{\sqrt{N}}|\leq \delta\}}      \ee^{-Nh\left(\frac{\hat{\xi}}{\sqrt{N}}\right)} d \hat{\xi} .
 \end{align}
 Moreover, by arguing similarly as in   \cite[(69)]{MeOt},   \eqref{EqLocalCramer02} yields that
 for $ \delta < \hat{\delta} $ small enough,
 \begin{align}\label{EqLocalCramer21}
 \mathrm{Re}\left( N h\left(\frac{\hat{\xi}}{\sqrt{N}}\right)  \right) ~\geq~
 \frac{|\hat{\xi}|^2}{2}
 -c_K\,\delta\, |\hat{\xi}|^2
 ~\geq~
     \frac{|\hat{\xi}|^2}{4}  .
 \end{align} 
 This in turn implies that, by proceeding as in  \cite[p.\ 32]{MeOt},
 \begin{align}\label{EqLocalCramer22}
 \left| \ee^{-N h\left(\frac{\hat{\xi}}{\sqrt{N}}\right)} - \ee^{- \frac{1}{2} \hat{\xi}^2} \right|  
 ~\leq~ \ee^{- \frac{1}{4} \hat{\xi}^2}   c_K\, \frac{|\hat{\xi}|^3}{\sqrt{N}}  ,
 \end{align}
 which yields, as in \cite[p.\ 32]{MeOt},  to the estimate  
 \begin{align}
 \begin{split}
 \left| I -  \sqrt{2\pi} \right|   
 ~\leq~ 
   \frac{C}{\sqrt{N}}.
 \end{split}
 \end{align} 

 \textbf{Step 3.} [Estimation of the term $ II $.]\\
 It remains to show that the term $ II $ is negligible.
 Recall from Lemma \ref{LemMoments} and Lemma \ref{LemFourierTrafo} that there exist $ \eps_K',c_K' >0 $ such that for all $\hat{\xi}\in \R  $,
  \begin{align}\label{EqLocalCramer61}
   \sup_{0<\eps<\eps_K'}	\sup_{m\in K }~\langle |\hat{z}| \rangle ~\leq~ c_K'  \quad \text{ and }  \quad
   \sup_{0<\eps<\eps_K'}	\sup_{m\in K }~ \left|\left\langle
  \ee^{i\hat{z}\hat{\xi}}
  \right\rangle\right| ~\leq~ \frac{c_K' }{|\hat{\xi}|}  .
  \end{align}
 Then, following the proof of \cite[Lemma 3.4]{MeOt},  the estimates in \eqref{EqLocalCramer61}   (which are the analogues of \cite[(52)]{MeOt} and  \cite[(53)]{MeOt}) imply that for all  $ \delta < \hat{\delta}  $ there exists $ \lambda_{K,\delta} <1  $ (which depends only on $ c_K'  $ and $ \delta $) such that
  \begin{align}\label{EqLocalCramer5}
 \sup_{0<\eps<\eps_K'}	\sup_{m\in K }~ \left|\left\langle
  \ee^{i\hat{z}\hat{\xi}}
  \right\rangle\right| ~\leq~  \lambda_{K,\delta} \qquad \text{ for all } |\hat{\xi}| ~\geq ~ \delta .
  \end{align}
 Finally, applying the same arguments as in \cite[p.\ 32]{MeOt} shows that 
 \begin{align}
 |II|~\leq ~ C \, N \lambda_{K,\delta}^{N-2}.
 \end{align}
 Hence, $|II|~\leq ~ C/ \sqrt{N}$ for $ N $   large enough.
 This concludes the proof.
 \end{proofofp}

 As a simple consequence of the ideas from the proof of Proposition \ref{PropEpsLocalCramer}, we can state the  result in a more precise way in the trivial case that the (effective) single-site potential is a quadratic function.   The result is given in the following lemma. 
 \begin{lemmaa} \label{LemAppLocalCramerQuadr}
 Let	$ V (z)=  \frac{\alpha}{2} z^2 $ for some $ \alpha >0 $.
 	Let $ \chi^*_{\eps} $, $ \chi_{\eps} $, $ \mu^{\eps, \chi'_\eps(m)} $  be defined by \eqref{EqLogKom}, \eqref{EqWCramer} and \eqref{EqGCMeasure}, respectively, with $ W $  replaced by $ \ieps V $.
 	Let  $\chi_{N,\eps}: \mathbb{R} \to \mathbb{R}$ be defined by 
 	\begin{align}\label{EqEpsMuBarQ}
 	\varphi_{N,\eps} (m) ~=~ - \frac{1}{N} \log \int_{P^{-1}(m) } \ee^{ -  \ieps\sum_{i=0}^{N-1}V(x_i) }\,  d\mathcal{H}^{N-1} (x).
 	\end{align}
 	Then,  for all $ m\in \R $, 
 	\begin{align}\label{EqAppLocalCramerQuadr}
 	\begin{split}
 	\ee^{    -N\chi_{N,\eps} (m)   } \ = \ \ee^{   - N \chi_{\eps} (m)   } \,  \frac{\sqrt{\varphi_{\eps}''(m)}}{\sqrt{2\pi}} .
 	\end{split}	\end{align} 
 \end{lemmaa}
 
 \begin{proof}
Using the same notation and the same arguments as in Step 1 of the proof of  Proposition \ref{PropEpsLocalCramer}, we see that it suffices to show that 
 		\begin{align}\label{EqAppLocalCramerQuadr2}
 g_{\eps,m,N}(0) ~=~  \frac{1}{\sqrt{2\pi}} .
 	\end{align}
 	Note that by a simple computation, for all $ \sigma, m \in \R $,
 	\begin{align}\label{EqAppLocalCramerQuadr1}
 	\begin{split}
 	\chi_{\eps} (m) & ~=~  \frac{\alpha}{2\eps} m^2 ~-~ \frac{1}{2} \log \left(2\pi \frac{\eps}{\alpha}\right) \qaq
 	\mu^{\eps, \chi'_\eps(m)}(z)~=~  \ee^{-\frac{\alpha}{2\eps}(z-m)^2} \ \frac{1}{\sqrt{2\pi\frac{\eps}{\alpha}}}\ dz.
 	\end{split}
 	\end{align}
 	In particular,  $ \mu^{\eps, \chi'_\eps(m)}(z) $ is a Gaussian measure.
 	Therefore,  the claim \eqref{EqAppLocalCramerQuadr2} is a simple consequence of the stability of Gaussian measures under convolution. 
 %
 \end{proof}
 
 \subsection{Proof of the equivalence of observables}\label{SecProofEquOfObserv}
 
 In this section we prove the equivalence of observables, which is stated in Proposition \ref{PropEpsOfObserv}.
%
 The proof is similar to the proof of Proposition \ref{PropEpsLocalCramer} and combines the ideas from \cite{KwMe18a} and \cite{MeOt}.

 \begin{proofofp}{\emph{\ref{PropEpsOfObserv}}.}  
 	For simplicity, we only consider the case $ \ell =1 $. 
 	A straightforward modification of the following proof yields the claim also in the case $ \ell \in \N $.

 	Fix $ m \in K=[-2,2] $.
 	In this proof, let   $ C $   denote a varying positive constant, which does not depend on   $  N, \eps $ and  $ m$, but may depend on $ b $ and $ K $.
 	
 	 	  \textbf{Step 1.} [Cram\'er's representation.]\\
 		Proceeding as in \cite{KwMe18a}, we use the so-called \emph{Cram\'er representation} in order to rewrite the left-hand side of \eqref{EqEquOfObserv} in terms of the density of a  certain random variable.
 
 Let $\mu^{\eps,\varphi_{\eps}'(m),N}~=~\otimes_{i=1}^N \mu^{\eps,\varphi_{\eps}'(m)}  $, and let, for $ \sigma \in \R$,    the measure $ \mu^{ \sigma,\eps,\varphi_{\eps}'(m),N}  \in \cP(\R^N) $ be defined by 
 	\begin{align}
 		&\mu^{ \sigma,\eps,\varphi_{\eps}'(m),N} (dx)~=~\frac{1}{Z} \exp \left( \varphi_{\eps}'(m)\sum_{i=0}^{N-1}x_i + \sigma b(x_{0})- \ieps \sum_{i=0}^{N-1}\psi_J(x_i)  \right) dx  
 		,
 	\end{align}
 	where $ Z $  denotes the normalization constant.
 	Note that $ \mu^{ 0,\eps,\varphi_{\eps}'(m),N}~=~
 	\mu^{\eps,\varphi_{\eps}'(m),N} $.
 	Let $ (Y_i)_{i=1,\dots,N} $ be a  random vector distributed according to $ \mu^{ \sigma,\eps,\varphi_{\eps}'(m),N} $, 
 	and let 
 	\begin{align}
 		S_{\sigma,\eps,m, N}= \frac{1}{\sqrt{N}} \sum_{i=0}^{N-1} (Y_i-m). 
 	\end{align}
 Let $ \tilde{g}_{\sigma,\eps,m, N} $ denote the Lebesgue density of the distribution of $ S_{\sigma,\eps,m, N}$.
 Note that $ \tilde{g}_{0,\eps,m, N} = \tilde{g}_{\eps,m,N} $, where $ \tilde{g}_{\eps,m,N} $ is   defined in Step 1 of the proof of Proposition \ref{PropEpsLocalCramer}.
 	Using the same arguments as in  \cite[Lemma 5 and Lemma 6]{KwMe18a}, we observe that
 	\begin{align}\label{EqEpsuOfObservEqEps1}
 		\begin{split}
 			\left|    \int_{\R^N}b(x_{0}) \ d\mu^{\eps,\varphi_{\eps}'(m),N} -    \int_{P^{-1}(m)}  b(x_{0}) \ d\mu_m\right| 
 			\,=\,  	\left| \frac{	\smallrestr{\frac{d}{d\sigma}}{\sigma=0} \, \tilde{g}_{\sigma,\eps,m, N} (0)   }{\tilde{g}_{0,\eps,m, N} (0)}\right| .
 		\end{split}
 	\end{align}
 Hence, in order to show \eqref{EqEquOfObserv}, 	It suffices to show that there exist $ \eps_{b,K}>0 $ and $ N_{b,K} \in \N  $ such that for all $ N \geq N_{b,K}  $, $ 0<\varepsilon<\eps_{b,K} $ and $ m \in K $, 
 	\begin{align}\label{EqEpsOfObservEq}
\left|  	\smallrestr{\frac{d}{d\sigma}}{\sigma=0} \tilde{g}_{\sigma,\eps,m, N} (0)   \right| &~\leq ~ \frac{C}{s_\eps(m) \, \sqrt{N}}  \\ \label{EqEpsOfObservEq1}
	\left|  \tilde{g}_{0,\eps,m, N} (0)\right|  &~\geq ~ \frac{1}{s_\eps(m)\sqrt{2\pi}} ~ \left( 1
 			~+~ O_K\left(\frac{1}{\sqrt{N}}\right)\right)  , 
  	\end{align}
 	where	$ s_\eps (m)  $ is defined in \eqref{EqSEps}. 
 	
 		 	  \textbf{Step 2.} [Proof of \eqref{EqEpsOfObservEq1}.]\\
 	Using the same arguments as in  Step 1 of the proof of Proposition \ref{PropEpsLocalCramer}, we observe that
 	\begin{align}
 		\tilde{g}_{0,\eps,m, N} (0) ~=~\ee^{ N \varphi_{\eps}(m)  -N\varphi_{N,\eps} (m)    }.
 	\end{align}
 Then, Proposition \ref{PropEpsLocalCramer} yields  \eqref{EqEpsOfObservEq1}.
 
 		 	  \textbf{Step 3.} [Proof of \eqref{EqEpsOfObservEq}.]\\
Recall the abbreviations from \eqref{EqAbrrev}.
 	Let $ (X_i)_{i=1,\dots,N} $ be a  random vector distributed according to $ \mu^{ \eps,\varphi_{\eps}'(m),N} $, and let $ X $ be a  random variable distributed according to $ \mu^{ \eps,\varphi_{\eps}'(m)} $.
 	By    \cite[Lemma 7]{KwMe18a}, we have that
 	\begin{align}\label{EqEpsuOfObservEqEps3}
 		\begin{split}
 			&2\pi \smallrestr{\frac{d}{d\sigma}}{\sigma=0} \tilde{g}_{\sigma,\eps,m, N} (0)  = 
 			\int_{\R} \E_{\mu^{ \eps,\varphi_{\eps}'(m),N}}\left[  (b(X_{0}) - \langle b \rangle)   \ee^{i \frac{1}{\sqrt{N}} \sum_{i=0}^{N-1} (X_i-m)\xi}   \right] d \xi\\
 			&\quad~=~ 
 			\int_{\R} \E_{\mu^{ \eps,\varphi_{\eps}'(m)}}\left[  (b(X )- \langle b \rangle)  \ee^{i \frac{1}{\sqrt{N}}  (X -m)\xi}   \right]  
 			\E_{\mu^{ \eps,\varphi_{\eps}'(m)}}\left[    \ee^{i \frac{1}{\sqrt{N}} (X-m)\xi}   \right]^{N-1} d \xi\\
 			&\quad~=~ s_\eps(m)^{-1}
 			\int_{\R}  \left\langle (b- \langle b \rangle)   \ee^{i \frac{1}{\sqrt{N}}  \hat{z} \hat{\xi}}   \right\rangle    \left\langle   \ee^{i \frac{1}{\sqrt{N}} \hat{z}  \hat{\xi}}   \right\rangle^{N-1} d \hat{\xi}.
 		\end{split}
 	\end{align}
 	It remains to show that   for $ N $ large enough,
 	\begin{align}\label{EqEpsuOfObservEqEps5}
 		\left|  \int_{\R}  \left\langle (b- \langle b \rangle)   \ee^{i \frac{\hat{\xi}}{\sqrt{N}}   \hat{z}}   \right\rangle    \left\langle   \ee^{i \frac{\hat{\xi}}{\sqrt{N}}   \hat{z}}   \right\rangle^{N-1} d \hat{\xi} \right| ~\leq~ 
 		 \frac{C}{\sqrt{N}}    .
 	\end{align}   
 	
 	In order to show \eqref{EqEpsuOfObservEqEps5}, we proceed as in the proof of  ~\cite[Proposition 3.1]{MeOt} and Proposition \ref{PropEpsLocalCramer}.
 Let $ \hat{\delta} > 0   $ and  $ h $ be given as in Step 2 of the proof of Proposition \ref{PropEpsLocalCramer}.
%
 	We split the integral on the left-hand side in \eqref{EqEpsuOfObservEqEps5} according to some $ \delta < \hat{\delta}  $ (which is chosen in Step 3.1) as
 	\begin{align} 
 		\nonumber
 		&	\int_{\{|\frac{\hat{\xi}}{\sqrt{N}}|\leq \delta\}}  \left\langle (b- \langle b \rangle)   \ee^{i \frac{\hat{\xi}}{\sqrt{N}}   \hat{z}}   \right\rangle    \ee^{-(N-1)h\left(\frac{\hat{\xi}}{\sqrt{N}}\right)} d \hat{\xi}   +
 		\int_{\{|\frac{\hat{\xi}}{\sqrt{N}}|> \delta\}}  \left\langle (b- \langle b \rangle)   \ee^{i \frac{\hat{\xi}}{\sqrt{N}}   \hat{z}}   \right\rangle    \left\langle   \ee^{i \frac{\hat{\xi}}{\sqrt{N}}   \hat{z}}   \right\rangle^{N-1} d \hat{\xi}  \\
 		&~=:
 		I + II .
 	\end{align}  
 	We now show that $ |I|+|II|\leq C/ \sqrt{N} $.

 	\textbf{Step 3.1.} [Estimation of the term $ I $.]\\
	This step is very similar to Step   2 of the proof of Proposition \ref{PropEpsLocalCramer}.
 Using \eqref{EqLocalCramer02}, we have  that there exists $ c_K>0 $ such that
 	\begin{align}
 		\left| (N-1) h\left(\frac{\hat{\xi}}{\sqrt{N}}\right) - \frac{1}{2} \hat{\xi}^2 \right|  
 		~\leq~ c_K\,  \frac{|\hat{\xi}|^3}{\sqrt{N}}+     \frac{|\hat{\xi}|^2}{2N}  .
 	\end{align} 
 Similarly as in \eqref{EqLocalCramer21},  this inequality yields that for $ N \geq 4 $ and for $ \delta  $ small enough,
 \begin{align}\label{EqEpsEqOfObserv21}
\mathrm{Re}\left( (N-1) h\left(\frac{\hat{\xi}}{\sqrt{N}}\right)  \right) ~\geq~
\frac{|\hat{\xi}|^2}{2}
	-\left(c_K\,\delta+\frac{1}{8}\right)\, |\hat{\xi}|^2
~\geq~
\frac{|\hat{\xi}|^2}{4}  ,
\end{align} 
and hence, as in  \eqref{EqLocalCramer22}, 
 	\begin{align}
 		\left| \ee^{-(N-1) h\left(\frac{\hat{\xi}}{\sqrt{N}}\right)} - \ee^{- \frac{1}{2} \hat{\xi}^2} \right|  
 		~\leq~ \ee^{- \frac{1}{4} \hat{\xi}^2}\, 
 		 \left| c_K\,  \frac{|\hat{\xi}|^3}{\sqrt{N}}+     \frac{|\hat{\xi}|^2}{N}\right|  .
 	\end{align}
 	This implies that
 	\begin{align}
 		\begin{split}
 			&\left| I -  \int_{\{|\frac{\hat{\xi}}{\sqrt{N}}|\leq \delta\}}  \left\langle (b- \langle b \rangle)   \ee^{i \frac{\hat{\xi}}{\sqrt{N}}   \hat{z}}   \right\rangle    \ee^{- \frac{\hat{\xi}^2}{2}} d \hat{\xi} \right|  
 			\,\leq \,
 			C
 			\int_{\{|\frac{\hat{\xi}}{\sqrt{N}}|\leq \delta\}} 
 			\left| \frac{|\hat{\xi}|^3}{\sqrt{N}} +   \frac{|\hat{\xi}|^2}{N} \right| \ee^{- \frac{\hat{\xi}^2}{4}} d \hat{\xi}	
 			\leq 
 			\frac{C}{\sqrt{N}} ,
 		\end{split}
 	\end{align}
 	since, in view of \eqref{EqEquOfObserv0},
 	\begin{align}\label{EqEpsuOfObservEqEps4}
 		\left| \left\langle (b- \langle b \rangle)   \ee^{i \frac{\hat{\xi}}{\sqrt{N}}   \hat{z}}   \right\rangle \right| 
 		~\leq~ \sup_{m \in K}
 		\left\langle \left| b- \langle b \rangle \right|      \right\rangle  
 		~<~ \infty.
 	\end{align}
 	Moreover,  by  Taylor's formula, for some $ \theta \in [0,1] $, 
 	\begin{align}
 		\begin{split}
 			&\int_{\{|\frac{\hat{\xi}}{\sqrt{N}}|\leq \delta\}}  \left\langle (b- \langle b \rangle)   \ee^{i \frac{\hat{\xi}}{\sqrt{N}}   \hat{z}}   \right\rangle \,    \ee^{- \frac{|\hat{\xi}|^2}{2}}\,  d \hat{\xi}
 			\\
 			& \quad ~\leq~ 0 + 
 			\left|  \frac{1}{\sqrt{N}} \int_{\{|\frac{\hat{\xi}}{\sqrt{N}}|\leq \delta\}}  \left\langle   \hat{z} (b- \langle b \rangle)   \ee^{i \theta\frac{\hat{\xi}}{\sqrt{N}}   \hat{z}}   \right\rangle \,    i \hat{\xi} \ee^{- \frac{|\hat{\xi}|^2}{2}} \, d \hat{\xi} \right|  \\
 			& \quad ~\leq~  \frac{1}{\sqrt{N}}\ 
 			\left|  \left\langle   \hat{z} (b- \langle b \rangle)   \ee^{i \theta\frac{\hat{\xi}}{\sqrt{N}}   \hat{z}}   \right\rangle\right|  \
 			 \int_{\R}      |\hat{\xi}| \, \ee^{- \frac{|\hat{\xi}|^2}{2}} \, d \hat{\xi} 
 			~ \leq~  			
 		  \left\langle   |\hat{z}|^2   \right\rangle
 		  \left
 		  \langle  |b|^2    \right\rangle 
 			\frac{C}{\sqrt{N}}  .
 		\end{split}
 	\end{align}
 Using \eqref{EqEquOfObserv0} and that $ \sum_{k=1}^3 \langle |\hat{z}|^k \rangle \leq C $ implies that  $ 
 	\left|  I\right|  ~\leq ~ 
 	C/\sqrt{N}     $.
 	
 		\textbf{Step 3.2.} [Estimation of the term $ II $.]\\
 First note that by using \eqref{EqEpsuOfObservEqEps4}  it only remains to show that 
 	\begin{align}\label{EqAppEqEpsuOfObserv6}
 		\int_{\{|\frac{\hat{\xi}}{\sqrt{N}}|\geq \delta\}}       \left\langle   \ee^{i \frac{\hat{\xi}}{\sqrt{N}}   \hat{z}}   \right\rangle^{N-1} d \hat{\xi} ~\leq ~ \frac{C}{\sqrt{N}}.
 	\end{align} 
 Then, a straightforward adaptation of the arguments in Step 3 of the proof of Proposition \ref{PropEpsLocalCramer} yields the claim.
 \end{proofofp}
 

 
 \smallskip
 \noindent
 {\bf Acknowledgement.} 	
 The authors would like to give many thanks to Anton Bovier,
  Lorenzo Dello Schiavo, Dmitry Ioffe and Andr\'e Schlichting for many useful discussions and  suggestions.
  Moreover, the authors would like to thank the anonymous referees for numerous helpful comments and  for having read the paper with great care.    
 
\bibliography{TEX-Bib.bib}

\begin{thebibliography}{10}

\bibitem{ArmendarizGrosskinskyLoulakis}
I.~Armend{\'a}riz, S.~Grosskinsky, and M.~Loulakis.
\newblock Metastability in a condensing zero-range process in the thermodynamic
  limit.
\newblock {\em Probab. Theory Related Fields}, 169(1-2):105--175, 10 2017.

\bibitem{ArnrichMielkePeletierSavareVeneroni}
S.~Arnrich, A.~Mielke, M.~A. Peletier, G.~Savar\'{e}, and M.~Veneroni.
\newblock Passing to the limit in a {W}asserstein gradient flow: from diffusion
  to reaction.
\newblock {\em Calc. Var. Partial Differential Equations}, 44(3-4):419--454,
  2012.

\bibitem{Barret}
F.~Barret.
\newblock Sharp asymptotics of metastable transition times for one dimensional
  {SPDE}s.
\newblock {\em Ann. Inst. Henri Poincar\'{e} Probab. Stat.}, 51(1):129--166,
  2015.

\bibitem{BBovM}
F.~Barret, A.~Bovier, and S.~M\'{e}l\'{e}ard.
\newblock Uniform estimates for metastable transition times in a coupled
  bistable system.
\newblock {\em Electron. J. Probab.}, 15:no. 12, 323--345, 2010.

\bibitem{Ba2020}
K.~Bashiri.
\newblock On the long-time behaviour of mckean-vlasov paths.
\newblock {\em Electron. Commun. Probab.}, 25:no. 52, 14, 2020.

\bibitem{BBGradFlow}
K.~Bashiri and A.~Bovier.
\newblock Gradient flow approach to local mean-field spin systems.
\newblock {\em Stochastic Process. Appl.}, 130(3):1461--1514, 2020.

\bibitem{Berglund}
N.~Berglund.
\newblock Kramers' law: validity, derivations and generalisations.
\newblock {\em Markov Process. Related Fields}, 19(3):459--490, 2013.

\bibitem{BGesuN}
N.~Berglund, G.~Di~Ges\`u, and H.~Weber.
\newblock An {E}yring-{K}ramers law for the stochastic {A}llen-{C}ahn equation
  in dimension two.
\newblock {\em Electron. J. Probab.}, 22:Paper No. 41, 27, 2017.

\bibitem{BerglundFernandezGentz}
N.~Berglund, B.~Fernandez, and B.~Gentz.
\newblock Metastability in interacting nonlinear stochastic differential
  equations. {II}. {L}arge-{$N$} behaviour.
\newblock {\em Nonlinearity}, 20(11):2583--2614, 2007.

\bibitem{BerglundGentz}
N.~Berglund and B.~Gentz.
\newblock Sharp estimates for metastable lifetimes in parabolic {SPDE}s:
  {K}ramers' law and beyond.
\newblock {\em Electron. J. Probab.}, 18:no. 24, 58, 2013.

\bibitem{Algebra}
J.~Bewersdorff.
\newblock {\em Algebra f\"{u}r {E}insteiger}.
\newblock Friedr. Vieweg \& Sohn, Wiesbaden, second edition, 2004.
\newblock Von der Gleichungsaufl\"{o}sung zur Galois-Theorie.

\bibitem{BovBBM}
A.~Bovier.
\newblock {\em Gaussian processes on trees}, volume 163 of {\em Cambridge
  Studies in Advanced Mathematics}.
\newblock Cambridge University Press, Cambridge, 2017.
\newblock From spin glasses to branching Brownian motion.

\bibitem{BdH15}
A.~Bovier and F.~den Hollander.
\newblock {\em Metastability}, volume 351 of {\em Grundlehren der
  Mathematischen Wissenschaften [Fundamental Principles of Mathematical
  Sciences]}.
\newblock Springer, Cham, 2015.
\newblock A potential-theoretic approach.

\bibitem{BEGK01}
A.~Bovier, M.~Eckhoff, V.~Gayrard, and M.~Klein.
\newblock Metastability in stochastic dynamics of disordered mean-field models.
\newblock {\em Probab. Theory Related Fields}, 119(1):99--161, 2001.

\bibitem{BEGK02}
A.~Bovier, M.~Eckhoff, V.~Gayrard, and M.~Klein.
\newblock Metastability and low lying spectra in reversible {M}arkov chains.
\newblock {\em Comm. Math. Phys.}, 228(2):219--255, 2002.

\bibitem{BEGK04}
A.~Bovier, M.~Eckhoff, V.~Gayrard, and M.~Klein.
\newblock Metastability in reversible diffusion processes. {I}. {S}harp
  asymptotics for capacities and exit times.
\newblock {\em J. Eur. Math. Soc. (JEMS)}, 6(4):399--424, 2004.

\bibitem{CGOV84}
M.~Cassandro, A.~Galves, E.~Olivieri, and M.~E. Vares.
\newblock Metastable behavior of stochastic dynamics: a pathwise approach.
\newblock {\em J. Statist. Phys.}, 35(5-6):603--634, 1984.

\bibitem{dawgaetunneling}
D.~A. Dawson and J.~G\"{a}rtner.
\newblock Large deviations and tunnelling for particle systems with mean field
  interaction.
\newblock {\em C. R. Math. Rep. Acad. Sci. Canada}, 8(6):387--392, 1986.

\bibitem{dawgae}
D.~A. Dawson and J.~G{\"a}rtner.
\newblock Large deviations from the {M}c{K}ean-{V}lasov limit for weakly
  interacting diffusions.
\newblock {\em Stochastics}, 20(4):247--308, 1987.

\bibitem{dawgaequasi}
D.~A. Dawson and J.~G\"{a}rtner.
\newblock Large deviations, free energy functional and quasi-potential for a
  mean field model of interacting diffusions.
\newblock {\em Mem. Amer. Math. Soc.}, 78(398):iv+94, 1989.

\bibitem{schiavo2018}
L.~Dello~Schiavo.
\newblock {The Dirichlet-Ferguson Diffusion on the Space of Probability
  Measures over a Closed Riemannian Manifold}.
\newblock Preprint, arXiv:1811.11598, 2019.

\bibitem{demzei}
A.~Dembo and O.~Zeitouni.
\newblock {\em Large deviations techniques and applications}, volume~38 of {\em
  Stochastic Modelling and Applied Probability}.
\newblock Springer-Verlag, Berlin, 2010.
\newblock Corrected reprint of the second (1998) edition.

\bibitem{eckhoff2005precise}
M.~Eckhoff.
\newblock Precise asymptotics of small eigenvalues of reversible diffusions in
  the metastable regime.
\newblock {\em Ann. Probab.}, 33(1):244--299, 2005.

\bibitem{MR0395659}
R.~S. Ellis, J.~L. Monroe, and C.~M. Newman.
\newblock The {\rm {ghs}} and other correlation inequalities for a class of
  even ferromagnets.
\newblock {\em Comm. Math. Phys.}, 46(2):167--182, 1976.

\bibitem{EvGar}
L.~C. Evans and R.~F. Gariepy.
\newblock {\em Measure theory and fine properties of functions}.
\newblock Studies in Advanced Mathematics. CRC Press, Boca Raton, FL, 1992.

\bibitem{FengKurtz}
J.~Feng and T.~G. Kurtz.
\newblock {\em Large deviations for stochastic processes}, volume 131 of {\em
  Mathematical Surveys and Monographs}.
\newblock American Mathematical Society, Providence, RI, 2006.

\bibitem{gae}
J.~G\"{a}rtner.
\newblock On the {M}c{K}ean-{V}lasov limit for interacting diffusions.
\newblock {\em Math. Nachr.}, 137:197--248, 1988.

\bibitem{gaudilliereMilanesi}
A.~Gaudilli{\`e}re, P.~Milanesi, and M.~E.~E. Vares.
\newblock {Asymptotic exponential law for the transition time to equilibrium of
  the metastable kinetic Ising model with vanishing magnetic field}.
\newblock {\em J. Statist. Phys.}, 179(2):263--308, Apr. 2020.

\bibitem{GORV}
N.~Grunewald, F.~Otto, C.~Villani, and M.~G. Westdickenberg.
\newblock A two-scale approach to logarithmic {S}obolev inequalities and the
  hydrodynamic limit.
\newblock {\em Ann. Inst. Henri Poincar\'e Probab. Stat.}, 45(2):302--351,
  2009.

\bibitem{GvalaniSchlichting2020}
R.~S. Gvalani and A.~Schlichting.
\newblock Barriers of the mckean–vlasov energy via a mountain pass theorem in
  the space of probability measures.
\newblock {\em J. Funct. Anal.}, page 108720, 2020.

\bibitem{HerrTug}
S.~Herrmann and J.~Tugaut.
\newblock Non-uniqueness of stationary measures for self-stabilizing processes.
\newblock {\em Stochastic Process. Appl.}, 120(7):1215--1246, 2010.

\bibitem{herrmann2016mean}
S.~Herrmann and J.~Tugaut.
\newblock Mean-field limit versus small-noise limit for some interacting
  particle systems.
\newblock {\em Commun. Stoch. Anal.}, 10(1):39--55, 2016.

\bibitem{KwMe18a}
Y.~Kwon and G.~Menz.
\newblock Decay of correlations and uniqueness of the infinite-volume {G}ibbs
  measure of the canonical ensemble of 1d-lattice systems.
\newblock {\em J. Stat. Phys.}, 176(4):836--872, 2019.

\bibitem{Landim19}
C.~Landim.
\newblock Metastable {M}arkov chains.
\newblock {\em Probab. Surv.}, 16:143--227, 2019.

\bibitem{MeOt}
G.~Menz and F.~Otto.
\newblock Uniform logarithmic {S}obolev inequalities for conservative spin
  systems with super-quadratic single-site potential.
\newblock {\em Ann. Probab.}, 41(3B):2182--2224, 2013.

\bibitem{MS}
G.~Menz and A.~Schlichting.
\newblock Poincar\'e and logarithmic {S}obolev inequalities by decomposition of
  the energy landscape.
\newblock {\em Ann. Probab.}, 42(5):1809--1884, 2014.

\bibitem{PatThesis}
P.~E. M\"uller.
\newblock Limiting properties of a continuous local mean-field interacting spin
  system: Hydrodynamic limit, propagation of chaos, energy landscape and large
  deviations.
\newblock PhD thesis, Rheinische Friedrich-Wilhelms-Universit\"at Bonn, 2016.

\bibitem{OV04}
E.~Olivieri and M.~E. Vares.
\newblock {\em Large deviations and metastability}, volume 100 of {\em
  Encyclopedia of Mathematics and its Applications}.
\newblock Cambridge University Press, Cambridge, 2005.

\bibitem{Orrieri}
C.~Orrieri.
\newblock Large deviations for interacting particle systems: joint mean-field
  and small-noise limit.
\newblock {\em Electron. J. Probab.}, 25:44 pp., 2020.

\bibitem{Ott01}
F.~Otto.
\newblock The geometry of dissipative evolution equations: the porous medium
  equation.
\newblock {\em Comm. Partial Differential Equations}, 26(1-2):101--174, 2001.

\bibitem{SandierSer}
E.~Sandier and S.~Serfaty.
\newblock Gamma-convergence of gradient flows with applications to
  {G}inzburg-{L}andau.
\newblock {\em Comm. Pure Appl. Math.}, 57(12):1627--1672, 2004.

\bibitem{Seo}
I.~Seo.
\newblock Analysis of metastable behavior via solutions of poisson equations.
\newblock {\em Preprint.\ arXiv:1905.00743}, 2019.

\bibitem{RenesseSturm}
M.-K. von Renesse and K.-T. Sturm.
\newblock Entropic measure and {W}asserstein diffusion.
\newblock {\em Ann. Probab.}, 37(3):1114--1191, 2009.

\end{thebibliography}

\end{document}